\documentclass[11pt]{article}
\usepackage{fancyhdr}
\usepackage{amsthm}
\usepackage{amsmath}
\usepackage{amssymb}
\usepackage{amstext}
\usepackage{graphicx} 
\usepackage{psfrag}
\usepackage{cite}
\usepackage{pstricks}

 \definecolor{lightgray}{gray}{.75}
\usepackage{booktabs}

\usepackage{float}

\newtheorem{thm}{Theorem}[section]
\newtheorem{theorem}[thm]{Theorem}
\newtheorem{definition}[thm]{Definition}
\newtheorem{proposition}[thm]{Proposition}
\newtheorem{lemma}[thm]{Lemma}
\newtheorem{claim}[thm]{Claim}
\newtheorem{corollary}[thm]{Corollary}
\newtheorem{conjecture}[thm]{Conjecture}
\newtheorem{remark}[thm]{Remark}
\numberwithin{equation}{section}
\newcommand{\CP}{\mathbb{CP}}
\newcommand{\D}{\mathbb{D}}
\newcommand{\C}{\mathbb{C}}
\newcommand{\Z}{\mathbb{Z}}
\newcommand{\R}{\mathbb{R}}
\newcommand{\M}{\mathcal{M}}
\newcommand{\X}{\mathcal{X}}
\newcommand{\Lag}{\mathcal{L}}
\newcommand{\LB}{\mathfrak{L}}
\newcommand{\x}{\tilde{x}}
\newcommand{\y}{\tilde{y}}
\newcommand{\z}{\tilde{z}}
\newcommand{\w}{\tilde{w}}

\newcommand{\del}{\partial}
\newcommand{\ot}{O(t^{1/2})}

\begin{document}
\thispagestyle{empty}
\begin{LARGE}
\begin{center}
\textbf{On Exotic Lagrangian Tori in $\CP ^2$}
\end{center}

 \end{LARGE}

\begin{large}

\begin{center}
\textbf{ Renato Vianna} 
\end{center}
\end{large}

\begin{center}\abstract We construct an exotic monotone Lagrangian torus in
$\CP^2$ using techniques motivated by mirror symmetry. We show that it bounds 10
families of Maslov index 2 holomorphic discs, and it follows that this exotic
torus is not Hamiltonian isotopic to the known Clifford and Chekanov tori.
\end{center}

\tableofcontents

\begin{center}
\section{Introduction}
\end{center}

Using Darboux's theorem, it is very easy to find Lagrangian tori inside a
symplectic manifold, because any open subset of $\C^n$ contains many. Therefore it
has been of interest in symplectic topology to understand Lagrangian
submanifolds satisfying some global property, such as monotonicity (for
definition of monotone Lagrangian submanifold, see section \ref{MonTor}). On the
other hand, for a long time, the only known monotone Lagrangian tori in $\C^n$
(up to Hamiltonian isotopy) were the products $(S^1(r))^n \subset \C^n$, the so
called Clifford tori. Only in 1995, Chekanov introduced in his paper \cite{CHE}
the first examples of monotone Lagrangian tori not Hamiltonian isotopic to these. 

The Clifford torus can be symplectically embedded into the complex projective
space $\CP^n$ and the product of spheres $\times_n \CP^1$, giving monotone tori.
Each one of these is also known as a Clifford torus. Chekanov's monotone tori were
also known to give rise to exotic monotone Lagrangian tori in these spaces. But
it was only much later that Chekanov and Schlenk, in \cite{CHESCH}, described in
detail their family of exotic monotone Lagrangian tori in these spaces, where by exotic we mean
not Hamiltonian isotopic to the Clifford torus.

In \cite{DA07}, Auroux studied the SYZ mirror dual (a ``Landau-Ginzburg model")
of a singular special Lagrangian torus fibration given on the complement of an
anticanonical divisor in $\CP^2$. This fibration interpolates between the
Clifford torus and a slightly modified version of the Chekanov torus described
by Eliashberg and Polterovich in \cite{EP}. This construction explains how the
count of holomorphic Maslov index 2 discs, described by the superpotential of
the Landau-Ginzburg model, changes from the Clifford torus to the Chekanov
torus. The key phenomenon that arises is {\it 
wall-crossing}: in the presence of the singular fiber, some other
fibers bound Maslov index 0 discs. These fibers form a ``wall" on the base of the
fibration, separating the Clifford type torus fibers and the Chekanov type
torus fibers, and accounting for differences in the count of Maslov index 2
discs between the two sides of the wall.

In this paper, we reinterpret Auroux's construction using almost toric
fibrations as defined by Symington in \cite{MS}; see also \cite{NLMS}. The base of
the relevant almost toric fibration can be represented by a {\it base diagram}
that resembles the base of the moment map of the
standard torus action on $\CP^2$, except that it has a marked point called {\it
node} in the interior, representing the singular fiber, and a cut that encodes
the monodromy around the singular fiber; see Figure \ref{figIntro}, where 
nodes are represented by $\times$'s and cuts by dotted lines. Modifying
the almost toric fibration of a four dimensional symplectic manifold by
replacing a corner (zero dimensional fiber) by a singular fiber in the interior
with a cut is called {\it nodal trade}, also referred in this paper as
`smoothing the corner', and lengthening or shortening the cut is called {\it
nodal slide}. Both operations are known to preserve the four-manifold up to
symplectomorphism; see \cite{MS, NLMS}.

The Clifford torus lies over the center of the
standard moment map picture of $\CP^2$, and the small cut introduced by a nodal
trade points towards it. We can lengthen the cut to
pass through the Clifford torus, which develops a singularity and then becomes
the Chekanov torus. This is illustrated on the first three base diagrams of Figure
\ref{figIntro}. 

We can continue further and introduce another cut by performing a nodal trade in
one of the remaining corners and lengthening it to pass by the Chekanov torus,
giving rise to another monotone torus, as illustrated in Figure \ref{figIntro}.
This particular torus is the main focus of this paper. However, we also note
that we can further perform a nodal trade on the remaining corner and pass it
through the central fiber. Not only that, we can then shorten the other cuts to
pass again through the central fiber, giving rise to an infinite range of
monotone tori, that we conjecture not to be Hamiltonian isotopic to each other.

To perform these modifications in a more orderly way, it is convenient to redraw
the almost toric base, after crossing the central fiber. This is done by fully
cutting the almost toric base in two, following the considered cut, then
applying the monodromy associated with the cut to one of the two components and
gluing again with the other component. This move straightens the edges that
intersected the original cut, while creating a new cut in the same direction as
the original one but on the other side of the node. Each one of the pictures at the bottom of 
Figure \ref{figIntro} is related with the one right above it via this cut and glue process. 
Figure \ref{Amt} illustrates more 
the case with only one cut: after we switch the cut to the other side, we end up
with an almost toric fibration on $\CP^2$ with a base that resembles the
polytope of the weighted projective space $\CP(1,1,4)$, but with a cut and node
replacing the corner that corresponds to the orbifold singularity, and having
the Chekanov torus as its central fiber. Following the isotopies generated by
shortening the cut in the re-glued picture to a limit situation where the node
hits the corner illustrates a degeneration of $\CP^2$ into $\CP(1,1,4)$.

More generally, the projective plane admits degenerations to weighted projective spaces
$\CP(a^2,b^2,c^2)$, where $(a,b,c)$ is a Markov triple, i. e., satisfies the
Markov equation: 
$$a^2 + b^2 + c^2 =3abc.$$

All Markov triples are obtained from
(1,1,1) by a sequence of `mutations' of the form 
\begin{equation*}
  (a , b , c) \rightarrow (a , b , c' =3ab - c) 
  \end{equation*}
  
\begin{figure}[h!]
  \begin{center}
\scalebox{0.89}{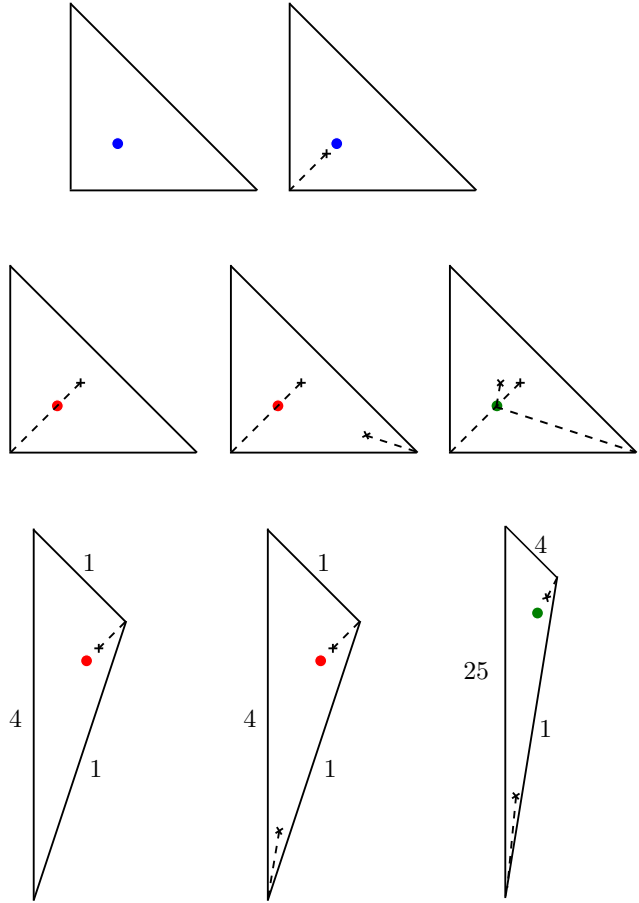}
 \caption{The procedure for going from the Clifford
 torus on the top left base diagram, to the Chekanov torus (third base 
 diagram) and to the $T(1,4,25)$ torus (fifth base diagram) by applying nodal
trades and nodal slides, where dots represent the image of the monotone tori in the base diagrams. 
Each of the bottom diagrams is equivalent to the one right
 above it since they are related by the cut and glue process described above and illustrated in Figure 
 \ref{Amt}. Affine lenghts of the edges are measured relative to the respective primitive vector.
  For detailed explanation of pictures see section \ref{atm}. }
\label{figIntro}
\end{center}
\end{figure}  
  
These degenerations of $\CP^2$ to other wighted projective spaces potentially
give an infinite range of exotic monotone Lagrangian tori in $\CP^2$, since they
are expected to bound different number of Maslov index 2 holomorphic discs. This was
conjectured by Galkin-Usnich in \cite{GalU}, where they also explain how to
predict the superpotential related to each one of the conjectured tori by
applying successive `mutations' to the superpotential \eqref{exWCLIF}.
  
A degeneration from $\CP^2$ to $\CP(a^2,b^2,c^2)$ can be illustrated by almost
toric pictures by introducing cuts in all corners of the standard polytope of
$\CP^2$ via nodal trades and then performing cut and glue operations as
described above, according to the sequence of mutations that links $(1,1,1)$ to
$(a, b, c)$. In view of this we call $B(a^2,b^2,c^2)$ the base of an almost toric
fibration on $\CP^2$ that is about to degenerate to the toric fibration of $\CP(a^2,b^2,c^2)$,
i.e., whose picture resembles a moment polytope of $\CP(a^2,b^2,c^2)$ but with
appropriate cuts, not passing through the center, joining each corner to a
node. We also call $T(a^2,b^2,c^2)$ the central fiber of $B(a^2,b^2,c^2)$, so
$T(1,1,1)$ is the Clifford torus and $T(1,1,4)$ is the Chekanov torus. Recalling
that walls of Maslov index 0 discs divide the base of a singular Lagrangian
fibration into chambers, we say that a fiber is of $T(a^2,b^2,c^2)$ type if it
belongs to a chamber that (continuously deforms to a chamber that) contains the
monotone $T(a^2,b^2,c^2)$ torus as a fiber, and hence bounds the same number
of regular Maslov index 2 $J$-holomorphic discs as $T(a^2,b^2,c^2)$.

The aim of this paper is to study $T(1,4,25)$. First we predict the number and
relative homotopy classes of regular Maslov index 2 $J$-holomorphic discs $T(1,4,25)$
bounds using wall-crossing formulas. Even though these formulas are believed to
hold for the almost toric case, they are not yet completely proven rigorously,
and neither is the relation between $J$-holomorphic discs and tropical curves
upon degeneration to a `large limit' almost complex structure. 

Therefore, after that we proceed to give, purely in the language of symplectic
topology, a complete self-contained proof of:\\
  
\begin{theorem}\label{mainthm}
  There exists a monotone Lagrangian torus in $\CP^2$ endowed with the standard Fubini-Study 
  form bounding 10 families of Maslov index 2 holomorphic discs, that is not 
  Hamiltonian isotopic to the Clifford and Chekanov tori.
\end{theorem}

For that we modify Auroux's example described in \cite{DA07}, by considering a
singular Lagrangian fibration that should interpolate between Chekanov type tori
and $T(1,4,25)$ type tori.    

More specifically, the rest of this paper is organized as follows.

In section \ref{MMS}, we review mirror symmetry in 
the complement of an anticanonical divisor, Landau Ginzburg models, 
wall-crossing phenomena and Auroux's example we mentioned above, following the 
approach in \cite{DA07, DA09}. 

In section \ref{atm}, we review almost toric fibrations and in section
\ref{Poten} we explain the relationship between $J$-holomorphic discs and
tropical discs in almost toric fibrations, working it out for the Example in
section \ref{Ex}. Even though the approach is not totally rigorous, in section
\ref{Pred} we use tropical discs and wall-crossing formulas for an almost toric
fibration to predict the existence of the $T(1,4,25)$ torus and the number of
Maslov index 2 discs it bounds, by computing the superpotential in an informal
manner.

In section \ref{ExoTor}, we use a explicit degeneration of $\CP^2$ into
$\CP(1,1,4)$ to define $T(1,4,25)$ type Lagrangian tori and set the conditions
for computing the Maslov index 2 holomorphic discs it bounds.

In section \ref{CHD}, we compute first the relative homotopy classes allowed to
have Maslov index 2 holomorphic discs and then the actual Maslov index 2
holomorphic discs a $T(1,4,25)$ type torus bounds. We also prove regularity
and orient the moduli space of holomorphic discs in each of the classes in order
to determine the correct signed count for the superpotential.

In section \ref{MonTor}, we modify the symplectic structure to construct the
monotone $T(1,4,25)$ torus and prove that it is not symplectomorphic to the
known Clifford and Chekanov tori. Finally, in section \ref{TorCP1xCP1}, we
repeat the techniques of sections \ref{Pred} and \ref{ExoTor} to conjecture the
existence of an exotic monotone torus in $\CP^1\times\CP^1$, bounding 9 families
of Maslov index 2 holomorphic discs.

\textbf{Acknowledgments.} I want to thank Denis Auroux for the invaluable
discussions and amazing support, Sergey Galkin for useful discussions, Xin Jin for 
pointing out a small technical mistake on the preprint and 
the referee for useful suggestions that improved the exposition. This work was
supported by The Capes Foundation, Ministry of Education of Brazil. Cx postal
365, Bras\'ilia DF 70359-970, Brazil; the CNPq - Conselho Nacional de
Desenvolvimento Cient\'ifico e Tecnol\'ogico, Ministry of Science, Technology
and Innovation, Brazil; The Fulbright Foundation, Institute of International
Education; the Department of Mathematics of University of California at
Berkeley; and the National Science Foundation grant number DMS-1007177.\\

\section{Motivation: Mirror symmetry}\label{MMS}

This section is a summary of the introduction to mirror symmetry in the
complement of a anti-canonical divisor explained in \cite{DA07, DA09}.
Mirror symmetry has been extended beyond the Calabi-Yau setting by considering
Landau-Ginzburg models. More precisely, it is conjectured that the Mirror of a
K\"{a}hler manifold $(X, \omega, J) $, with respect to a effective anticanonical
divisor $D$, is a Landau-Ginzburg model $(X^{\vee}, W)$, where $X^{\vee}$ is a
mirror of the almost Calabi-Yau $X \backslash D$ in the SYZ sense, i.e. a
(corrected and completed) moduli space of special Lagrangian tori in $X
\backslash D$ equipped with rank 1 unitary local systems ($U(1)$ flat
connections on the Lagrangian), and the superpotential $W: X^{\vee} \rightarrow
\mathbb{C}$ given by Fukaya-Oh-Ohta-Ono's $m_0$ obstruction to Floer homology,
which is a holomorphic function defined by a count of Maslov index 2 holomorphic
discs with boundary on the Lagrangian; see \cite{DA07, DA09}.
Kontsevich's homological mirror symmetry conjecture predicts that the Fukaya
category of $X$ is equivalent to the derived category of singularities of the
mirror Landau-Ginzburg model $(X^{\vee}, W)$.

In order to apply the SYZ construction to $X \backslash D$, we have to represent
it as a (special) Lagrangian fibration over some base. Also, to ensure that the
count of Maslov index 2 holomorphic discs is well defined, one asks $L$ to
satisfy some assumptions. More precisely, we require:

\begin{itemize}
\item[(1)] there are no non-constant holomorphic discs of Maslov index 0 in $(X,L)$;
\item[(2)] holomorphic discs of Maslov index 2 in $(X,L)$ are regular;
\item[(3)] there are no non-constant holomorphic spheres in $X$ with $c_1(TX)\cdot [S^2] \leq 0$.
\end{itemize}

In this case one defines the superpotential $ W = m_0 : X^{\vee}  \rightarrow \mathbb{C}$ by

\begin{definition}
\begin{equation} \label{m0}
m_0 (L , \nabla) = \sum_{\beta, \mu(\beta) = 2} n_{\beta}(L) \text{exp} ( -\int_{\beta} \omega) 
\text{hol}_{\nabla}(\del \beta)
\end{equation}
\end{definition}

where $\nabla$ is a $U(1)$ flat connection on $L$, $\text{hol}_{\nabla}(\del \beta)$
is the holonomy around the boundary of $\beta$ and $n_{\beta}(L)$ is the
(algebraic) count of holomorphic discs in the class $\beta $ whose boundary
passes though a generic point $p \in L$. More precisely, considering
$\mathcal{M}(L, \beta)$ the oriented (after a choice of spin structure for $L$)
moduli space of holomorphic discs with boundary in $L$ representing the class
$\beta$, $n_{\beta}(L)$ is the degree of its push forward under the evaluation
map at a boundary marked point as a multiple of fundamental class $[L]$, i.e.,
$ev_{\ast}[\mathcal{M}(L, \beta)] = n_{\beta}(L) [L]$.

 In principle one does not know if the series \eqref{m0} converge. Thus, it is
 preferable to replace the exponential by a formal parameter and the
 superpotential then takes values in the Novikov field. Nevertheless, all the
 superpotentials computed in this paper are given by a finite sums, and we use
 the exponential for consistency with \cite{DA07}.
 
 For each $\beta \in H_2(X, L, \Z)$, with $\del\beta \neq 0 \in H_1(L,\Z)$, we can define a
 holomorphic function $z_\beta : X^{\vee} \rightarrow \C^*$ by 
 \begin{equation}
 z_{\beta}(L,\nabla) = \text{exp}(-\int_\beta \omega) \text{hol}_\nabla (\del\beta); 
 \end{equation}
 see Lemma 2.7 in \cite{DA07}. 

\begin{remark} Actually, the function $z_\beta$ is only defined locally, for we have to keep track
of the relative class $\beta$ under deformations of $L$. In the presence of non-trivial monodromy,
which appears when we allow the fibration to have singular fibers, the function becomes multivalued.
\end{remark}

 In some cases, including the Lagrangian fibrations considered in this paper,
 the map $H_1(L) \rightarrow H_1(X)$ induced by inclusion is trivial, and then
 we can get a set of holomorphic coordinates $z_j = z_{\beta_j}$ by considering
 relative classes $\beta_j$ so that $\del\beta_j$ forms a basis of $H_1(L)$.
 Then our superpotential can be written as a Laurent series in terms of such
 holomorphic coordinates.

 In many cases we consider Lagrangian fibrations with singular fibers, and some
 of the Lagrangian fibers bound Maslov index 0 holomorphic discs, passing
 through the singular point. The projection of such Lagrangians forms ``walls"
 in the base, dividing it into chambers. The count of Maslov index 2
 holomorphic discs bounded by Lagragian fibers can vary for different chambers.
  This is called ``wall-crossing phenomenon"; see section \ref{WCross} and section 3
 of \cite{DA07}. Nevertheless, one can still construct the mirror by gluing the
 various chambers of the base using instanton corrections; see Proposition 3.9
 and Conjecture 3.10 in \cite{DA07}. 
 
The example below not only illustrates wall-crossing, but also
serves as the main model for the rest of the paper. For a more detailed
account, see section 5 of \cite{DA07} or section 3 of \cite{DA09}.

\subsection{A motivating example} \label{Ex} 

The following example is taken from \cite{DA07}, section 5. We will describe it 
in detail because our main construction, given in section \ref{ExoTor}, can be 
thought as a further development of the same ideas. 

Consider $\CP^2$, equipped with the standard Fubini-Study K\"{a}hler form,
and the anticanonical divisor $ D = \{ (x : y : z); (xy - cz^2)z = 0 \} $, for
some $c \neq 0$. We will construct a family of Lagrangian tori in the complement
of the divisor $D$. For this we look at the pencil of conics defined by the
rational map $f : (x : y : z) \mapsto (xy : z^2)$. We will mostly work with $f$
in the affine coordinate given by $z = 1$, as a map from $\mathbb{C}^2$ to
$\mathbb{C}$, $f(x,y) = xy$. The fiber of $f$ over any non-zero complex number
is then a smooth conic, while the fiber over 0 is the union of two lines, and
the fiber over $\infty$ is a double line.

There is a $S^1$ action on each fiber of $f$ given by $(x,y) \mapsto
(e^{i\theta}x, e^{-i\theta}y)$. Recall that the symplectic fibration $f$ carries
a natural connection induced by the symplectic form, whose horizontal
distribution is the symplectic orthogonal to the fiber. Our family of tori will
consist then of parallel transports of each $S^1$ orbit, along circles in the
base of the fibration, centered at $c \in \mathbb{C}$. We say that the height of
an $S^1$ orbit is the value of $\mu(x,y) = \frac{1}{2}\frac{|x|^2 - |y|^2}{1 +
|x|^2 + |y|^2}$, which is the negative of the moment map of the $S^1$ action.
Let $V_\theta$ be the vector field generated by the $S^1$ action. Since $d\mu = 
- \omega(V_\theta, \cdot)$ and $V_\theta$ is contained in the tangent space of the
fibers, we see that the moment map remains invariant under parallel transport.
Therefore we get that our family of Lagrangian tori is given by

\begin{definition}
Given $r >0$, and a real number $\lambda \in \mathbb{R}$, set
\begin{eqnarray}
T^c_{r,\lambda}   & = &  \left\{ (x:y:z) ;\left|f(x:y:z) - c \right| = r ; \mu(x:y:z) = \lambda \right\}   
\nonumber \\     & = &  \left\{ (x,y) ;\left|xy - c \right| = r ; |x|^2 - |y|^2 = 2\lambda(1 + |x|^2 + |y|^2) \right\}   
\end{eqnarray}
\end{definition}

\begin{figure}[h]
\setlength{\unitlength}{6.5mm}
\begin{center}
\begin{picture}(13.5,5)(0,0.2)
\psset{unit=\unitlength}
\newgray{ltgray}{0.85}
\psellipse[linewidth=0.5pt,fillstyle=solid,fillcolor=ltgray](8.6,4.5)(2.98,0.5)
\psellipse[linewidth=0.5pt,fillstyle=solid,fillcolor=white](8.4,4.5)(2.02,0.2)
\put(0,0){\line(1,0){11.5}}
\put(0,0){\line(1,1){2}}
\put(2,2){\line(1,0){11.5}}
\put(11.5,0){\line(1,1){2}}
\put(3.5,1){\makebox(0,0)[cc]{\tiny$\times$}}
\put(3.75,0.9){\small $0$}
\put(8.5,1){\makebox(0,0)[cc]{\tiny$\times$}}
\put(8.75,0.9){\small $c$}
\put(13.5,1){\makebox(0,0)[cc]{\tiny$\times$}}
\put(13.75,0.9){\small $\infty$}
\psline{<->}(6.2,1)(8.2,1)
\put(7.1,1.1){\tiny $r$}
\psellipse(8.5,1)(2.5,0.5)
\psline[linearc=0](2.7,3)(3.5,4)(2.7,5)
\psline[linearc=0](4.3,3)(3.5,4)(4.3,5)
\psline[linearc=0.7](5.2,3)(6,4)(5.2,5)
\psline[linearc=0.7](6.8,3)(6,4)(6.8,5)
\psline[linearc=1.7](10.2,3)(11,4)(10.2,5)
\psline[linearc=1.7](11.8,3)(11,4)(11.8,5)
\put(3.5,4){\circle*{0.15}}
\psellipse[linestyle=dotted,dotsep=0.6pt](6,4)(0.2,0.07)
\psellipse[linestyle=dotted,dotsep=0.6pt](11,4)(0.5,0.12)
\psellipse(11,4.5)(0.58,0.14)
\psellipse(6,4.5)(0.38,0.1)
\psline{<->}(11.7,3.95)(11.7,4.55)
\put(11.9,4.15){\small $\lambda$}
\psline(13.45,5)(13.45,3)
\psline(13.55,5)(13.55,3)
\psline{->}(3.5,2.8)(3.5,1.3)
\psline{->}(6,2.8)(6,1.3)
\psline{->}(11,2.8)(11,1.3)
\psline{->}(13.5,2.8)(13.5,1.3)
\put(0.5,1.3){\makebox(0,0)[cc]{\small $\C$}}
\put(0.5,4.1){\makebox(0,0)[cc]{\small $\C^2$\!\!}}
\put(0.65,2.7){\small $f$}
\psline{->}(0.5,3.6)(0.5,1.8)
\put(8.5,3.5){\small $T^c_{r,\lambda}$}
\end{picture}
\end{center}
\caption{The special Lagrangian torus $T^{c}_{r,\lambda}$ in $\C^2\setminus D$ (from \cite{DA07})}
\end{figure}
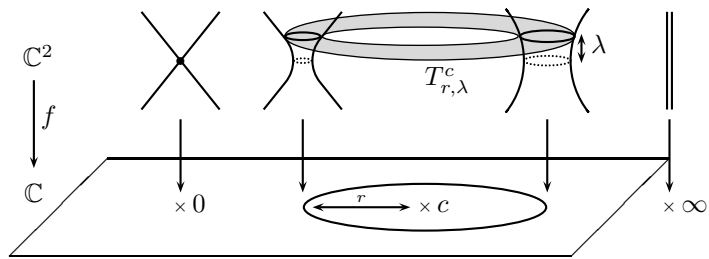

\begin{remark}
  All the pairs consisting of a symplectic fibration together with a map from the symplectic
  manifold to $\R$ (real data) used to define the Lagrangian fibrations considered in this paper
  form pseudotoric structures as defined by Tyurin in \cite{Ty}. 
\end{remark}

Note that actually $T^c_{|c|, 0}$ is a singular torus, pinched at $(0,0)$, so
varying $r$ and $\lambda$ give us a singular toric fibration. If $r > |c|$, we
say that $T^c_{r,\lambda} $ is of Clifford type, and if $r < |c|$ , of Chekanov
type. The motivation for this terminology is that in the first case we can
deform the circle centered at $c$ with radius $r$ in the base to a circle
centered at the origin, without crossing it, and with it we obtain a Lagrangian
isotopy from $T^c_{r,0}$ to a Clifford torus $S^1(\sqrt{r})\times
S^1(\sqrt{r})$. Not crossing the origin implies that no torus in the deformation
bounds Maslov index 0 discs, hence the count of Maslov index 2 discs remains the
same; see section 5.2 of \cite{DA07}. On the other hand, for $r < |c|$,
$T^c_{r,0}$ is the Eliashberg-Polterovich version of the so-called Chekanov torus;
see \cite{EP}.

To compute the Maslov index of discs in terms of their algebraic intersection number 
with the divisor $D$, one can prove that these Lagrangian tori are special with
respect to the holomorphic 2-form $\Omega(x,y) = (xy - c)^{-1}dx\wedge dy$. In
general, we can associate to an anticanonical divisor $D$ a nonvanishing holomorphic
n-form $\Omega$ on the complement $X \setminus D$ given by the inverse of a section of
the anticanonical bundle that defines $D$. Recall the following definition: 

\begin{definition} \label{SLag} 
A Lagrangian submanifold L is said to be 
special Lagrangian, with respect to $\Omega$ and with phase $\phi$, 
if $Im(e^{-i\phi}\Omega)_{|L} = 0$.
\end{definition}

For a proof that $T^c_{r,\lambda} $ are special Lagrangian with respect to
$\Omega$, see proposition 5.2 of \cite{DA07}. The following is Lemma 3.1 of \cite{DA07}.

\begin{lemma}\label{MI} 
If $ L \subset X \backslash D$ is special Lagrangian,
then for any relative homotopy class $\beta \in \pi_2(X,L)$ the Maslov index of
$\beta$, $\mu(\beta)$, is equal to twice the algebraic intersection number
$\beta \cdot [D]$. 
\end{lemma}

It can also be shown that $T^c_{r,\lambda} $ bounds Maslov index 0 holomorphic discs if and
only if $ r = |c|$. So we see that $r = |c|$ creates a wall in the base of our
Lagrangian fibration given by pairs $(r , \lambda)$. Then we need to treat the
cases $r > |c|$ and $r < |c|$ separately. 

For $ r > |c|$, we argue that $T^c_{r, \lambda}$ is Lagrangian isotopic to a
product torus $S^1(r_1)\times S^1(r_2)$, without altering the disc count
throughout the deformation. Denote by $z_1$ and $z_2$ respectively the
holomorphic coordinates on the mirror associated to the relative homotopy
classes $\beta_1$ and $\beta_2$ of discs parallel to the $x$ and $y$ coordinate
axes in $(\mathbb{C}^2,S^1(r_1)\times S^1(r_2) )$. Namely, $z_i = \text{exp}
(-\int_{\beta_i} \omega) \text{hol}_{\nabla}(\del \beta_i)$. We get from
Proposition 4.3 of \cite{DA07} that the superpotential recording the counts of
Maslov index 2 holomorphic discs bounded by $T^c_{r, \lambda}$ for $r > |c|$
is given by

\begin{equation}\label{exWCLIF}
W = z_1 + z_2 + \frac{e^{-\Lambda}}{z_1z_2},
\end{equation}

where $\Lambda = \int_{[\CP^1]} \omega$. The term $\frac{e^{-\Lambda}}{z_1z_2}$
corresponds to discs that project via $f$ to a double cover of $\C^2 \setminus
\Delta$ branched at infinity lying in the class $[\CP^1] - \beta_1 - \beta_2 \in
\pi_2( \CP^2,T^c_{r,0})$. The other terms $z_1$ and $z_2$ of the superpotential
correspond to sections of $f$ over the disc $\Delta$ centered at $c$ with radius
$r$, intersecting respectively the components $\{ x=0 \}$ and $\{ y =0 \}$ of
the fiber $f^{-1}(0)$.

Now we look at the case $r < |c|$, and consider the special case $\lambda = 0$,
the Chekanov torus considered by Eliashberg-Polterovich in
\cite{EP}. One family of Maslov index 2 holomorphic discs lies over the disc $\Delta$ centered at $c$
with radius $r$, given by the intersection of $f^{-1}(\Delta)$ with the lines $x =
e^{i\theta}y$. We denote by $\beta$ their relative class in $\pi_2(\CP^2,
T^c_{r,0})$. The other discs are harder to construct. Consider the
class $\alpha$ of the Lefschetz thimble associated with the critical point of
$f$ at the origin and the vanishing path $[0, c - re^{\text{arg}(c)i}]$. One can see that
$\alpha$, $\beta$ and $H =[\CP^1]$ form a basis of $\pi_2(\CP^2, T^c_{r,0})$.
The following Lemma and Proposition, due to Chekanov-Schlenk \cite{CHESCH}, have
their proofs sketched in \cite{DA07}. 

\begin{lemma}[Chekanov-Schlenk \cite{CHESCH}] 
The only classes in $\pi_2(\CP^2,T^c_{r,0})$ 
which may contain Maslov index 2 holomorphic discs are $\beta$ and
$H - 2\beta + k\alpha$ for $k \in \{-1,0,1\}$. 
\end{lemma}

\begin{proposition}[Chekanov-Schlenk \cite{CHESCH}] 
The torus $T^c_{r,0}$ bounds a unique $S^1$ family of holomorphic
discs in each of the classes $\beta$ and $H- 2\beta + k\alpha$ 
for $k \in \{-1,0,1\}$. These discs are regular, and the
corresponding algebraic count is 2 for $H - 2\beta$ and 1 for the other classes.
\end{proposition}

Since deforming $\lambda$ to $0$ yields a Lagrangian isotopy from
$T^c_{r,\lambda}$ to $T^c_{r,0}$ in the complement of $f^{-1}(0)$, so without encountering any
Maslov index $0$ holomorphic discs, the disc count remains the same and
we have that for $r < |c|$ the superpotential is given by

\begin{equation} \label{exWCHE}
W = u +  \frac{e^{-\Lambda}}{u^2w} +  2\frac{e^{-\Lambda}}{u^2} +  \frac{e^{-\Lambda}w}{u^2} =
 u +  \frac{e^{-\Lambda}(1 + w)^2}{w u^2}
\end{equation}

where $u$ and $w$ are the holomorphic coordinates on the mirror associated to 
the class $\beta$ and $\alpha$.

\begin{figure}[htb]
\begin{center}
\scalebox{0.5}{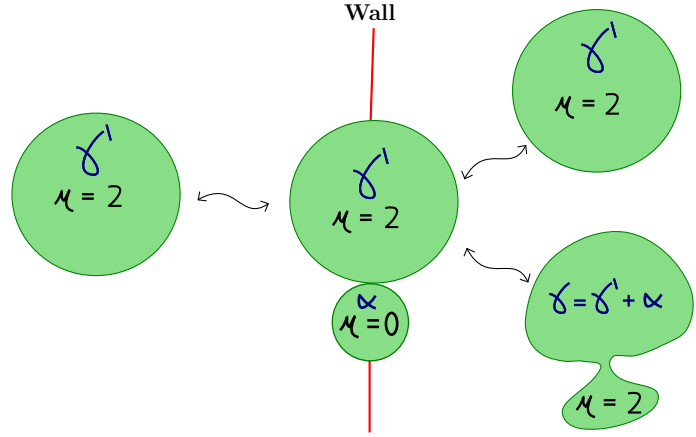}
 \end{center} 
\caption{Wall-crossing: Following a Maslov index 2 holomorphic disc 
through a Lagrangian deformation of the
 fibers crossing a wall consisting of fibers bounding Maslov index 0 discs (see Figure
\ref{figIlust} for the tropical picture).} 
\label{fig:wallcross}
\end{figure} 

\subsection{Wall-crossing} \label{WCross}

In this section we explain the wall-crossing phenomenon. Then we see how it 
happens in Example \ref{Ex} and explain the relation between
the two formulas for the superpotential in terms of the wall-crossing at $r =
|c|$, still following section 5 of \cite{DA07}.

Let us follow a Maslov index 2
 holomorphic disc in a class $\gamma'$ through a Lagrangian deformation of the
 fibers crossing a wall (formed by projection of fibers bounding Maslov index 0 discs).
 Assume that the given disc continues to exist
 throughout the deformation. The following phenomenon typically happens: if the
 boundary of such a disc intersects that of a Maslov index 0 holomorphic disc in
 a class $\alpha$ while on the wall, they can be glued into another Maslov index
 2 disc, in the class $\gamma = \gamma' + \alpha$, on the other side of the wall,
 besides the deformation that passes through, in the ``same" class $\gamma'$,
 without attaching the Maslov index 0 disc. Conversely, a Maslov index 2
 holomorphic disc in a class $\gamma$ can split into a Maslov index 2 holomorphic
 disc in a class $\gamma'$ and a Maslov index 0 holomorphic disc in a class
 $\gamma$, while on the wall, and then disappear after the Lagrangian passes
 through; see Figure \ref{fig:wallcross}. 
 
 We see how this phenomenon appears in the Example \ref{Ex}. Begin considering
 the case where $\lambda > 0$, so $T^c_{r,\lambda}$ lies in the region where
 $|x| > |y|$. Then when $r = |c|$ the torus intersects $\{y =0\}$ in a circle
 bounding a Maslov index $0$ disc, $u_0$. This disc represents the class
 $\alpha$, on the Chekanov side, and $\beta_1 - \beta_2$, on the Clifford side.
 As $r$ decreases through $|c|$, the family of holomorphic discs in the class
 $\beta_2$ on the Clifford side become the family of discs on the class $\beta$
 on the Chekanov side, and the discs in the class $H - \beta_1 - \beta_2$ on the
 Clifford side becomes the discs in the class $H -2\beta - \alpha$ on the
 Chekanov side. 
 
 Since a disc in the class $H -2\beta - \alpha$, bounded by a torus over the
 wall $r = |c|$, intersects $u_0$ in $[H -2\beta - \alpha]\cdot[\alpha] = 2$
 points, new discs in the classes $H- 2\beta$ and $H - 2\beta + \alpha$ arise
 from attaching $u_0$ to a disc in the class $H - \beta_1 - \beta_2 = H -2\beta
 - \alpha$ at one or both points where their boundaries intersect. On the other
 hand, a discs in the class $\beta$, at the wall, intersects the Maslov index 0
 disc $u_0$ at one point. When crossed to the Clifford side, a disc in the class
 $\beta_1 = \beta_2 + \alpha$ arrises from attaching $u_0$ to a disc in the
 class $\beta_2 = \beta$; see Figures \ref{figClif}, \ref{figChe},
 \ref{figIlust} (in these figures, discs are represented tropically). Conversely,
 one can think that a holomorphic disc in the class $\beta_1$ on the Clifford
 side, when deformed towards the wall, breaks into a holomorphic disc in the
 class $\beta = \beta_1 - \alpha$ and the Maslov index 0 disc $u_0$ and then
 disappears on the Chekanov side.

For $\lambda < 0$, when $r = |c|$ the torus intersects $\{x =0\}$ in a circle
bounding a Maslov index $0$ disc in the class $\beta_2 - \beta_1 = -\alpha$. As
$r$ decreases through $|c|$, the families of holomorphic discs that survive the
deformation through the wall are in the classes $\beta_2$ and $H - \beta_1 -
\beta_2$ on the Clifford side, becoming $\beta$ and $H -2\beta - \alpha$ on the
Chekanov side. As before, two new families of discs are created in the classes $H- 2\beta$ and $H -
 2\beta + \alpha$, while discs in the classes $\beta_1$ disappear, after wall-crossing. 

The difference between the ``naive" gluing formulas, which for $\lambda > 0$
matches $\beta $ $\leftrightarrows$ $\beta_2$ and for $\lambda < 0$ matches
$\beta$ $\leftrightarrows$ $\beta_1$, is due to the monodromy of the Lagrangian
fibers $T^{c}_{r ,\lambda}$ around the nodal fiber $T^{c}_{|c| ,0}$, which is
explained in the next section. However, one can perform wall-crossing
corrections to take care of this discrepancy and yield a single consistent
gluing for both halves of the wall. 

A holomorphic disc in the class $\beta$ on the Chekanov side is thought to
correspond to both discs in the classes $\beta_1$ and $\beta_2$, taking into
account the attachment of the holomorphic disc $u_0$ in the class $\alpha$. 
In terms of the coordinates $z_1$, $z_2$ on the Clifford side, associated
with $\beta_1$ and $\beta_2$, and coordinates $u$ and $w$ on the Chekanov
side, associated with $\beta$ and $\alpha$, the gluing becomes $u \leftrightarrows z_1 + z_2$.

For $\lambda > 0$, one can think that the ``naive" formula $u =
z_2$ is modified by a multiplicative factor of $1 + w$, i. e. , $u = (1 + w)z_2
= z_1 + z_2$, as predicted in Proposition 3.9 of \cite{DA07}. For $\lambda < 0$,
the correct change of coordinates is $ u = z_1(1 + w^{-1})
= z_1 + z_2$, $w^{-1} = z_2/z_1$, which is the same as for $\lambda > 0 $. 

Taking the wall-crossing into account the correct change of
 coordinates in the mirror is given as follows:

\begin{center}
\begin{tabular}{ c | c } 
    Homology Classes   & Coordinates  \\ 
   \hline \\
     $\ \ \ \ \ \ \, \alpha$ $\leftrightarrows$ $ \beta_1 - \beta_2$ & $w \leftrightarrows
\frac{z_1}{z_2}$ \\ \\
    $\hspace{1cm} \beta $ $\leftrightarrows$ $ \{ \beta_1, \beta_2\}$ &  $ \hspace{1cm}  
u \leftrightarrows z_1 + z_2$   \\   \\
$H -2\beta + \{ -1, 0 , 1\} \alpha$  $\leftrightarrows$ $   H - \beta_1 - \beta_2$ & 
$ \frac{e^{-\Lambda}(1 + w)^2}{u^2w}  \leftrightarrows \frac{e^{-\Lambda}}{z_1z_2}$   \\   
  \end{tabular}
\end{center}

It is then easy to check that the formulas \eqref{exWCLIF} and \eqref{exWCHE}
for the superpotential, and this corrected coordinate change, do match up.

\subsection{Almost toric manifolds} \label{atm}

The aim of this section is to explain the geometry of almost toric fibrations
and use it for a better understanding of the singular Lagrangian fibration
in the previous example. Most importantly, we can use it to construct 
other fibrations and predict the superpotential on each of the chambers divided 
by the walls. This way we can predict existence of exotic Lagrangian tori in almost toric manifolds,
and in particular the torus in $\CP^2$ that appears in Theorem \ref{mainthm}.
For a more detailed explanation of almost toric fibrations, see \cite{MS, NLMS}.

The following is definition 2.2 of \cite{NLMS}:
\begin{definition} 
 An \emph{almost toric fibration} of a symplectic four manifold $(M,\omega)$ is 
 a Lagrangian fibration $\pi: (M, \omega) \rightarrow B$ such that any point of 
 $(M, \omega)$ has a Darboux neighborhood (with symplectic form $dx_1\wedge dy_1 +
 dx_2\wedge dy_2$) in which the map $\pi$ has one of the following forms:

\begin{eqnarray*}
 \pi(x,y) & = & (x_1, x_2),
\hspace{5,7cm} \text{regular point}, \\
 \pi(x,y) & = & (x_1, x_2^2 + y_2^2),
\hspace{4,85cm} \text{elliptic, corank one}, \\ 
 \pi(x,y) & = & (x_1^2 + x_2^2, x_2^2 + y_2^2),
\hspace{4cm} \text{elliptic, corank two}, \\
 \pi(x,y) & = & (x_1y_1 + x_2y_2, x_1y_2 - x_2y_1),
\hspace{2,55cm} \text{nodal or focus-focus},  
\end{eqnarray*}
with respect to some choice of coordinates near the image point in $B$. An 
\emph{almost toric manifold} is a symplectic manifold equipped with an almost 
toric fibration. A \emph{toric fibration} is a Lagrangian fibration induced by 
an effective Hamiltonian torus action.
\end{definition}

We call the image of each nodal singularity a node.


Recall that a Lagrangian fibration yields an integral affine structure, called
symplectic, on the complement of the singular values on the base, i.e. each
tangent space contains a distinguished lattice. These lattices are defined in terms of
the isotropy subgroups of a natural action of $T^* B$ on $M$ given by the
time-one flow of a vector field associated with each covector of $T^* B$. More
precisely, take $\xi \in T^* B$ and consider the vector field $V_\xi$ defined by
$\omega(.,V_\xi) = \pi^* \xi$. Set $\xi \cdot x = \phi_\xi(x)$, where $\phi_\xi$
is the time-one flow of $V_\xi$. Call $\Lambda^*$ the isotropy subgroup of the
action, which is a lattice such that $(T^*B / \Lambda^* , d\alpha_{can})$ and
$(M, \omega)$ are locally fiberwise symplectomorphic (here, $\alpha_{can}$ is
induced by the canonical 1-form of $T^*B$). This induces two other lattices, the
dual lattice, $\Lambda$ given by $\Lambda_b = \{ u \in T_bB\ ; \ v^* u \in
\mathbb{Z} , \forall \ v^* \in \Lambda^*_b \}$, inside $TB$, and the vertical
lattice, $\Lambda^{vert} = \{V_\xi \ ; \ \xi \in \Lambda^*\}$, inside the
vertical bundle in $TM$. We call the pair $(B, \Lambda)$ a \emph{almost toric
base}.

For an almost toric four manifold, the affine structure defined by the
lattice above, completely determines $M$ up to symplectomorphism, at least when the
base is either non-compact or compact with non-empty boundary; see Corollary
5.4 in \cite{MS}. Also, since $(T^*B / \Lambda^* , d\alpha_{can})$ and $(M,
\omega)$ are locally fiberwise symplectomorphic, a basis of the lattice is in
correspondence with a basis of the first homology of the fiber over
a regular point, $H_1(F_b)$. Therefore, the topological monodromy around each
node is equivalent to the integral affine monodromy. The neighborhood of a nodal fiber is
symplectomorphic to a standard model (see section 4.2 in \cite{MS}) and the
monodromy around a singular fiber (of rank 1) is given by a Dehn twist. In
suitable coordinates the Dehn twist can be represented by the matrix : 

$$ A_{(1,0)} =
\begin{pmatrix} 1 & 1 \\ 0 & 1 \end{pmatrix}$$

A change of basis of $H_1(F_b)$ gives a conjugate of $A_{(1,0)}$, which is, in
terms of its eigenvector $(a,b)$: 

$$ A_{(a,b)} = \begin{pmatrix} 1 -ab & a^2 \\
-b^2 & 1 + ab \end{pmatrix}$$

Due to the monodromy, one cannot find an affine embedding of the base of an
almost toric fibration with nodes into $\mathbb{R}^2$ equipped with its standard
affine structure $\Lambda_0$. However, after removing a set of \emph{branch curves} in the
base $B$, i.e., a collection of disjoint properly embedded curves connecting each node
to a point in $\del\bar{B}$, it may be possible to define such an embedding. 

\begin{figure}[h!]
  \begin{center}
\scalebox{1}{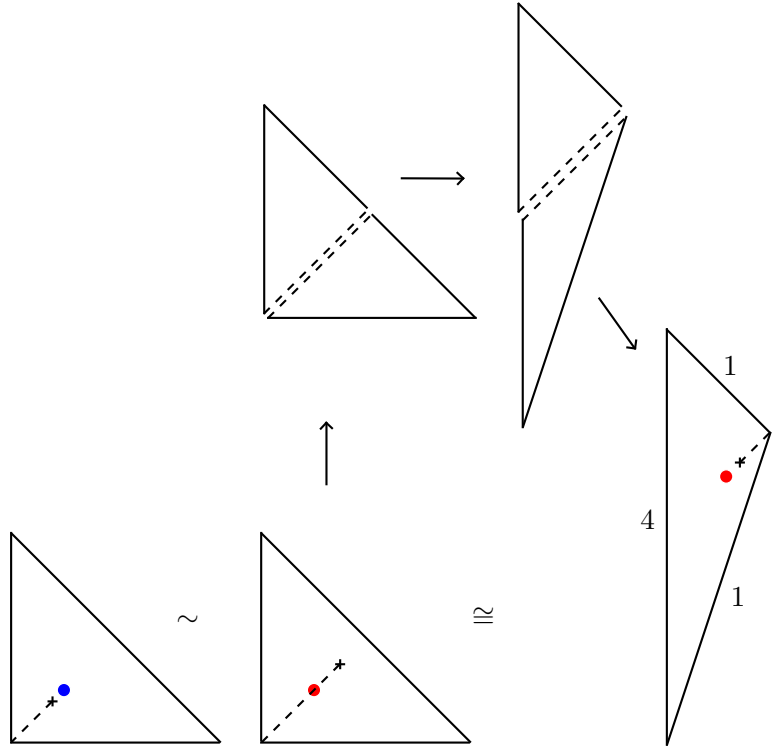}
 \caption{The leftmost picture is a base diagram of the almost toric fibration of $\CP^2$
 related to the singular Lagrangian fibration  
 given in Example \ref{Ex} (for small $c$), having the Clifford torus as its central fiber.
 The next base diagram is obtained by applying a
 nodal slide passing through the Clifford torus, so that the central fiber 
 becomes the Chekanov torus. Following the arrows we first cut the previous 
 picture in the direction of $(1,1)$, then we apply the monodromy $A_{(1,1)}$ to the 
 bottom part and finally we re-glue the parts to obtain a base diagram 
 representing the same almost toric fibration.} 
\label{Amt}
\end{center}
\end{figure}

\begin{definition}[3.2 of \cite{NLMS}]
  Suppose we have an integral affine embedding $\Phi: (B - R, \Lambda) \to (\mathbb{R}^2, \Lambda_0)$,
  where $(B, \Lambda)$ is an almost toric base and $R$ is a set of branch curves. A \emph{base
  diagram} of $(B, \Lambda)$ with respect to $R$ and $\Phi$ is the image of $\Phi$ decorated with 
  the following data:
  
  \begin{itemize}
  \item[-] an \emph{x} marking the location of each node and
  \item[-] dashed lines indicating the portion of $\del\overline{\Phi(B - R)}$ that 
  corresponds to $R$.
  \end{itemize}    
\end{definition}

\begin{remark} The presence of monodromy in the affine structure on $B$ implies
the existence of monodromy in the affine structure induced on the mirror
$X^{\vee}$. This explains the discrepancy between the uncorrected coordinate
changes across the two halves of the wall in Example \ref{Ex}. See remark 5.11 in \cite{DA07}.
\end{remark}

The affine direction(s) of the image of such a branch curve in $\R^2$ determine
the monodromy around the corresponding node. If the image is contained in a
line with direction $(a,b)$, the monodromy is given by $A_{(a,b)}$; for a more
detailed account of base diagrams, see section 5.2 of \cite{MS}. For instance,
the leftmost picture of Figure \ref{Amt} represents the Lagrangian fibration
seen in Example \ref{Ex}. The ray represented by dashed lines in the direction
$(1,1)$ is an eigenvector of the monodromy, which hence is given by $A_{(1,1)}$. 

Two almost toric surgery operations are of importance for us. They change the
almost toric fibration into another almost toric fibration of the same 
symplectic four manifold and are defined as follows:

\begin{definition}[4.1 of \cite{NLMS}]
  Let $(B,\Lambda_i)$ be two almost toric bases, $i = 1, 2$. We say that 
  $(B,\Lambda_1)$ and $(B,\Lambda_2)$ are related by a \emph{nodal slide} 
  if there is a curve $\gamma$ in $B$ such that 
  \begin{itemize}
    \item[-]$(B - \gamma,\Lambda_1)$ and $(B - \gamma,\Lambda_2)$
  are isomorphic,
    \item[-]$\gamma$ contains one node of $(B,\Lambda_i)$ for each $i$ and 
    \item[-]$\gamma$ is contained in the eigenline (line preserved by the monodromy)
     through that node.
\end{itemize}   
\end{definition}

\begin{definition}[4.2 of \cite{NLMS}]
  Let $(B_i,\Lambda_i)$ be two almost toric bases, $i = 1, 2$. We say that 
  $(B_1,\Lambda_1)$ and $(B_2,\Lambda_2)$ differ by a \emph{nodal trade} 
  if each contains a curve $\gamma_i$ starting at $\del B_i$ such that
  $(B_1 - \gamma_1,\Lambda_1)$ and $(B_2 - \gamma_2,\Lambda_2)$
  are isomorphic, and $(B_1, \Lambda_1)$ has one less vertex than $(B_2,\Lambda_2)$.
\end{definition}

\begin{remark}
  The rightmost picture of Figure \ref{Amt} is not considered to differ by a 
  nodal trade from the moment polytope of $\CP(1,1,4)$, because the latter, being the base
  of an \emph{orbifold} toric Lagrangian fibration, is not considered to be an almost toric 
  base.  
\end{remark}

In Figure \ref{Amt}, the leftmost base diagram is obtained by applying a nodal
trade to a corner of the moment polytope of $\CP^2$, which is the base diagram
for the standard toric fibration of $\CP^2$. The following picture is then
obtained by a nodal slide. As explained in the introduction, once the singular
fiber passes trough the Clifford torus, the central fiber develops a singularity and
then becomes the Chekanov torus. The rightmost picture is a $B(1,1,4)$
base diagram representing the same almost toric fibration. Shortening the cut to
a limit situation where it hits the corner describes a degeneration of $\CP^2$
into $\CP(1,1,4)$.

\subsection{Holomorphic discs viewed from almost toric fibrations and wall-crossing} \label{Poten}

 In this section we use almost toric pictures to describe a limit affine
 structure of the fibration for which holomorphic curves converge to tropical
 curves. We illustrate the Maslov index 2 tropical discs given in this limit
 affine structure for the almost toric fibration considered of Example \ref{Ex}.
 This section is not intended to contain a rigorous approach to the
 correspondence between tropical curves and holomorphic discs in an almost toric
 setting.

Assume one has an almost toric fibration with special Lagrangian fibers with
respect to $\Omega$, a holomorphic 2-form with poles on the divisor $D$ that
projects to the boundary of the base $B$. Then the interior of $B$ carries a
second affine structure, sometimes called complex. The lattice which describes
this affine structure, which we denote by $\Lambda^c$, is given by identifying
$T_bB \simeq H^1(L_b, \R)$, via the flux of the imaginary part of $\Omega$ and
via Poincar\`e duality with $ H_1(L_b, \R) \supset H_1(L_b, \Z)$. More
precisely, for each vector $v \in T_bB$ one gets the element of $H^1(L_b, \R)$
given by the homomorphism 

\begin{equation*}
[\gamma] \in H_1(L_b,\R) \mapsto \left. \frac{d}{dt} \right| _{t=0} \int_{\Gamma_t} 
Im(\Omega),
\end{equation*}

\noindent where $\Gamma_t$ is given by any parallel transport of $\gamma$ over a curve
$c(t)$ on the base, with $c(0) = b$ , $c'(0) = v$. Since $Im(\Omega)$ is a
closed form, vanishing on the fibers, the above is independent of $c(t)$ and
$\Gamma_t$, and hence well defined. A fiber over the
boundary of $B$ is infinitely far from a given fiber over an interior point, since
$\Omega$ has a pole on the divisor $D$.

\begin{figure}[h] 
\begin{center}
\setlength{\unitlength}{7mm}
\begin{picture}(13.5,5)(0,0.2)
\psset{unit=\unitlength}
\newgray{ltgray}{0.85}

\psellipse[linecolor=red](4.7,4)(0.6,1.2)
\psline[linearc=0.7](4.7,3.3)(5,4)(4.7,4.8)
\psline[linearc=0.7](4.8,3.5)(4,4)(4.8,4.6)
\psellipse(4.7,4)(1,2)

\psellipse[linecolor=red](9.7,4)(0.6,1.2)
\psline[linearc=0.7](9.7,3.3)(10,4)(9.7,4.8)
\psline[linearc=0.7](9.8,3.5)(9,4)(9.8,4.6)
\psellipse(9.7,4)(1,2)

\psline[linecolor=red](4.7,5.2)(9.7,5.2)
\psline[linecolor=red](4.7,2.8)(9.7,2.8)

\psline[linecolor=blue](3.2,1)(11.5,1)
\psline[linecolor=blue](2.5,2)(3.2,1)
\psline[linecolor=blue](1.4,0.1)(3.2,1)

\psline[linearc=1,linecolor=red](1.4,0.1)(3.4,0.8)(11.5,1)
\psline[linearc=0.5,linecolor=red](1.4,0.1)(2.5,1.1)(2.5,2)
\psline[linearc=1,linecolor=red](2.5,2)(3.4,1.2)(11.5,1)

\put(0,0){\line(1,0){10.5}}
\put(0,0){\line(1,1){2}}
\put(2,2){\line(1,0){10.5}}
\put(10.5,0){\line(1,1){2}}
\end{picture}
\end{center}
\caption{After a deformation of the almost complex structure, $J$-holomorphic discs 
project to amoebas eventually converging to tropical curves in the large complex structure limit.}
\label{amoeba}
\end{figure}
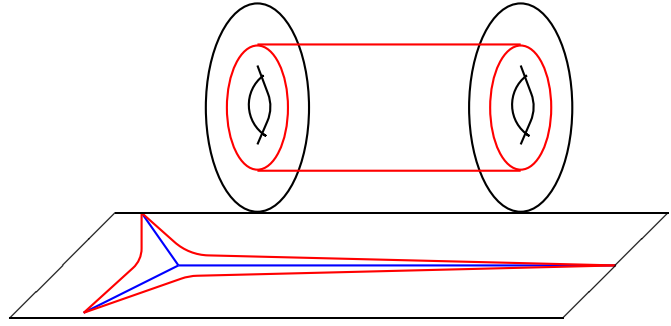

In general, the projections to $B$ of holomorphic curves, called amoebas, can be
fairly complicated. However, it is expected that under a suitable deformation of
the almost complex structure $J$ towards a `large limit' (where the base
directions are stretched), the amoebas converge to tropical curves; see Figure
\ref{amoeba}. Also, the wall generated by the singular fiber converges to a
straight line with respect to this affine structure, since it is the projection
of a holomorphic curve containing Maslov index zero discs bounded by the fibers.
Moreover, since the boundary of such a disc represents the vanishing cycle in
the neighborhood of the nodal fiber, its homology class is fixed by the
monodromy. Hence the straight line corresponding to the wall is in the direction
of the eigenvector of the affine monodromy. In a neighborhood of a fiber away
from the singular ones the almost toric fibration are expected to approach $TB/\epsilon \Lambda^c$
with $\epsilon \to 0$ at the limit. This way, the change of coordinates and 
monodromy for this `large limit' complex affine structure is given by the 
transpose inverse of the symplectic affine structure defined in section 
\ref{atm}, where the neighborhood of a regular fiber is isomorphic to
a neighborhood of $T^*B/\Lambda^*$. Also, our `limit lattice' at a point
$b$ on the base is identified with $H_1(L_b, \Z)$.

This principle is illustrated for Example \ref{Ex} in Figures \ref{figClif} 
and
\ref{figChe}. In these two figures:
\begin{itemize}
  \item[-] The Lagrangian torus under consideration is
the fiber over the thick point. 
  \item[-] The dashed lines represent the walls (long
dashes) and the cuts (short dashes), and `x' represents the node
(singular fiber). 
 \item[-] A {\it tropical disc} is a tree whose edges are straight lines with
rational slope in $B$, starting at the torus and ending on the nodes or
perpendicular to the boundary at infinity. The internal vertices satisfy the
balancing condition that the primitive integer vectors entering each vertex of
the tree, counted with multiplicity, must sum to 0. 
\item[-] The Maslov index of the disc equals twice the number of intersection with 
the boundary at infinity, i.e., the divisor.
\item[-] The multiplicity of each
edge is depicted by the numbers of lines on Figure \ref{figChe}, but on some other
figures the multiplicities are represented by the thickness of the line, for
visual purposes (they can be computed taking into account the balancing
condition). 
\item[-] The vanishing cycle is represented by $(-1 , 1)$ on the lattice $H_1(L_b, \Z)$.

\end{itemize}

 \begin{figure}[h!] 
\begin{minipage}[b]{0.5\linewidth}
\begin{center}
\scalebox{0.5}{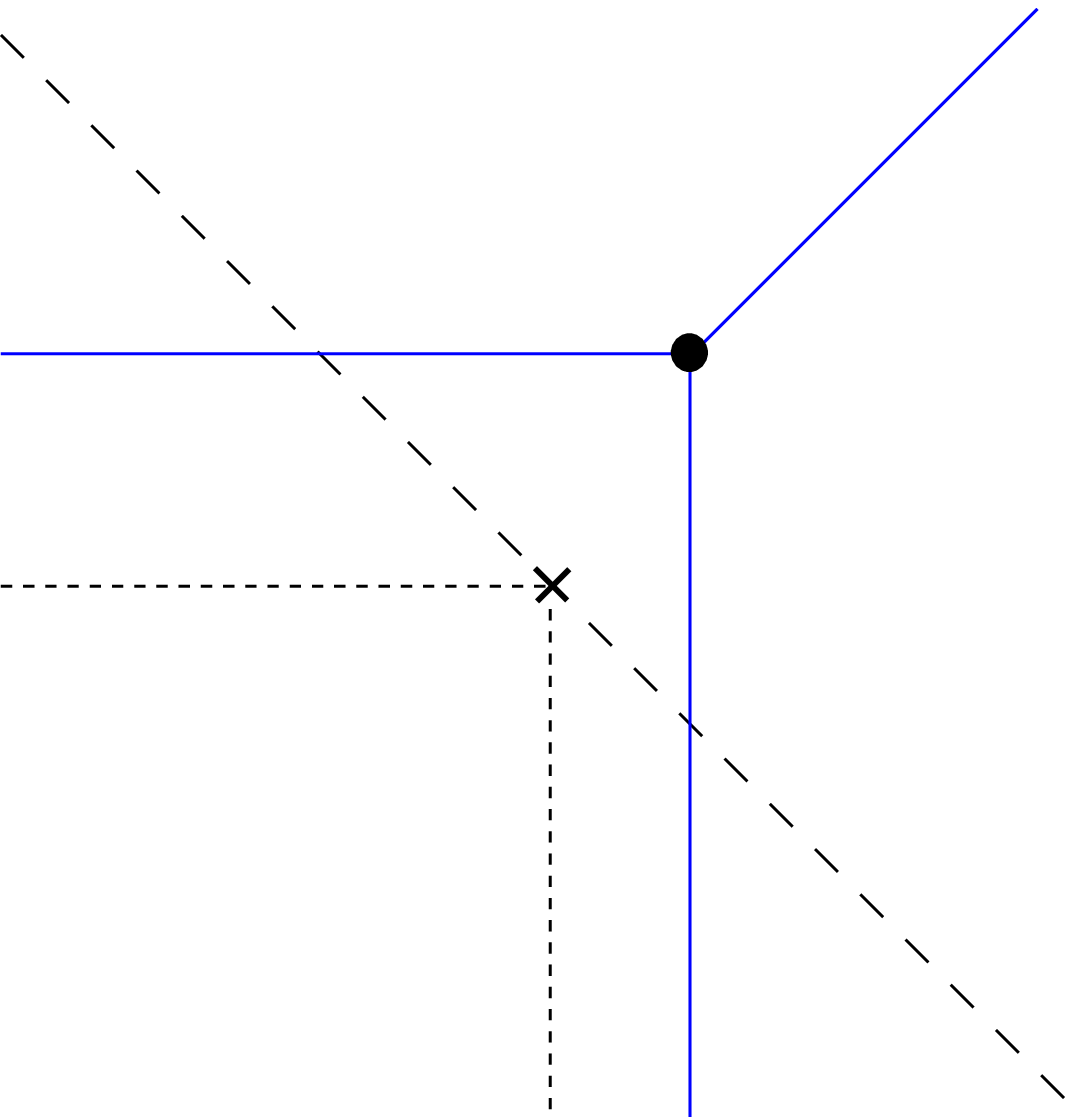}
 \end{center}
 \caption{Clifford type torus. \newline  $W = z_1 + z_2 + \frac  {e^{-\Lambda }}{z_1z_2}$.}
\label{figClif}
\end{minipage}
\hfill
\begin{minipage}[b]{0.5\linewidth}
\begin{center}
\scalebox{0.5}{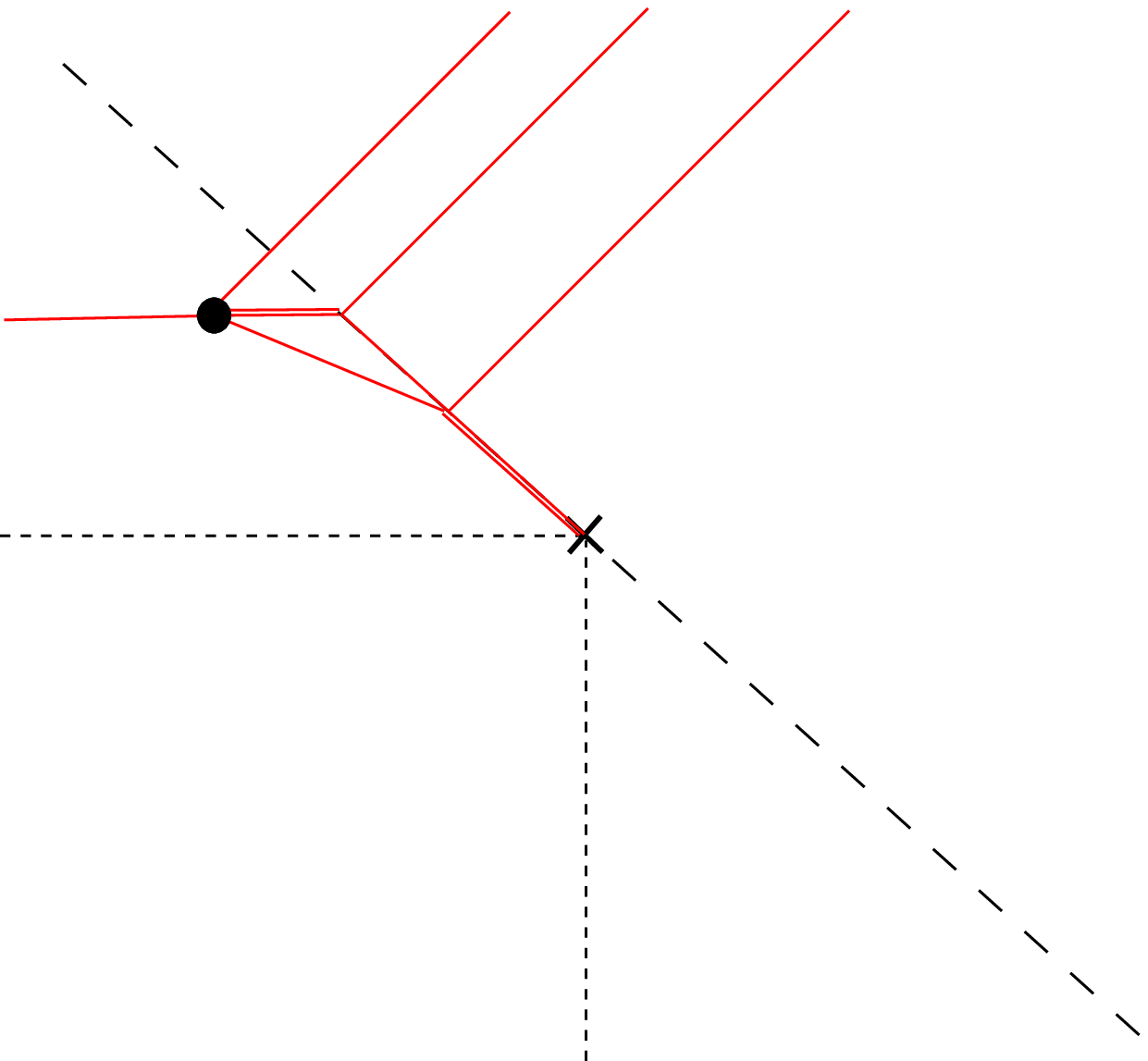}
 \end{center}
 \caption{Chekanov type torus. $W = u + \frac  {e^{-\Lambda }}{u^2w} + 2\frac  {e^{-\Lambda }}{u^2} 
 + \frac  {e^{-\Lambda }w}{u^2} = u + \frac  {e^{-\Lambda }(1 + w)^2}{u^2}$.}
\label{figChe}
\end{minipage}
\end{figure}

 \begin{figure}[h!] 
\begin{minipage}[b]{0.3\linewidth}
\begin{center}
\scalebox{0.4}{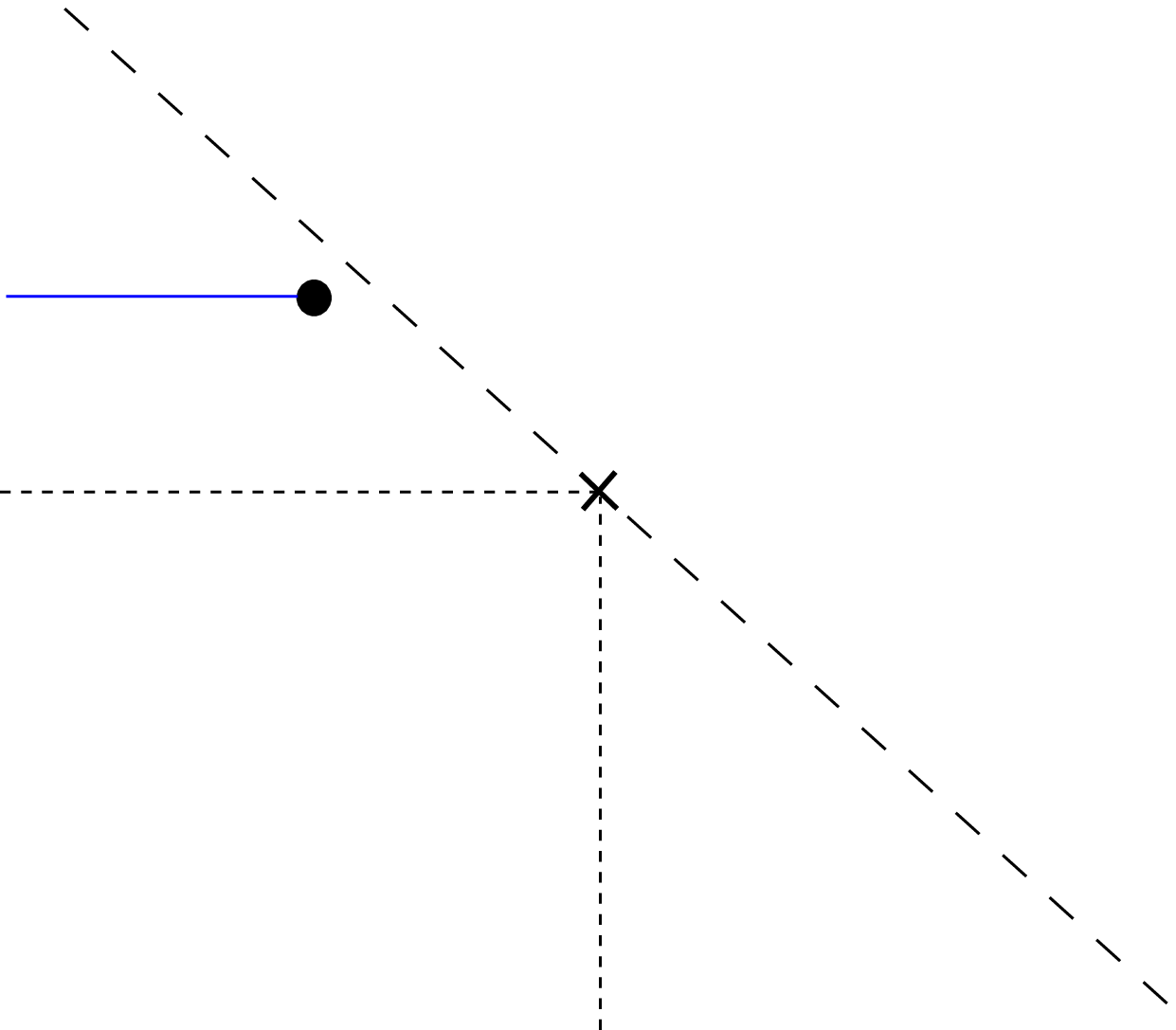}
 \end{center}

\end{minipage}
\hfill
\begin{minipage}[b]{0.3\linewidth}
\begin{center}
\scalebox{0.4}{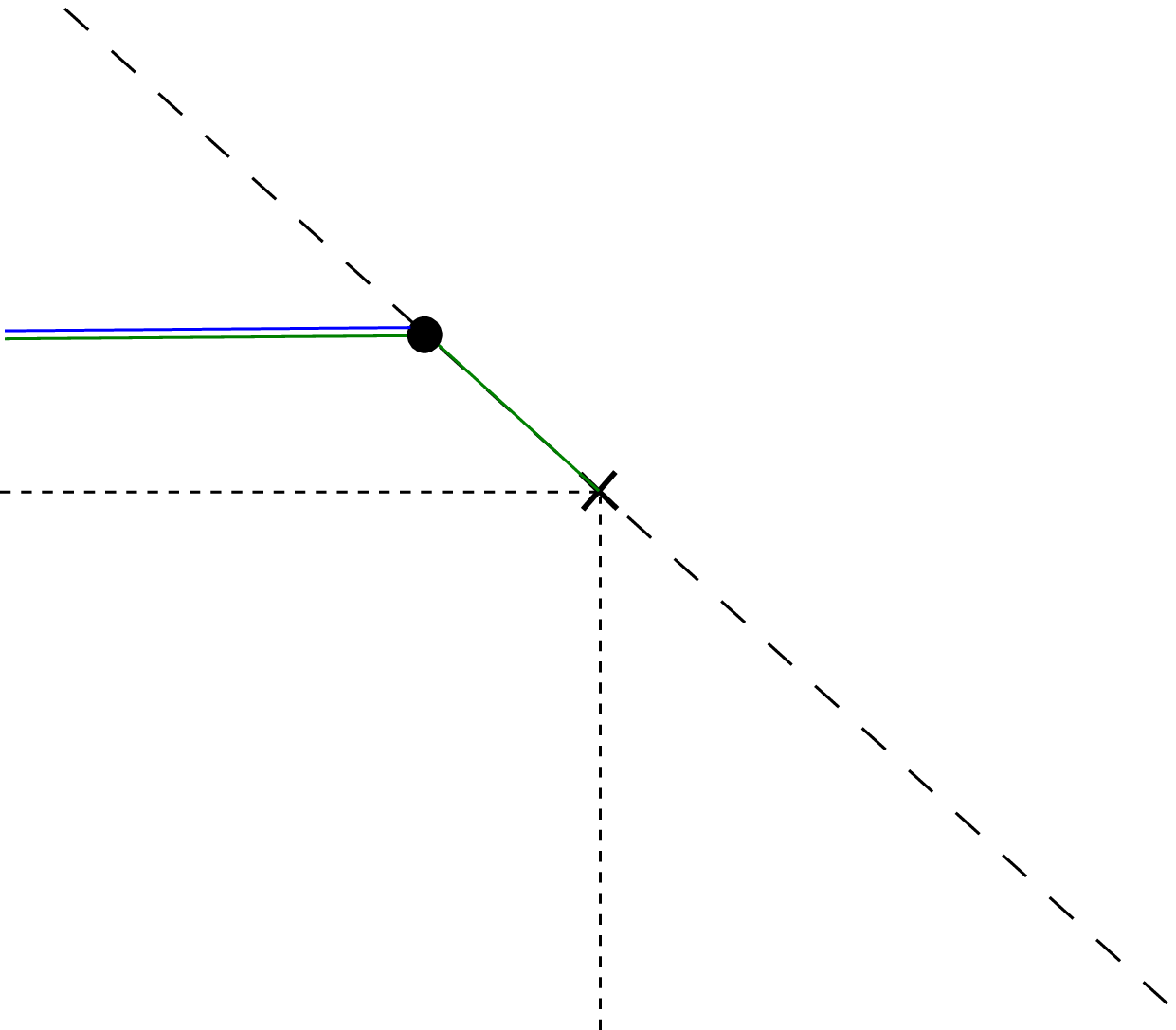}
 \end{center}
 
\end{minipage}
\hfill
\begin{minipage}[b]{0.3\linewidth}
\begin{center}
\scalebox{0.4}{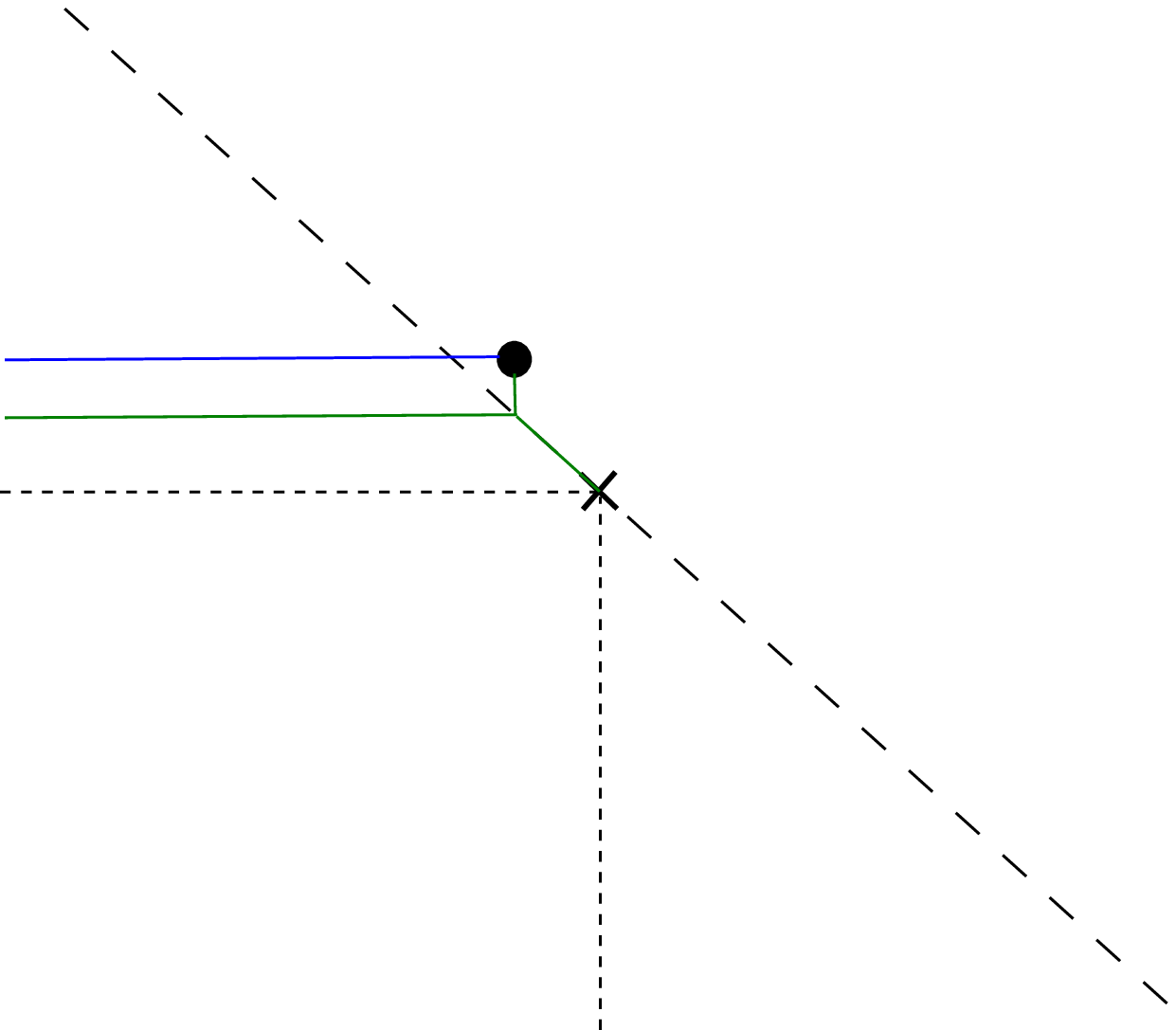}
 \end{center}

\end{minipage}

\caption{As in Example \ref{Ex} (see also Figure \ref{fig:wallcross}), a disc in the class $\beta _2$ breaks into a disc in the class $\beta _1$ and 
 the exceptional disc $u_0$, and disappears after crossing the wall, for $\lambda <
0$.}
\label{figIlust}
\end{figure}

 The relation between these pictures and the formulas in section \ref{Ex} is as
 follows: $z_1$, $z_2$ are coordinates on the Clifford side associated with
 the vectors $(1,0)$ and $(0,1)$, respectively, and $u$, $w$ are coordinates on
 the Chekanov side associated with the vectors $(1,0)$ and $(-1,1)$,
 respectively, for the top part of the Chekanov side (when $\lambda < 0$).
 The direction of the edge leaving the torus can be read off from
 the superpotential and is the negative of the vector representing the
 exponents of the corresponding monomial. For instance, the disc associated with
 the monomial $\frac {e^{-\Lambda }}{z_1z_2}$ in \eqref{exWCLIF} leaves the torus
 with tangent vector $(1,1)$ in Figure \ref{figClif}, while the disc associated
 with the term $\frac {e^{-\Lambda }}{u^2} $ in \eqref{exWCHE} has tangent vector
 $(2,0)$ (multiplicity 2) in Figure \ref{figChe}. We call this vector the
 ``class" of the tropical disc.
 


The formulas for the two superpotentials are related by a {\it wall-crossing}
{\it transformation} (or {\it mutation}). We describe it now for the case of a
two dimensional base. In what follows, when referring to a particular fiber
$L_b$, let $\beta_1$, $\beta_2$ be relative homotopy classes of discs with
boundary on $L_b$ such that $\del \beta_1$ , $\del \beta_2$ are associated with
$(1,0)$ and $(0,1)$ seen as elements of $H_1(L_b, \Z)$, respectively. Moving the
point $b$ on the base, we keep denoting by $\beta_1$, $\beta_2$ the continuous
deformations of this relative classes. Consider a wall in $B$ coming from the
projection of a family of Maslov index zero discs propagating out of a node in
the base along the affine direction $(m,n) \in \Z^2$. (Here we only consider the
part of the wall that lies on the positive half of the eigenline generated by
$(m,n)$). Set $\mathcal{W}^+ = \{ v \in \R^2 | \{v, (m,n)\} \ \text{is
positively oriented} \}$ and $\mathcal{W}^- = \{ v \in \R^2 | \{v, (m,n)\} \
\text{is negatively oriented} \}$. Denote by $z_1$, $z_2$ (respectively $u_1$,
$u_2$), the coordinates associated with $\beta_1$, $\beta_2$ for the fibers
$L_b$ with $b$ in the chamber contained in $\mathcal{W}^+$ (respectively
$\mathcal{W}^-$). The class of the primitive Maslov index zero discs bounded by
the fibers along the wall is of the form $m\beta_1 + n\beta_2 + k[\CP^1]$ and
hence is represented by the monomial $w = e^{-k \Lambda} z_1^m z_2^n = e^{-k
\Lambda} u_1^m u_2^n$. The coordinates $z_1$, $z_2$ and $u_1$, $u_2$ are then related
by the:

\textbf{Wall-crossing formula}
\begin{equation}\label{WCForm}
   u_1^r u_2^s \rightleftharpoons z_1^r z_2^s(1 + w)^{nr - ms} 
\end{equation}

That way, knowing the superpotential on one side of the wall, it is expected
that one can compute the superpotential of the other side by applying the above
wall-crossing formula. Note that the absolute value of the exponent $|nr - ms|$
is the intersection number between a disc represented by $z_1^r z_2^s$ and the
Maslov index zero disc at a fiber over the wall.

\section{Predicting the number and relative 
homotopy classes the $T(1,4,25)$ torus bounds} \label{Pred}

In this section we apply the same ideas as in the previous section to another
almost toric fibration, any of the ones shown on Figure \ref{AmtC}, to predict the superpotential
of the $T(1,4,25)$ {\it type} torus, obtained from the previous Chekanov torus after
another wall-crossing. 
 
Figure \ref{AmtC} represents almost toric fibrations on $\CP^2$ containing two
singular fibers of rank one. The middle diagram arises by applying a nodal trade
to one corner of the right-most diagram in Figure \ref{Amt}, after redrawing the
diagram via an element of $AGL(2,\Z)$, which is, up to translation,
$\begin{pmatrix} 1 & 0 \\ -3 & 1 \end{pmatrix}$. Lengthening the new cut to pass
trough the central fiber we end up with an almost toric fibration having the
monotone $T(1,4,25)$ torus as a fiber, as illustrated by the right-most diagram
in Figure \ref{AmtC}.

We assume the Lagrangian fibers are special with respect to some 2-form $\Omega$
with poles on the divisor, and that in a `large limit' almost complex structure,
pseudo-holomorphic curves project to tropical curves. We will start the
description of the superpotential in the chambers where the fibers are of
Clifford type and successively cross two walls in order to arrive at a tentative
formula for the superpotential in the chamber where the fibers are $T(1,4,25)$
type tori. In section \ref{ExoTor}, we construct a singular Lagrangian fibration
interpolating between Chekanov type tori and $T(1,4,25)$ type tori.
For the later, the count of Maslov index 2 holomorphic discs is verified 
rigorously in section \ref{CHD} using only symplectic geometry techniques.

\begin{figure}[h!] 
\begin{center}
\scalebox{1}{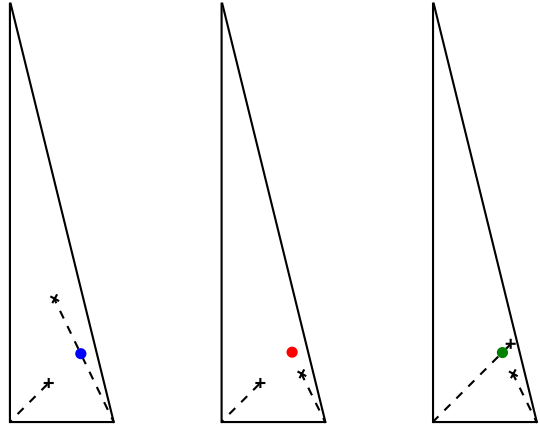}
 \end{center}
 \caption{Almost toric base diagrams of $\CP^2$ having, respectively, the monotone 
 Clifford, Chekanov and $T(1,4,25)$ tori as the a fiber.}
\label{AmtC}
\end{figure}

Now we focus on the chamber containing the Clifford type tori illustrated on the
top right part of Figure \ref{figClif1}. Since we applied $\begin{pmatrix} 1 & 0
\\ -3 & 1 \end{pmatrix}$ to the right-most diagram of Figure \ref{Amt}, we need
to `dually' change the coordinates used on Figure \ref{figClif} by applying the
transpose inverse $\begin{pmatrix} 1 & 3 \\ 0 & 1\end{pmatrix}$. Calling these
new coordinates $\hat{z_1}$, $\hat{z_2}$, associated with $(1,0)$ and $(0,1)$
on the top right part of Figure \ref{figClif1}, they are described in terms of $z_1$,
$z_2$ by $z_1 = \hat{z}_1 $, $z_2 = \hat{z}^3_1\hat{z}_2$. This way a tropical
disc in the class $(p,q)$ on the top right part of Figure \ref{figClif1}
corresponds to a monomial with exponents $\hat{z}_1^p\hat{z}_2^q$. After this
change of coordinates, the invariant direction at the singularity is $(2,1)$ and
the monodromy is given by $A_{(2,1)}$; see Figure \ref{figClif1}. Therefore at the top
part of the chamber corresponding to the Clifford type tori the superpotential
is given by 

\begin{equation} \label{WClif1}
W_{Clif} = \hat{z}_1 + \hat{z}^3_1\hat{z}_2 + \frac
{e^{-\Lambda }}{\hat{z}^4_1z_2}
\end{equation}

As the vanishing vector of the first wall is $(2,1)$, the vanishing class is
represented by the coordinate $\w = \hat{z}_1^2\hat{z}_2 = \z_1^2\z_2$, where
$\z_1$, $\z_2$ are the coordinates corresponding to the standard basis on the
Chekanov side. Applying the wall crossing formula \ref{WCForm} to
\eqref{WClif1}, the superpotential in the Chekanov region is 

\begin{equation} \label{WChe1}
W_{Che} = \z_1 + e^{-\Lambda}\frac{(1 + \w)^2}{\z_1^4\z_2} = \z_1 +
e^{-\Lambda }\z_2 + 2\frac  {e^{-\Lambda }}{\z_1^2} + \frac  {e^{-\Lambda}}{\z_1^4\z_2}.
\end{equation}

\begin{figure}[h!] 
\begin{minipage}[b]{0.5\linewidth}
\begin{center}
\scalebox{0.4}{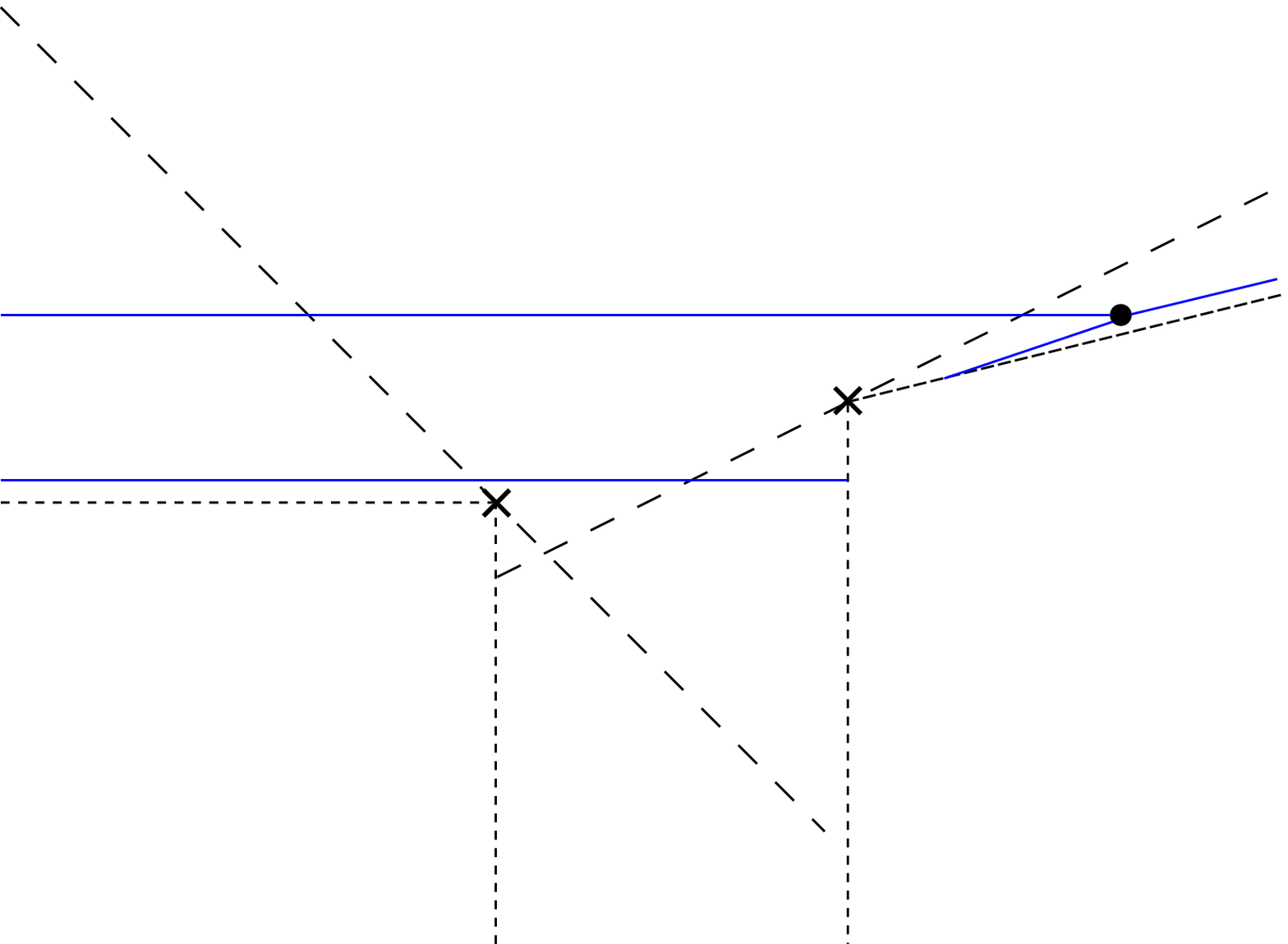}
 \end{center}
 \caption{Clifford type torus.\newline$W_{Clif} = \hat{z}_1 + \hat{z}^3_1\hat{z}_2 + 
 \frac {e^{-\Lambda }}{\hat{z}^4_1z_2}$.}
\label{figClif1}
\end{minipage}
\hspace{0.2cm}
\begin{minipage}[b]{0.5\linewidth}
\begin{center}
\scalebox{0.4}{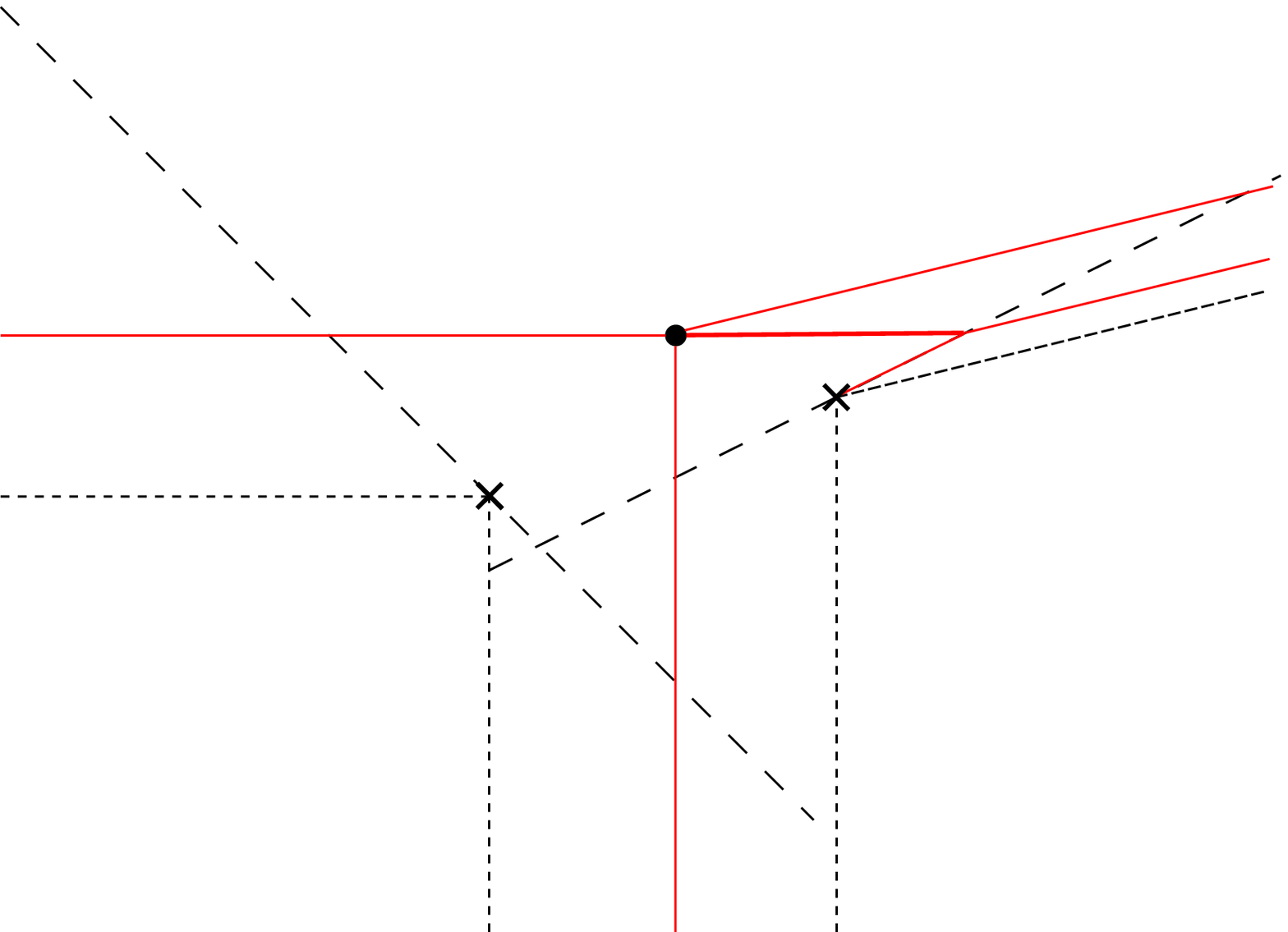}
 \end{center}
 \caption{Chekanov type torus. $W_{Che} = \z_1 +e^{-\Lambda}\z_2 + 
 2\frac{e^{-\Lambda }}{\z_1^2} + \frac{e^{-\Lambda }}{\z_1^4\z_2}$.}
\label{figChe1}
\end{minipage}
\end{figure}

We now cross the second wall towards a $T(1,4,25)$ type torus. The second wall
has vanishing vector $(-1,1)$ and the monomial corresponding to the vanishing
class is $w = e^{-\Lambda}\z_2 \z_1^{-1} = e^{-\Lambda}u_2 u_1^{-1}$, where
$u_1$, $u_2$ are are the coordinates corresponding to the standard basis on the
$T(1,4,25)$ side and the factor $e^{-\Lambda}$ is present because the class of
the Maslov index 0 disc is $-\beta_1 + \beta_2 + [\CP^1] \in \pi_2(\CP^2 ,L)$,
where $\beta_1$ and $\beta_2$ are the classes associated with the coordinates,
$\z_1$, $\z_2$. Indeed, knowing the boundary of $w$ represents the class
$(-1,1)$, we get the first two coefficients of $\beta_1$ and $\beta_2$. To
obtain the coefficient of $[\CP^1]$ we compute the Maslov index. We have that
$\mu([\CP^1]) = 6$ (see Lemma \ref{MI}) and $\z_1 $ and $e^{-\Lambda}\z_2 $ are
terms in $W_{Che}$, hence $\mu(\beta_1) = 2 $ and
$\mu(\beta_2) = -4$. In order to have Maslov index 0, the coefficient of
$[\CP^1]$ must be 1. Finally, applying the wall-crossing formula \ref{WCForm} to
\eqref{WChe1} we get that

\begin{eqnarray}\label{WChe2}
W_{T(1,4,25)} & = & u_1 + 2\frac{e^{-\Lambda }}{u_1^2}(1 + w)^2 + \frac{e^{-\Lambda }}{u_1^4u_2}(1 + w)^5 \nonumber \\
                 & = & u_1 + 2\frac{e^{-\Lambda }}{u_1^2} + 4\frac{e^{-\Lambda }u_2}{u_1^3} + 
 2\frac{e^{-\Lambda}u_2^2}{u_1^4} + \frac{e^{-\Lambda }}{u_1^4u_2} + 5\frac{e^{-2\Lambda}}{u_1^5}  \nonumber \\
                 &  & + 10\frac{e^{-3\Lambda}u_2}{u_1^6}  + 10\frac{e^{-4\Lambda}u_2^2}{u_1^7} + 
 5 \frac{e^{-5\Lambda}u_2^3}{u_1^8}  + \frac{e^{-6\Lambda}u_2^4}{u_1^9} \nonumber \\
                 & = & u + 2\frac{e^{-\Lambda }}{u^2}(1 + w)^2 + \frac{e^{-2\Lambda }}{u^5w}(1 + 
                 w)^5.
\end{eqnarray}

The last formula is a more simplified expression in terms of the coordinates $u = u_1$ and
$w$. The expanded version in coordinates $u_1$, $u_2$ makes it easier to visualize the
class of each disc. Figure \ref{fig10fam} illustrates a $T(1,4,25)$ type torus,
predicted to bound 10 different families of holomorphic discs, corresponding to 
the 10 terms in this expression.

\begin{figure}[h!]
\begin{center}
\scalebox{0.7}{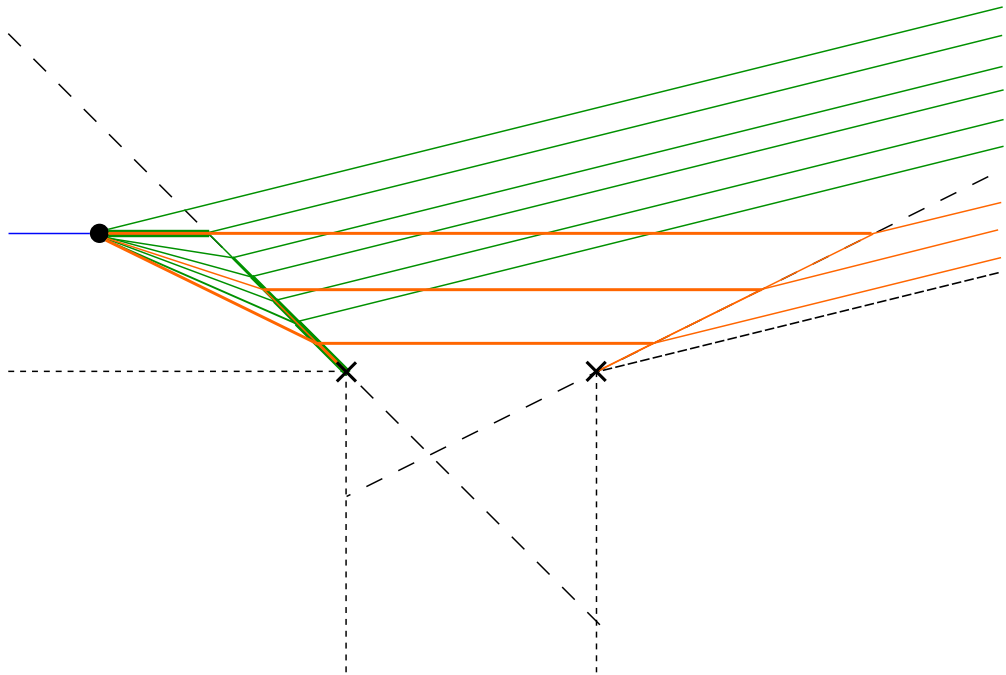}
 \end{center}
 \caption{A $T(1,4,25)$ type torus bounding 10 families of Maslov index 2 holomorphic discs. 
 The superpotential is given by $ W_{T(1,4,25)} = u + 2\frac{e^{-\Lambda }}{u^2}(1 + w)^2 + 
 \frac{e^{-2\Lambda }}{u^5w}(1 + w)^5$.}
\label{fig10fam}
\end{figure}

Even though our approach in this section was not completely rigorous, it points
toward the existence of such an exotic torus bounding 41 discs, if we count with
multiplicity (sum the coefficients of each monomial). The theory for proving the
correspondence between tropical curves on the base and holomorphic curves on
the total space is not fully developed yet, so the actual proof in section \ref{CHD}
will use a different approach.\\

\begin{remark} 
The bottom most region on Figures \ref{figClif1}-\ref{fig10fam} is known to have infinitely many
walls, since it can have Maslov index 0 discs ending in both nodes with
different multiplicities. This can be detected by the need for consistency 
of the changes of coordinates due to wall-crossing when we go around the 
point where the walls intersect. This phenomenon is called scattering,
first described by M. Kontsevich and Y.Soibelman in \cite{KS}. See also
M. Gross \cite{GS}.
\end{remark}

\section{The exotic torus} \label{ExoTor}

This section is devoted to the actual construction of the exotic torus.
Heuristically, we try to mimic the following procedure: first perform a nodal
trade in a smooth corner of the moment polytope of $\CP(1,1,4)$ to get an
``orbifold almost toric fibration"; then we smooth the orbifold singularity
and trade it for an interior node, obtaining the almost toric fibration described by the middle diagram of
Figure \ref{AmtC}. This way the analogue of a Chekanov type torus in
$\CP(1,1,4)$ deforms to a $T(1,4,25)$ type torus. For the first step, we
consider a symplectic fibration on $\CP(1,1,4)$ given by $f_0(\x: 1: \z) = \x
\z$, with fibers preserved by a circle action $e^{i\theta}\cdot(\x: 1: \z) =
(e^{-i\theta}\x : 1 :e^{i\theta}\z)$. The parallel transport of an orbit along a
circle centered at $c \in \R_{>0}$ with radius $r < c$ is then a `Chekanov type
torus in $\CP(1,1,4)$'. The second part is carried out using a degeneration from
$\CP^2$ to $\CP(1,1,4)$ parametrized by a real parameter $t$. We then consider a
family of symplectic fibrations $f_t$ on $\CP^2 \setminus \{y = 0\}$ converging
to $f_0$, compatible with a circle action, so that, for each $t > 0$, the
parallel transport of an orbit along the same circle at the base is a
$T(1,4,25)$ type torus which converges to a `Chekanov type torus in $\CP(1,1,4)$ 
as $t \to 0$. One technical issue that arrises is that we need to equip $\CP^2$
with a non-standard K\"ahler form (symplectomorphic to the standard one) in order to
be able to give explicit descriptions of these tori.

As mentioned in the introduction, the projective plane degenerates to weighted projective spaces
$\CP(a^2,b^2,c^2)$, where $(a,b,c)$ is a Markov triple. For $c' =3ab - c$, a deformation
from $\CP(a^2,b^2,c^2)$ to $\CP(a^2,b^2,c'^2)$ can be seen
explicitly inside $\CP(a^2,b^2,c, c')$ via the equation , $z_0z_1 -
(1-t)z_2^{c'} - tz_3^c = 0$.

We are going to work only with $\CP^2 = \CP(1,1,1)$ and $\CP(1,1,4)$ inside
$\CP(1,1,1,2)$. For $t \in [0,1]$, let $X_t$ be the surface $z_0z_1 - (1-t)z_2^{2} - tz_3 = 0$.
Explicit embeddings are

\begin{eqnarray} \label{embCP2}
 \CP(1,1,1) &\longrightarrow& \CP(1,1,1,2)     \nonumber \\
   (x :y: z) &\mapsto& (x :y : z : \frac{xy - (1-t)z^2}{t}) \ \ \  \ \ \ \ \text{for}  \ \ t \neq 0, 
\end{eqnarray}

\begin{eqnarray}
  \CP(1,1,4) &\longrightarrow& \CP(1,1,1,2)     \nonumber \\
   (\x :\y: \z) &\mapsto& (\x^2 : \y^2 :\x\y : \z) \ \ \  \ \ \ \ \text{ for}  \ \ t = 0. 
\end{eqnarray}

Set $\xi =\frac{xy - (1-t)z^2}{t}$. We now consider a fibration given by $F =
\frac{z_2z_3}{z_1^3}$ from
$\CP(1,1,1,2)$ (minus two lines) to $\CP^1$, coinciding with $f_0$ on $X_0
\cong \CP(1:1:4)$. We restrict $F$ to $X_t$, for $t > 0$, obtaining:

$$ 
f_t : X_t \setminus \{(1:0:0:0)\} \simeq \CP^2 \setminus
\{(1:0:0)\} \rightarrow \CP^1 
$$

\begin{equation}
f_t(x:y:z) =   \frac{z\xi}{y^3} 
\end{equation}

Also consider the divisor $D = f_t^{-1}(c)$, where we take $c$ to be a positive
real number, thought of as a smoothing of $f_t^{-1}(0) = \{z\xi = 0\}$. We can
define a circle action on $\CP^2 \setminus \{y = 0\}$, given, using coordinates
$(z, \xi =\frac{x - (1-t)z^2}{t})$, by $e^{i\theta}\cdot(z, \xi) =
(e^{-i\theta}z,e^{i\theta}\xi)$. This action does not extend to all of $\CP^2$,
and it does not preserve the Fubini-Study K\"ahler form. However, we can modify
the K\"ahler form to make it $S^1$-invariant in a open subset; see below. As in
Example \ref{Ex}, we can consider Lagrangian $T(1,4,25)$ type tori, built up as
the parallel transport of orbits along the circle centered at $c$ with radius $r
< c$. The other parameter of this fibration of Lagrangian tori is given by the
moment map of the circle action (with respect to the modified K\"ahler form). 
In order to make everything explicit and be able to actually compute the
Maslov index 2 holomorphic discs bounded by these tori, we will construct a K\"{a}hler
form $\omega$, for which the moment map is given by:  

\begin{equation}
  \mu_{\omega}(x:1:z) = 2\frac{|z|^2 - |\xi|^2}{1 + |z|^2 + |\xi|^2},
\end{equation}
on an open set contained in the inverse image with respect to $f_t$ of an open disc of radius $R > 2c$
centered at $0$.

For that, on the region described above, we take $\omega$ to be equal to
$\frac{i}{4}\del\bar{\del}\log( 1 + |z|^2 + |\xi|^2)$, in the coordinate chart
$y=1$. In homogeneous coordinates this form is given by

\begin{equation} 
\tilde{\omega} =
\frac{i}{4}\del\bar{\del}\text{log}\left( 1 +\left|\frac{z}{y}\right|^2 +
\left|\frac{\xi}{y^2}\right|^2\right) = \frac{i}{4}\del\bar{\del}\text{log}( |y|^4 +
|z|^2|y|^2 + |\xi|^2) 
\end{equation}

The second expression is well defined on $\CP^2 \setminus (1:0:0)$ , and equal
to the first one since $\del\bar{\del}\log( |y|^4) = 0$. A calculation in the
affine chart $x =1$ shows that, along the complex line $y = 0$, it becomes
$dy\wedge d\bar{y}/ \left|\frac{(1 - t)z}{t}\right|^2$. So we see that
$\tilde{\omega}$ is well defined and nondegenerate away from $y=0$, but it is
degenerate along the line $y = 0$, and also is singular at $(1:0:0)$. In order
to define a nearby symplectic form, set $\rho = \left|\frac{z\xi}{y^3}\right|$
and $\ell = \frac{\left|\frac{z}{y}\right|^2 - \left|\frac{\xi}{y^2}\right|^2}{1
+ \left|\frac{z}{y}\right|^2 + \left|\frac{\xi}{y^2}\right|^2}$ for $y \neq 0$,
and consider a cut off function $\eta$ that is zero for $(x:y:z) \in \{\rho < R;
|\ell| < \lambda_0\}$ and one for $(x:y:z) \in \CP^2 \setminus \{\rho \leq 2R;
|\ell| \leq 2\lambda_0\}$. The parameters $R$ and $\lambda_0$ are chosen so that
$c + r < 2c < R$ and $0 < \lambda_0 < 1/2$, this way $\CP^2 \setminus \{\rho
\leq 2R; |\ell| \leq 2\lambda_0\}$ is an open neighborhood of $\{y = 0\}$.
Define

\begin{equation}\label{sympform}
 \omega = \frac{i}{4}\del\bar{\del}\log( |y|^4 + |z|^2|y|^2 + |\xi|^2 + 
 s^2\eta(\rho, \ell)(|x|^2+|y|^2+|z|^2)^2)
\end{equation} 
where $s$ is a very small constant. We see that $\omega$ is well defined
 in the whole $\CP^2$ since it is an interpolation between
$\tilde{\omega}$ and the K\"ahler form $\omega_s =
\frac{i}{4}\del\bar{\del}log(|y|^4 + |z|^2|y|^2 + |\xi|^2 +
s^2(|x|^2+|y|^2+|z|^2)^2)$ which is 1/2 of the pullback of the Fubini-Study form
on $\CP^{11}$ via the embedding,

\begin{eqnarray}
  \iota : \CP^2 &\longrightarrow& \CP^{11}     \nonumber \\
   (x: y :z) &\mapsto& (y^2 : zy : \xi : sx^2 : sy^2 : sz^2 : sxy: sxy : syz : syz : szx : szx)  \nonumber
\end{eqnarray}

\begin{proposition}\label{Sform} 
  For $s >0$ sufficiently small, keeping fixed
the other parameters $c$, $r$, $t$, $0 < \lambda_0 < 1/2$ and $R$, $\omega$ is a well defined nondegenerate
K\"ahler form. Moreover, $\omega$ lies in the same cohomology class as the
Fubini-Study form $\omega_{FS}$. 
\end{proposition}

\begin{proof} 
We note that for $y \neq 0$, 
$$\omega = \frac{i}{4}\del\bar{\del}
log( (1- \eta) \varphi_1 + \eta \varphi_2) = \frac{1}{2}dd^clog( (1- \eta)
\varphi_1 + \eta \varphi_2),
$$ 
where 
$$
\varphi_1 = \frac{|y|^4 + |z|^2|y|^2 +
|\xi|^2}{|y|^4},
$$
 $$
 \varphi_2 = \frac{|y|^4 + |z|^2|y|^2 + |\xi|^2 +
s^2(|x|^2+|y|^2+|z|^2)^2}{|y|^4},
$$ 
and on a neighborhood of $y =0$, $\omega$ is equal to $\omega_s$, hence it is
K\"ahler. 

We already know $\omega$ is nondegenerate at $\{\rho < R; |\ell| < \lambda_0\}$
 and $\CP^2 \setminus \{\rho \leq 2R; |\ell| \leq 2\lambda_0\}$. Since $\omega_s$ converges to $\tilde{\omega}$
uniformly on the compact
set $\{\rho \leq 2R; |\ell| \leq 2\lambda_0\}$ (where $\tilde{\omega}$ is nondegenerate) 
as $s \to 0$,
there is a small enough $s$ making $\omega$ nondegenerate. 

To determine the cohomology class of $\omega$, it is enough to compute
$\int_{[\CP^1]} \omega$. Considering $[\CP^1] = \{y = 0\}$ we see that
$\int_{[\CP^1]} \omega = \int_{[\CP^1]} \omega_s$ and so $[\omega] = [\omega_s] =
\frac{1}{2}\iota^*[\omega_{\CP^{11}}] = [\omega_{FS}]$. 

\end{proof}

The space of K\"ahler forms in the same cohomology class is
connected. Hence, by Moser's theorem, $(\CP^2, \omega)$ and $(\CP^2,
\omega_{FS})$ are symplectomorphic. After applying such a symplectomorphism, we
get Lagrangian tori in $(\CP^2, \omega_{FS})$ with the same properties as the ones we consider
in $(\CP^2, \omega)$. 

The constants $c$, $r$, $t$, $\lambda_0$, $R$ and $s$ are chosen in this order.
For what follows $c$, $0 < \lambda_0 < 1/2$ and $R > 2c$ are fixed, $r < c$ and $t$ is
thought to be very small with respect to $c$ and $r$. Considering the symplectic
form $\omega$ from now on, we define the following Lagrangian tori:

\begin{definition}
For $\lambda \in \mathbb{R}$, $|\lambda| < \lambda_0$:
\begin{equation}
T^c_{r,\lambda} = \left\{ (x:y:z) \in \CP^2 ; \left|\frac{z\xi}{y^3} - c \right| = r, 
 \ \frac{\mu_{\omega}}{2} = \frac{\left|\frac{z}{y}\right|^2 
- \left|\frac{\xi}{y^2}\right|^2}{1 + \left|\frac{z}{y}\right|^2 
+ \left|\frac{\xi}{y^2}\right|^2} = \lambda \right\}.  
\end{equation}

\end{definition}

For the sake of using Lemma \ref{MI}, which gives a convenient formula for
computing Maslov index for {\it special} Lagrangian submanifolds, we now
consider the meromorphic 2-form on $\CP^2$ which is the quotient of
$\Omega_{\C^3} = \frac{dx\wedge dy \wedge dz}{t(\xi z - cy^3)}$ defined on
$\C^3$ and has poles on the divisor $D$. On the complement of $\{y = 0\}$,
taking $y=1$, it is given by

\begin{equation} 
 \Omega = \frac{dx\wedge dz}{t(\xi z - c)} = \frac{d\xi \wedge
dz}{\xi z - c}.
\end{equation} 
Here $\xi =\frac{x - (1-t)z^2}{t}$. 

\begin{proposition} 
For $c$, $r$ and $\lambda$ as above, the tori described in the $(\xi , z)$ coordinate chart by
$T^c_{r,\lambda} = \{(\xi,z) ; |\xi z - c| = r ; |z|^2 - |\xi|^2 = \lambda(1 + |z|^2 + |\xi|^2) \}$
are special Lagrangian with respect to $\Omega$. 
\end{proposition}

\begin{proof} 
 Take $V_H$ the Hamiltonian vector field of the Hamiltonian $H(\xi,z)
= |\xi z - c|^2$, i.e., defined via $\omega(V_H, \cdot) = dH$. Since $H$ is
constant on the Lagrangian $T^c_{r,\lambda}$ and on each symplectic fiber of
$f(\xi,z) = \xi z$, $V_H$ is symplectically orthogonal to both, hence tangent to 
the Lagrangian $T^c_{r,\lambda}$ and not tangent to the symplectic fibers of $f$. 
Consider the vector field $\vartheta = (i\xi , -iz)$, tangent to the fibers and the Lagrangian
torus, as they intersect along circles of the form $(e^{i\theta}\xi_0,
e^{-i\theta}z_0)$. So, $\{\vartheta, V_H\}$ form a basis for the tangent space of
$T^c_{r,\lambda}$. Now note that 
$$ 
\iota_\vartheta \Omega = \frac{i\xi dz +
izd\xi}{\xi z - c} = id\text{log}(\xi z - c) . 
$$ 
Therefore, 
$$ 
Im(\Omega)(\vartheta, V_H) = d\text{log}|\xi z - c|(V_H) = 0 .
$$ 
\end{proof}

\section{Computing holomorphic discs in $\mathbb{CP}^2$ bounded by $T^c_{r, 0}$}\label{CHD}

In this chapter we focus on the case $\lambda = 0$ and show that, at least for
small enough $t$ with respect to $r$ and $c$, it bounds the expected 10
different families of Maslov index 2 holomorphic discs (with the expected
multiplicity modulo signs). We often use the coordinates $(z_0:z_1:z_2:z_3)$, but
restricted to $\CP^2 \simeq X_t$ via the embedding \eqref{embCP2}.

\subsection{The homology classes}\label{HomClass}

We omit the subindex $t$ and consider
$f(x:y:z) =  \frac{z\xi}{y^3} $ mapping $\CP^2$ minus $(1:0:0)$ to $\CP^1$. 

\begin{proposition} \label{bdisc}
 There is only one family of holomorphic discs, up to
reparametrization, in $\CP^2$ with boundary on $T^c_{r,0}$ $(r < c)$ that is mapped
injectively to the disc $|w -c| \leq r$ by $f$, where $w$ is the coordinate in
$\C$. 
\end{proposition}

\begin{proof}
  
Let $u: \mathbb{D} \rightarrow \CP^2 $ be such a disc so, up to
reparametrization, $f\circ u (w)= \Psi (w) = rw + c$. The map $u$ can be
described using coordinates $y = z_1 (w) = 1$, $z = z_2(w)$ and $\xi = z_3(w)$,
so $z_2 (w) z_3(w) = \Psi(w)$. Since zero (and infinity) does not belong to the
image of $\Psi$ as $r < c$, $z_2$ and $z_3$ have no zeros or poles on the disc
(note that if $z(w) = 0$ and $\xi(w) = \infty$ then $x(w) = \infty$,
contradicting $y(w) \neq 0$, for the same reason $z(w) \neq \infty$). At the
boundary of the disc, mapped by $u$ to $T^c_{r,0}$, $|z_2|=|z_3|$. So, by
holomorphicity, $z_2(w) = e^{i\theta}z_3(w) =
e^{i\frac{\theta}{2}}\sqrt{\Psi(w)}$, for some choice of square root and some
constant $\theta$. 

\end{proof}

Call $\beta$ the relative homotopy class of the above family of discs, $\alpha$ the class
of the Lefschetz thimble associated to the critical point of $f$ at the origin
lying above the segment $[0, c -r]$ (oriented to intersect positively $\{z =
0\}$) and $H = [ \CP^1 ]$ the image of the generator of $\pi_2(\CP^2)$ in
$\pi_2(\CP^2, T^c_{r,0})$. One checks that $\alpha$, $\beta$, $H$ form a basis
of $\pi_2(\CP^2, T^c_{r,0})$. On Figure \ref{fig10fam}, holomorphic 
discs in the class $\beta$ are represented by the tropical disc arriving at the torus
in the direction $(1,0)$ and associated to the term $u$ of the superpotential
$W_{T(1,4,25)}$.
If we consider $\lambda > 0$, and increase $r \to c$, we
see that the torus depicted on Figure \ref{fig10fam} approaches the 
wall, and  the class $\alpha$ corresponds to a tropical Maslov index $0$ disc
that runs along the wall and ends at the node. 

In order to understand what relative homotopy classes are allowed to have Maslov index 2
holomorphic discs, we analyze their intersection with some other complex curves,
for instance the line over $\infty$, $y = 0$, the line and conic over $0$, $z =
0$ and $D_2: xy - (1-t)z^2 = 0$ . Another curve we use is a quintic, $D_5$ that
converges to $(\x\z - c\y^5)^2 = 0$ on $\CP(1,1,4)$. It is given by 
$$ 
D_5 :
z_0z_3^2 - 2cz_1^2z_2z_3 + c^2 z_1^5 = x\xi^2 - 2cy^2z\xi + c^2y^5 = 0
$$

For $z \neq 0 $ and setting $ f = \frac{z\xi}{y^3}$, we can write this equation as

\begin{equation} \label{eqD5}
y^5(c^2 - 2cf + \frac{xy}{z^2}f^2) = 0
\end{equation}

\begin{remark} Again relating to Figure \ref{fig10fam}, the leftmost node is
thought to be the torus $T^c_{c,0}$ and the wall in the direction $(1,-1)$ to be
formed by the tori $T^c_{c,\lambda}$. $D_2 = \{\xi = 0\}$ projects to the upper
part of the wall, while $\{z = 0\}$ projects to the lower part of the same wall,
as the Maslov index 0 discs bounded by $T^c_{c,\lambda}$ are contained in these
two divisors. Using similar reasoning, one expects that over the rightmost wall
lie complex curves converging to $\{\y(\x\z - c\y^5) = 0\}$, the boundary
divisor of the ``orbifold almost toric fibration" on $\CP(1,1,4)$. $D_5$
converges to $\{(\x\z - c\y^5)^2 = 0\}$, hence it is thought to lie over the
lower part of the rightmost wall, while $\{y = 0\}$, which converges to $\{\y^2
= 0\}$, is thought to lie over the upper part of the same wall. \end{remark}

\begin{remark} 
  
To compute intersection number with $H$ in Figure \ref{fig10fam}, one can use
the tropical rational curve formed by the union of all the tropical discs
depicted on Figure \ref{figClif1}. Also, note that each wall hits the cut of the
other, leaving them with slopes $(-5,2)$ on the top left and $(7,2)$ on the top
right; the additional dotted lines are omitted on all pictures for simplicity.

\end{remark}

\begin{lemma}\label{HomTab}
For fixed $c$ and $r < c$, and for $t$ sufficiently small, the classes $\alpha$, $\beta$ and $H$ 
intersect the varieties $\{z = 0\}$, $\{y =0\}$, $D_3 = f^{-1}(c) \cup \{(1:0:0)\}$, 
$D_2: xy - (1-t)z^2 = 0$, $D_5$ and have Maslov index, according to the table below. 
\begin{center}
  \begin{tabular}{  | c || c | c | c | c | c || c |}
  
\hline
    Class & $z = 0$ & $y = 0$ & $D_3$ & $D_2$ & $D_5$ & Maslov index $\mu$  \\ 
   \hline 
    $\alpha$ & 1 & 0 & 0 & -1 & 0 & 0 \\ \hline
    $\beta$   & 0 & 0 & 1 &  0  & 2 & 2  \\   \hline
      $H$       & 1 & 1 & 3 &  2  & 5 & 6  \\   \hline
  \end{tabular}
\end{center}
\end{lemma}

\begin{proof}
In order to use these curves to compute intersection numbers, we first need to ensure 
they don't intersect $T^c_{r,0}$. This is clear for $z = 0$, $y =0$, $D_3$, and $D_2$. 
Later we will see that $T^c_{r,0}\cap D_5 = \emptyset$.

The intersections with $H$ follow from Bezout's Theorem. By construction, $\alpha$ 
(represented by the Lefschetz thimble over the segment $[0, c- r]$ which can be
parametrized by $y = 1$, $z = \rho e^{i\theta}$, $\xi = \rho e^{-i\theta}$,
$\rho \in [0, c- r]$) does not intersect
$y = 0$ and $D_3$, also, it intersects both
$z = 0$ and $D_2$ at one point with multiplicity $1$ and $-1$, respectively. 
Each disc computed in Proposition \ref{bdisc} representing the class 
$\beta$ does not intersect $z = 0$, $y =0$, and $D_2$, and intersect $D_3$ positively
at one point. 

It remains for us to understand the intersection of $D_5$ with the torus $T^c_{r,0}$, $\alpha$ and
$\beta$. For that, we look at the family of conics $\mathcal{C} = \{z = 
e^{i\theta}\xi ; \ 
\theta \in [0, 2\pi]\}$ containing $T^c_{r,0}$, the thimble representing the class $\alpha$ 
and the discs representing the class $\beta$ computed in Proposition \ref{bdisc}.

On $\mathcal{C}$, using the coordinate chart $y =1$, we have: 
\begin{equation}\label{OnC1}
  z = e^{i\theta}\xi = e^{i\theta}\frac{x - (1-t)z^2}{t}.
\end{equation}

So, solving for $x$ in \eqref{OnC1}, and using $f = z\xi 
= e^{-i\theta}z^2,$ we get: 
\begin{equation}\label{OnC2}
  x = te^{-i\frac{\theta}{2}}f^{1/2} + (1-t)e^{i\theta}f,
\end{equation}
 
for some square root of $f$. Then, by \eqref{eqD5} and \eqref{OnC2}, the points of $D_5 \cap
\mathcal{C}$ are those where

\begin{equation*}
c^2 - 2cf + e^{-i\theta}f(te^{i\frac{-\theta}{2}}f^{1/2} + (1-t)e^{i\theta}f) 
= (f - c)^2 + tf^{3/2}(e^{-i\frac{3\theta}{2}} - f^{1/2}) = 0
\end{equation*}

For $t$ small enough, for each value of $\theta$, all the
solutions of this equation lies in the region $|f - c| < r$. From this we can
conclude that $D_5 \cap T^c_{r,0} = \emptyset$, $D_5 \cap \alpha = \emptyset$,
since, in $T^c_{r,0}$, $|f - c| = r$ and the thimble representing $\alpha$ lies
over $[0, c -r]$. 

Now, a holomorphic disc representing the class $\beta$ is given by $z
= \xi = f^{1/2}$ and $Re(z) > 0$; see Proposition \ref{bdisc}. This means that this disc
intersects $D_5$ in exactly two points, namely the two solutions of 
$(z^2 - c)^2 + tz^{3}(e^{-i\frac{3\theta}{2}} - z) = 0$, where
 $z$ is close to $\sqrt{c}$. As both are complex curves, the
intersections count positively, so the intersection number between $D_5$ and
$\beta$ is equal to 2.

Finally, from lemma \ref{MI}, we see that the Maslov index is twice the
intersection with the divisor $D_3$.

 
\end{proof}

\begin{lemma} \label{lemHomClass}
The only classes in $\pi_2(\CP^2, T_{r,0})$ which may contain holomorphic discs of 
Maslov index 2 are $\beta$, $H - 2\beta + m\alpha$, $-1 \leq m \leq 2$ and 
$2H - 5\beta + k\alpha$, $ -2\leq k \leq 4$.
\end{lemma}

\begin{proof}
To have Maslov index $2$ the class must have the form $\beta + l(H - 3\beta) + k\alpha$. 
Considering positivity of intersections with $y = 0$ we get $l \geq 0$, with $z = 0$ and 
$D_2$ we get $-l \leq k \leq 2l$, and finally with $D_5$,  $l \leq 2$.
\end{proof}

\subsection{Discs in classes $H - 2\beta + m\alpha$}
\begin{theorem}\label{3fam}
There are no Maslov index 2 holomorphic discs in the class $H - 2\beta -\alpha$;
there are one-parameter families of holomorphic discs in the classes $H -
2\beta$ and $H - 2\beta + 2\alpha$, with algebraic counts equal to 2 up to sign in both cases, and
a one-parameter family of holomorphic discs in the class $H - 2\beta + \alpha$,
with algebraic count equal to 4 up to sign. 
\end{theorem}

This is precisely what we expect from the term $2\frac{e^{-\Lambda}}{u^2}(1 + w
)^2 $ in $W_{T(1,4,25)}$; see equation \eqref{WChe2}.

\begin{proof}
We will try to find holomorphic discs $u : (\D, S^1) \rightarrow (\CP^2, T^c_{0,r})$ 
in the class $H - 2\beta + m\alpha$, $-1 \leq m \leq 2$. 
Recall $f : \CP^2 \setminus (0: 0 :1) \rightarrow \CP^1$, $f(x:y:z) = \frac{z\xi}{y^3}$, 
and set $\Psi = f \circ u : \D \rightarrow \CP^1$.
 Since $u$ has Maslov index 2 it doesn't go through $(1:0:0)$, where $D_3$ has a 
 self intersection, so $\Psi$ is well-defined. 

We look at $\frac{\Psi(w) - c}{r}$, which maps the unit circle to the unit
circle. Looking at the intersection numbers given in Lemma \ref{HomTab}, we see
that our disc must intersect $y = 0$ and the divisor $D_3$ at 1 point.
Therefore, as $\Psi = f \circ u$, $f^{-1}(\infty) = \{y^3 = 0\}$ and $D_3 = f^{-1}(c) 
\cup (1:0:0)$, the map  
$\frac{\Psi(w) - c}{r}$ has a pole of order 3 and a simple zero, so

\begin{equation} \label{Psi-c}
\frac{\Psi(w) - c}{r} = \frac{\tau_{w_0}(w)}{\tau_{w_1}^3(w)} e^{i\phi} , \ \ \ \ \text{where} 
\ \ \ \  \tau_\upsilon(w) = \frac{w - \upsilon}{1 - \bar{\upsilon}w},
\end{equation}

for some $w_0$, $w_1$ in $\D$ and $e^{i\phi} \in S^1$. We can use automorphisms
of the disc to assume $w_1 = 0$, $\phi = 0$ and write $w_0 = a$. Note that the
disc automorphism $w \mapsto e^{i\phi'}w$ amounts to $w_0 \mapsto e^{-i\phi'}w_0$ and
$ \phi \mapsto \phi - 2\phi'$ in \eqref{Psi-c}. So, $\phi' = \pi$
keeps $e^{i\phi}$ invariant, therefore we need to keep in mind that $\pm a$ gives the same
holomorphic disc modulo reparametrization.\\

\begin{figure}[htb]
\begin{center}
\scalebox{0.5}{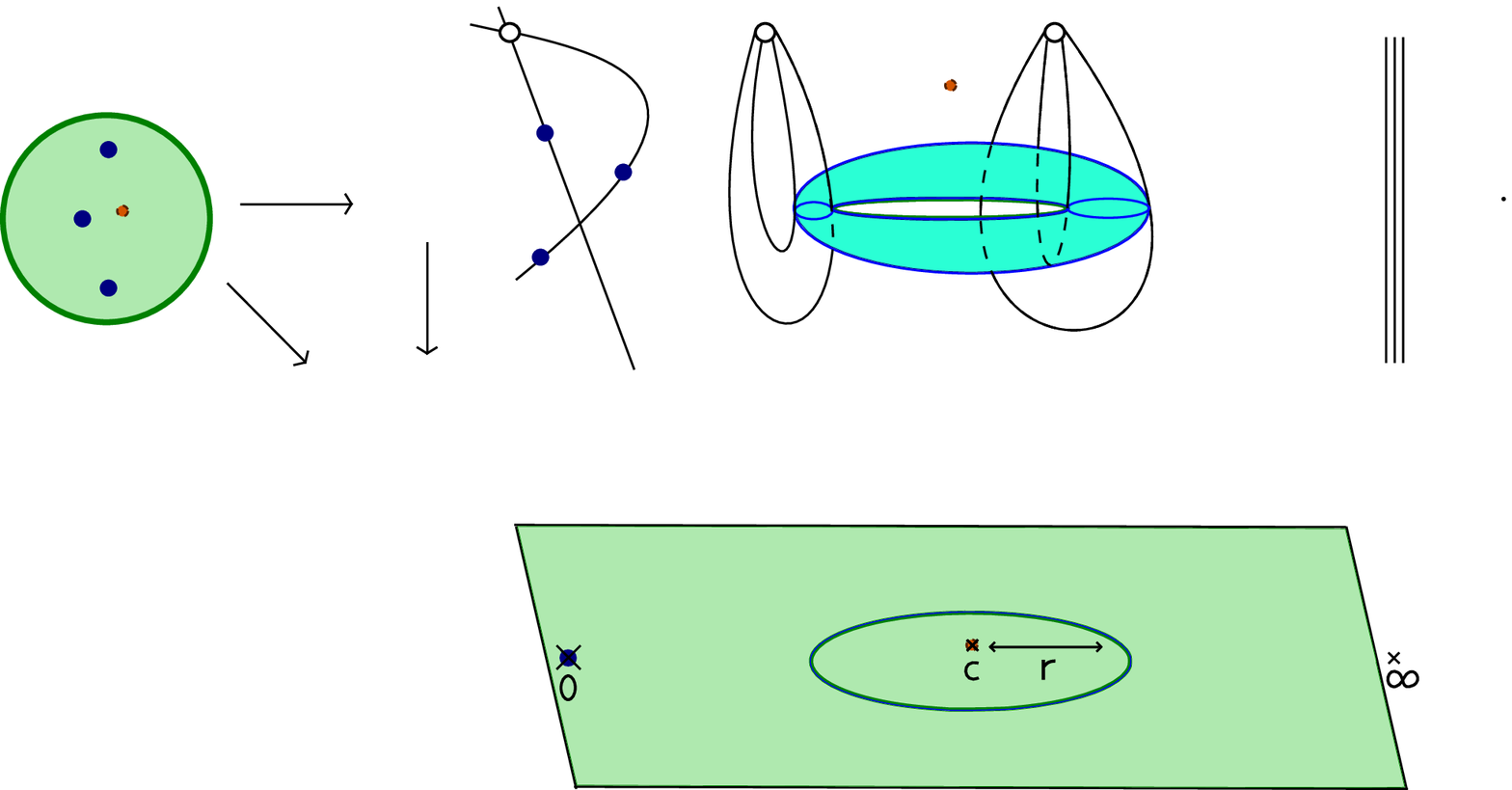}
 \end{center} 
\caption{Picture of $\Psi(w) = f \circ u (w) = \frac{z(w)\xi(w)}{y^3(w)} = \frac{z_2(w)z_3(w)}{z_1^3(w)}$
 (for the case $m = 0$; $|I| = 2$, i.e., $u$ intersects
the conic $D_2$ twice). $\{\eta_0, \eta_1, \eta_2\} = \Psi^{-1}(0)$ and $\{a\} = \Psi^{-1}(c)$
 behave as in Lemma \ref{h0a} for small $t$. Recall that the divisor $D_3$ 
is the closure of $f^{-1}(c) \ni u(a)$.} 
\label{fig:Thm3fam}
\end{figure} 

Since $r < c$, the image of $u$ intersects $f^{-1}(0)\subset D_2 \cup \{z = 0\}$ 
in three points $u(\eta_j)$'s,
j = 0, 1, 2, i.e., $\eta_j \in \D$ are so that $\Psi(\eta_j) = 0$. 
The integer $m$ in $H - 2\beta + m\alpha$ determines how many times the disc $u$ intersects
$D_2$, and we consider a set $I \subset \{0, 1, 2\}$ with that number of
elements. Writing $z_1 = y$, $z_2 = z$ and $z_3 = \xi$, and $\tau_j =
\tau_{\eta_j} $ we see that the map $u$ can be expressed in the form

\begin{equation}\label{discform}
 z_1(w) = w \ , \  z_2(w) = e^{-i\theta} h(w)\prod_{j \notin I}\tau_j(w)  \ , \ z_3(w) 
 = e^{i\theta} h(w)\prod_{j \in I}\tau_j(w), 
\end{equation}
where $h(w)$ is a nonvanishing holomorphic functions and
$e^{i\theta} \in S^1$. 

\begin{remark}
A suitable scaling of the homogeneous coordinates
eliminates the need for a multiplicative factor in the expression for
$z_1(w)$. In principle we know that $z_2(w) = e^{-i\theta} h_2(w)\prod_{j \notin I}\tau_j(w)$, $ z_3(w) 
 = e^{i\theta} h_3(w)\prod_{j \in I}\tau_j(w)$. But we see that on $\del \D$, $\left|\frac{z_2}{z_1}\right| =
\left|\frac{z_3}{z_1^2}\right|$ and $|z_1| = 1$, so $|h_2| = |h_3|$ on the unit
circle, therefore $h_3 = e^{i\theta'}h_2$ for some constant $\theta'$. Note that
we can absorb $-\theta' /2$ in $\theta$ and assume that $h_3 = h_2 = h$.
\end{remark}
 
 Since $\Psi(w) = \frac{z_2(w)z_3(w)}{z_1^3(w)} =
 \frac{h^2(w)\prod\tau_j(w)}{w^3}$, we get $h(w) =
 \left(\frac{\Psi(w)w^3}{\prod\tau_j(w)}\right)^{1/2}$, for some choice of
 square root. The other choice is equivalent to a translation by $\pi$ of the
 parameter $\theta$. Rewriting this last equation and using 
 $\frac{\Psi(w) - c}{r} = \frac{\tau_{a}(w)}{w^3}$, we get:

\begin{equation} \label{PrePol}
 w^3\Psi(w) = r\tau_a(w) + cw^3 = h^2(w)\tau_0(w)\tau_1(w)\tau_2(w).
\end{equation}

We expect one parameter family(ies) of solutions and we see that $u$ is
determined by the parameters $a$ and $\theta$. Therefore we want to understand
how many possible choices for $a$ there are, for any given $\theta$. Moreover,
understanding how these solutions vary with $\theta$, we can describe the moduli
space of holomorphic discs in the class $H - 2\beta + m\alpha$ bounded by
$T^c_{r,0}$, denoted by $\mathcal{M}(T^c_{r,0}, H - 2\beta + m\alpha)$, for each
$m$. 

The possible values of a are constrained by the following equation coming from the fact that
$\xi = -\frac{(1 - t)}{t} z^2$ when $y =0$:

\begin{equation}\label{cond}
\frac{z_3(0)}{z^2_2(0)} = -\frac{1-t}{t} = \frac{e^{3i\theta} \prod_{j \in I}\tau_j(0)}
{ h(0)\prod_{j \notin I}\tau^2_j(0)} = 
\frac{ e^{3i\theta}(-1)^{|I|}\prod_{j \in I}\eta_j}{ h(0)\prod_{j \notin I}\eta^2_j}. 
\end{equation}

Since \eqref{PrePol} implies that $-ar = -h^2(0)\eta_0\eta_1\eta_2$, \eqref{cond} can be rewritten as
\begin{equation}\label{cond2}
ar\left(\frac{1-t}{t}\right)^2\prod_{j \notin I}\eta^3_j = e^{6i\theta} \prod_{j \in 
I}\eta^3_j.
\end{equation}

Note that solving \eqref{cond2} for $a$ amounts to a solution of \eqref{cond}
for some choice of square root for $h(w) =
\left(\frac{\Psi(w)w^3}{\prod\tau_j(w)}\right)^{1/2}$. 

Understanding the behavior of the parameters $\eta_j$ and $a$ as $t \to 0$
will allow us to analyze the existence of $a$ solving this equation for small
values of $t$. Note that the right side of \eqref{cond2} is uniformly bounded for
all $t$. So we conclude that $a\prod_{j \notin I}\eta_j \to 0$, as $t \to 0$. 
Moreover, if we can show that $h(0)$ is bounded away
from zero, then we can conclude that $\eta_j \to 0$ for some $j \notin I$;
see \eqref{cond}.  

\begin{lemma} \label{h0a}
  $h(0)$ is bounded away from 0 and, after possibly relabeling the $\eta_j$'s,
  the following asymptotic hold as $t \to 0$: $$a = \ot, \ \eta_0
  = \ot, \ \eta_1 \to
  \sqrt{\frac{r}{c}}i, \ \eta_2 \to
  -\sqrt{\frac{r}{c}}i.$$ 
  
  Moreover, $0 \notin I$. Therefore, as $|I| < 3$ represents the number of 
  intersection with $D_2$, there is no holomorphic disc
  in the class $H - 2\beta -\alpha$ (i.e. for $m = -1$). 
  \end{lemma}

\begin{proof}
 Consider the polynomial: 

$$ 
\Xi(w) = w^3(1 - \bar{a}w) \Psi(w) = r(w - a) + cw^3(1 - \bar{a}w) = 
-c\bar{a}(w - \zeta)(w - \eta_0)(w - \eta_1)(w - \eta_2), 
$$

for some $\zeta$, with  $|\zeta| > 1$. 
Assume $|\eta_0| \leq |\eta_1|$, $|\eta_0|  \leq |\eta_2| $, and write

$$(w - \eta_0)(w - \eta_1)(w - \eta_2) = w^3 -\sigma w^2 + qw - p$$\\
 
 By comparing coefficients,  we get

\begin{equation}\label{eq1}
p\zeta = \frac{ra}{c\bar{a}}
\end{equation}



\begin{equation}\label{eq2}
1 = \bar{a}(\zeta + \sigma)
\end{equation}

By \eqref{PrePol}, 

$$h^2(0) = \frac{ar}{\eta_0\eta_1\eta_2} = \frac{ar}{p}.$$

By 
equations \eqref{eq1}, \eqref{eq2} and noting that $|\sigma| \leq 3 \leq 3|\zeta|$, 

$$1 =
|a||\zeta + \sigma| \leq 4|a\zeta| = \frac{4|a|r}{c|p|} = \frac{4|h^2(0)|}{c}.$$ 

So, we get $|h(0)| \geq
\frac{\sqrt{c}}{2}$. Looking at equation \eqref{cond}, we get that at least one
$\eta_j$ must be in the denominator. More precisely $\prod_{j \notin I}\eta_j =
O(t^{1/2})$. As the other $\eta_j$'s lie in the unit
disc, $p = O(t^{1/2}) \to 0$ and by \eqref{eq1}, $\zeta \mapsto \infty$.
Also by \eqref{eq2} , as $\sigma$ is bounded, we get that $a \mapsto 0$, in fact
$a = O(t^{1/2})$. 

Therefore,  

$$ \Xi(w) = w(cw^2 + r) + w^4O(t^{1/2}) + O(t^{1/2}), $$ 
and we see that $\eta_0
= O(t^{1/2})$ and $\eta_1\eta_2 \mapsto \frac{r}{c}$, say $\eta_1 \mapsto +
\sqrt{\frac{r}{c}}i$, $ \eta_2 \mapsto - \sqrt{\frac{r}{c}}i$. In particular, we
conclude that $0 \notin I$. Also, $p\zeta = \zeta\eta_0\eta_1\eta_2 =
\zeta\eta_0(\frac{r}{c} + O(t^{1/2})) = \frac{ra}{c\bar{a}}$, hence $\zeta\eta_0
= \frac{a}{\bar{a}} + O(t^{1/2})$. Note that since $|I| < 3$, there are no
holomorphic discs for $m = -1$, and this finishes the proof of Lemma \ref{h0a}.
 
\end{proof}

Now we need to analyze the cases $I = \emptyset , \{1\}, \{2\}, \{1,2\}$.\\

\textbf{Case $I = \emptyset$, $m = 2$:}\\

By (\ref{cond2}), (\ref{eq1}): 

\begin{equation}\label{eq001}
(\bar{a}\zeta)^3 = a^4\left(\frac{1-t}{t}\right)^2\frac{e^{-6\theta i}r^4}{c^3}  = a^4K
\end{equation}

where  $K =\left(\frac{1-t}{t}\right)^2\frac{e^{-6\theta i}r^4}{c^3}$. 

\begin{proposition}\label{m=2}
  For small enough $t > 0$, equation \eqref{eq001} has four solutions for each 
  given parameter $\theta$. Moreover, naming  these solutions 
  $a_1(\theta)$, $a_2(\theta)$, $a_3(\theta)$, $a_4(\theta)$, as we vary continuously with 
  $\theta$, in counter-clockwise order, we have that $a_j(\theta + \pi/3) = a_{j+1}(\theta)$.
\end{proposition}

\begin{proof}
By \eqref{eq2}

\begin{equation} \label{eqaz1}
(\bar{a}\zeta)^3 = 1 + O(t^{1/2})
\end{equation}

Combining \eqref{eq001} with \eqref{eqaz1} we see that for $g(a) = 1 - (\bar{a}\zeta)^3$

\begin{equation}
a^4K - 1 + g(a) = 0
\end{equation}

One sees that, for sufficiently small $t$, there are 4 solutions of such
equation since $g(a)$ and $g'(a)$ are $\ot$. (To see that $g'(a) = \ot$, we use
that $g(a) = \tilde{g}(a , \bar{a}) $, where $\tilde{g}(a, b)$ is a holomorphic
function, and using Cauchy's differentiation formula, we get that each partial
derivative is $\ot$.)

So for each $\theta$ there are four solutions for $a$, each of them close to a
fourth root of $K^{-1}$. Recall we named these solutions, as varying continuously with
$\theta$, in counter-clockwise order,
$a_1(\theta)$, $a_2(\theta)$, $a_3(\theta)$, $a_4(\theta)$. We see from \eqref{eq001}
that $a_j(\theta + \pi/3) = a_{j+1}(\theta)$. 
\end{proof}

Let $u^\theta_{a_1(\theta)}$ be the holomorphic disc given by 
\eqref{discform}, for a given value of $\theta$ and the other parameters
determined by $a_1(\theta)$. 

\begin{lemma}\label{modulim=2}
  The moduli
space of holomorphic discs in the class $H - 2\beta + 2\alpha$, $\mathcal{M}(T^c_{r,0}, H -
2\beta + 2\alpha)$, can be parametrized using only
holomorphic discs $u^\theta_{a_1(\theta)}$ for $\theta \in [0,2\pi]$. Also, 
 the algebraic count of holomorphic discs in $\mathcal{M}(T^c_{r,0}, H -
2\beta + 2\alpha)$, $n_{H - 2\beta + 2\alpha}(T^c_{r,0})$, is equal to $2$ up to sign.
\end{lemma}

\begin{proof}
  Since $a_j(\theta + \pi/3) = a_{j+1}(\theta)$, we can parametrize the moduli
space of holomorphic discs in the class $H - 2\beta + 2\alpha$ using only
holomorphic discs $u^\theta_{a_1(\theta)}$ for $\theta \in [0,4\pi]$. But recall 
that solutions are counted twice as the
disc automorphism $w \mapsto -w$ amounts to $a \mapsto -a$ and $0 =\phi \mapsto
\phi - 2\pi = -2\pi$ in \eqref{Psi-c}. Hence we see that $u^{\theta+2\pi}_{a_1(\theta
+2\pi)} = u^\theta_{a_3(\theta)} = u^\theta_{-a_1(\theta)}$ is the same up to
reparametrization. Therefore the map $\theta \mapsto u^\theta_{a_1(\theta)}$,
from $S^1= [0,2\pi]/0\sim 2\pi$ to the moduli space $\mathcal{M}(T^c_{r,0}, H -
2\beta + 2\alpha)$ gives a diffeomorphism.

  In order to compute $ n_{H - 2\beta +
2\alpha}(T^c_{r,0})$ we need to look at $ev_{\ast}[\mathcal{M}(T^c_{r,0}, H -
2\beta + 2\alpha)] = n_{H - 2\beta + 2\alpha}(T^c_{r,0}) [T^c_{r,0}]$. The
boundary of each holomorphic disc lies in the class $2(\del\alpha - \del\beta)$,
and the parameter $\theta$ comes from the action $e^{i\theta}\cdot(\xi ,z) =
(e^{i\theta}\xi, e^{-i\theta}z)$, described in coordinates $(\xi ,z)$ for $y=1$,
whose orbits are in the class of the thimble, i.e., $\del\alpha$. Therefore,
$n_{H - 2\beta + 2\alpha}(T^c_{r,0}) = \pm2$.\\ 
\end{proof}

\textbf{Case} $I = \{2\}$ (similarly $ I = \{1\}$) , $m = 1$\textbf{:} \\

By \eqref{cond2}, \eqref{eq1}, setting now $K = \left(\frac{1 - t}{t}\right)^2re^{-6\theta i}$: 

\begin{equation}\label{eq101}
K\eta_0^3 = \frac{1}{a}\frac{\eta^3_2}{\eta^3_1} 
= \frac{1}{a}(-1 + O(t^{1/2}) ).
\end{equation}

Similarly to the previous case we have

\begin{proposition}\label{m=1}
  For small enough $t > 0$, equation \eqref{eq101} has four solutions for each 
  given parameter $\theta$. Moreover, naming  these solutions 
  $a_1(\theta)$, $a_2(\theta)$, $a_3(\theta)$, $a_4(\theta)$, as we vary continuously with 
  $\theta$, in counter-clockwise order, we have that $a_j(\theta + \pi/3) = a_{j+1}(\theta)$.
\end{proposition} 

\begin{lemma}\label{modulim=1}
  The moduli
space of holomorphic discs in the class $H - 2\beta + \alpha$, $\mathcal{M}(T^c_{r,0}, H -
2\beta + \alpha)$, can be parametrized using only
holomorphic discs $u^\theta_{a_1(\theta)}$ for $\theta \in [0,4\pi]$. 
Moreover, the algebraic 
count of holomorphic discs in $\mathcal{M}(T^c_{r,0}, H -
2\beta + \alpha)$, $n_{H - 2\beta + \alpha}(T^c_{r,0})$, is equal to $4$ up to 
sign.
\end{lemma}

\begin{proof} (\emph{of Proposition} \ref{m=1} \emph{and of Lemma} \ref{modulim=1})
  
Since $\zeta\eta_0 = \frac{a}{\bar{a}} + O(t^{1/2})$,  and  $a = 
O(t^{1/2})$:

\begin{equation*} 
(\bar{a}\zeta)^3 \frac{1}{a}(-1 + O(t^{1/2})) =
K(\bar{a}\zeta\eta_0)^3 = K[a(1 + O(t^{1/2}))]^3 = Ka^3(1 + O(t^{1/2})).
\end{equation*}

Using $(\bar{a}\zeta)^3 = 1 + O(t^{1/2})$, we get

\begin{equation}
Ka^4 + O(t^{1/2}) = -1.
\end{equation}
  
 Using the same argument as before we get four solutions for $a$, $a_1(\theta)$,
 $a_2(\theta)$, $a_3(\theta)$, $a_4(\theta)$, varying continuously with
 $\theta$, ordered in the counter-clockwise direction. Again $a_j(\theta +
 \pi/3) = a_{j+1}(\theta)$, but now the disc automorphism $w \mapsto -w$, not
 only switches $a \mapsto -a$ but also $\eta_1 \leftrightarrow \eta_2$, which
 accounts for the case $I = \{1\}$. Therefore the moduli space
 $\mathcal{M}(T^c_{r,0}, H - 2\beta + \alpha)$ is given by
 \{$u^\theta_{a_1(\theta)}$; $\theta \in [0,4\pi]$\}, and hence $n_{H - 2\beta +
 \alpha}(T^c_{r,0}) = \pm4$.
 \end{proof}

The case $I = \{1 , 2\}$ , $m = 0$, works in a totally analogous way, with 
$n_{H- 2\beta}(T^c_{r,0}) = \pm2$. This concludes the proof of Theorem \ref{3fam}. 
\end{proof}

\subsection{Discs in classes $2H - 5\beta + k\alpha$}

\begin{theorem}\label{6fam} 
There are no Maslov index 2 holomorphic discs in the
class $2H - 5\beta -2\alpha$, and one-parameter families of holomorphic discs in the
classes $2H - 5\beta + k\alpha$, $k = -1, 0, 1, 2, 3, 4$, with algebraic counts
equal to $1, 5 ,10 ,10 ,5 ,1$, up to sign, respectively. 
\end{theorem}

This is precisely what we expect from the term $\frac{e^{-2\Lambda}}{u^5w}(1 + w
)^5$ in $W_{T(1,4,25)}$; see equation \eqref{WChe2}.

\begin{proof} 

We start approaching the problem following the same reasoning as in the previous
subsection. But we will get an extra parameter, since the intersection with $\{y
= 0\}$ is 2 for discs in the classes $2H - 5\beta + k\alpha, -2 \leq k \leq 4$.
This will make our computation harder. Nonetheless, looking at Figure
\ref{fig10fam}, for $t \to 0$ we expect the discs in these classes (for $k \geq -1$) to converge to holomorphic
discs in $\CP(1,1,4)$ that remain away from the orbifold point, since they don't touch the
singular fiber that collapses into the singular orbifold point when $t=0$. The
idea is then to understand the limits of such discs when $t \to 0$, `count' them
for $t = 0$ and use Lemma \ref{x0xt} to show that the count remains the same for
small $t > 0$.

Consider a holomorphic map $u : (\D, S^1) \rightarrow (\CP^2,
T^c_{0,r})$ in the class $2H - 5\beta + k\alpha, -2 \leq k \leq 4$, and $\Psi(w)
= f\circ u(w)$. Analyzing intersection numbers with divisors we get

\begin{equation}\label{Psi6}
\frac{\Psi(w) - c}{r} = \frac{\tau_{w_0}(w)}{\tau_{w_1}^3(w)\tau_{w_2}^3(w)} e^{i\phi}
\end{equation}

and denote by $\eta_0, \dots, \eta_5$ the zeros of
$\Psi(w)$. Again using automorphisms of the disc we can choose $w_1 =0$ and $\phi = 0$ and
also rename $w_2 = \nu$ and $w_0 = b$. Then, the holomorphic disc can be described by

\begin{equation} \label{disc6fam}
z_1(w) = w \tau_{\nu}(w) \ , \ z_2(w) = e^{-i\theta} h(w)
\prod_{j \notin I}\tau_j(w) \ , \ z_3(w) = e^{i\theta} h(w)\prod_{j \in I}\tau_j(w)
\end{equation}

where $h(w) = \left(\frac{\Psi(w)w^3\tau_{\nu}^3(w)}{\prod\tau_j(w)}\right)^{1/2}$,
and $I \subset \{0, 1, 2, 3, 4, 5\}$. Recall that $y = z_1$, $z = z_2$, $\xi 
= z_3$ and $\tau_j = \tau_{\eta_j}$.

In the same way as in the previous section, we get a pair of equations

\begin{equation}\label{cond61}
\frac{z_3(0)}{z^2_2(0)} = -\frac{1-t}{t} = \frac{ e^{3i\theta} \prod_{j \in I}\tau_j(0)}{ h(0)\prod_{j \notin I}
\tau^2_j(0)} = \frac{ e^{3i\theta} (-1)^{|I|}\prod_{j \in I}\eta_j}{ h(0)\prod_{j \notin I}\eta^2_j} 
\end{equation}

\begin{equation} \label{cond62}
\frac{z_3(\nu)}{z^2_2(\nu)} = -\frac{1-t}{t} = \frac{ e^{3i\theta} \prod_{j \in I}\tau_j(\nu)}{ h(\nu)
\prod_{j \notin I}\tau^2_j(\nu)} = \frac{ e^{3i\theta} \prod_{j \in I}q_j}{ h(\nu)\prod_{j \notin I}q^2_j} 
\end{equation}

where we write $q_j = \tau_j(\nu)$. Again we want to understand the asymptotic 
of the parameters $b$, $\nu$ and $\eta_j$ as $t \to 0$.

\begin{figure}[htb]
\begin{center}
\scalebox{0.5}{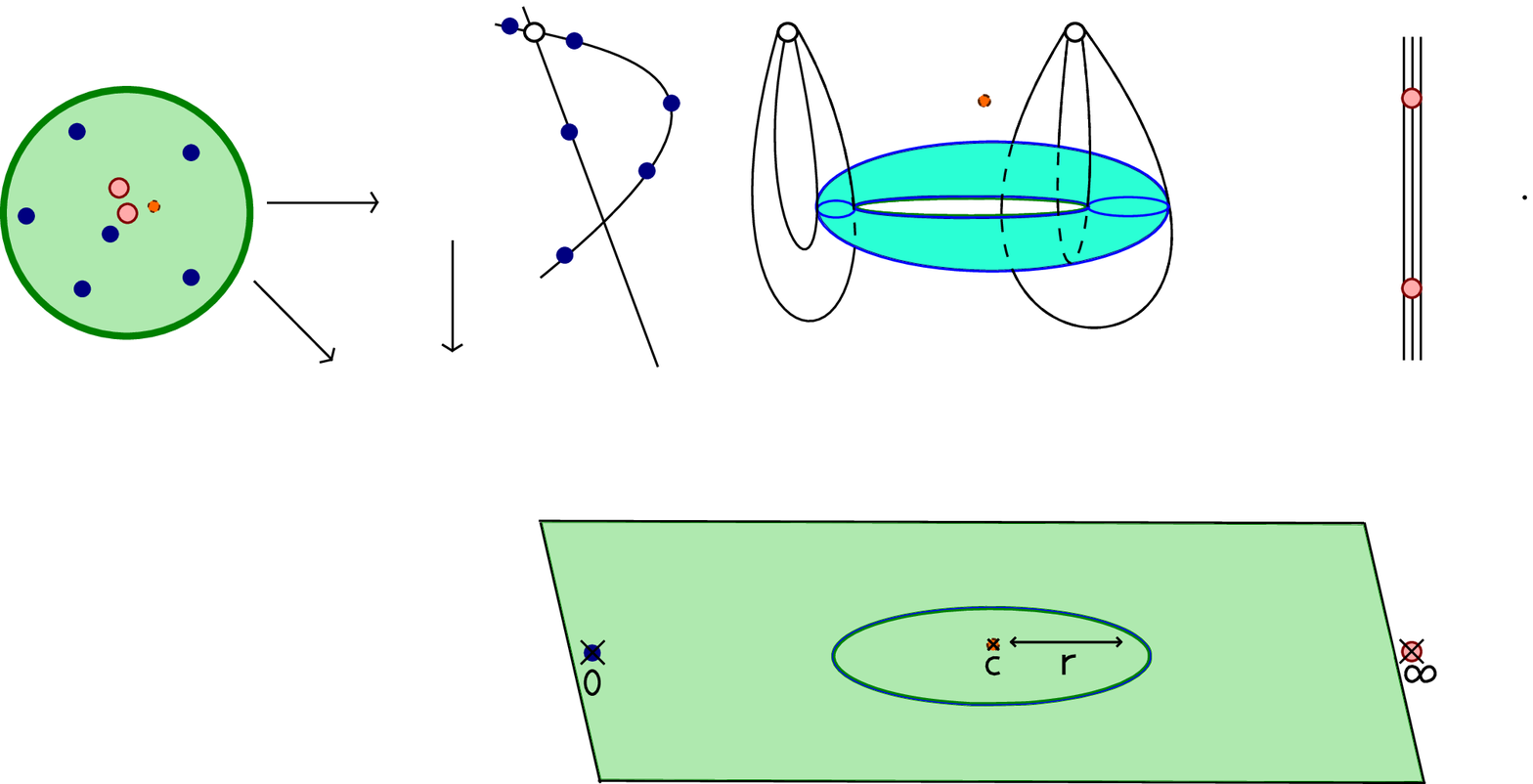}
 \end{center} 
\caption{Picture of $\Psi(w) = f \circ u(w) = \frac{z(w)\xi(w)}{y^3(w)} = \frac{z_2(w)z_3(w)}{z_1^3(w)}$
 (for the case $k = -1$; $|I| = 4$, i.e., $u$ intersects the conic
$D_2$ in five points). $\{\eta_0, \cdots, \eta_5\} = \Psi^{-1}(0)$, $\{0,\nu\}= \Psi^{-1}(\infty)$
 and $\{b\} = \Psi^{-1}(c)$
 behave as in Lemma \ref{asymp} for small $t$. Recall that the divisor $D_3$ 
is the closure of $f^{-1}(c) \ni u(b)$.} 
\label{fig:Thm6fam}
\end{figure} 

\begin{lemma} \label{asymp}
  $h(0)$ is bounded away from 0 and, after possibly relabeling the $\eta_j$'s,
  the following asymptotics hold as $t \to 0$: $$b = \ot, \ \nu = \ot, \ \eta_0
  = \ot \ \text{and for} \ j \neq 0, \ \eta_j \to
  -\frac{r}{c}^{\frac{1}{5}}e^{\frac{2\pi i}{5}j}.$$ 
  \end{lemma}

Using \eqref{cond61} and the above Lemma \ref{asymp} we see that $0 \notin I$.
Therefore there is no holomorphic disc representing the class $2H - 5\beta +
k\alpha$, for $k = -2$. To prove the existence of such discs for $-1 \leq k \leq
4$ with the right count, we look at the limit $t=0$, i.e., in $\CP(1:1:4)$.
These six families of discs are expected to `survive' in the limit and not pass
through the orbifold point of $\CP(1,1,4)$; see Figure \ref{fig10fam}. We show
that this is the case for the limits of the above families of holomorphic discs 
and, assuming regularity (proven in section \ref{reg}), we
argue that the disc counts are the same for $X_0$ and $X_t$ for a sufficiently
small $t$. Lemma \ref{asymp} will
allow us to prove:

\begin{proposition} \label{limDiscs}
For all $I \subset \{1,2,3,4,5\}$ and $\theta \in \R$, the discs described by \eqref{disc6fam}
uniformly converge to discs contained in the complement of the orbifold point of 
$\CP(1,1,4)$, described in the coordinates $(\x : \y : \z)$ by

\begin{equation} \label{discP114}
  \x(w) = e^{-i\theta}\sqrt{rw^5 + c} \underset{j\neq0}{\prod_{j\notin I}}\tau_j(w); \ \y(w) = w; \ \z(w) =  
  e^{i\theta}\sqrt{rw^5 + c}\prod_{j\in I }\tau_j(w),
\end{equation}

where $\tau_j = \tau_{ -\frac{r}{c}^{\frac{1}{5}}e^{\frac{2\pi i}{5}j}}$. For
each $I$ and $\theta$, we denote this disc by $u_{I}^{\theta}$.
Moreover, the algebraic count of discs in the relative class $[u_I^\theta]$, for $|I|
= 0, 1, 2, 3 , 4, 5$ is
equal to $1, 5 ,10 ,10 ,5 ,1$, up to sign, respectively. 
\end{proposition}

\begin{proof} (\emph{of Lemma} \ref{asymp})
  
Consider the polynomial: 
\begin{eqnarray} 
 \Xi(w) & = & w^3(1 - \bar{b}w)(w - \nu)^3 \Psi(w) \nonumber \\ 
           & = & r(w - b)(1 - \bar{\nu} w)^3 + cw^3(1 - \bar{b}w)(w - \nu)^3 \label{poly6fam} \\ 
           & = & -c\bar{b}(w - \zeta)\prod_j (w - \eta_j), \label{poly6fam2}
\end{eqnarray}

where $|\zeta| > 1$, and write

$$\prod_j (w - \eta_j) = w^6 -\sigma_1w^5 + \sigma_2w^4 - \sigma_3w^3 +
\sigma_4w^2 - \sigma_5w + p.$$

Comparing the coefficients of 1 and $w^6$, we get:

\begin{equation} \label{eq61}
  p\zeta = -\frac{rb}{c\bar{b}},
\end{equation}

\begin{equation} \label{eq62}
1 + 3\nu \bar{b} = \bar{b} (\zeta + \sigma_1).
\end{equation}

For the following we recall that $h^2(w) =\frac{\Psi(w)w^3\tau_{\nu}^3(w)}{\prod\tau_j(w)}$,
in particular, by \eqref{Psi6}, $h^2(0) = -\frac{rb}{p}$ and $h^2(\nu) = \frac{r\tau_b(\nu)}{\prod_j q_j }$.

By \eqref{eq61}, \eqref{eq62}, and noting that $|\sigma_1 -3\nu| \leq 
|\sigma_1| + 3|\nu| \leq 9 \leq 9|\zeta|$, we get 
that:

$$1 = |b||\zeta + \sigma_1 -3\nu| \leq 10|\bar{b}||\zeta| = 10\frac{|h(0)|^2}{c}.$$

So $|h(0)|^2 \geq \frac{c}{10}$, proving the first statement of Lemma
\ref{asymp}. We see from (\ref{cond61}) that $\prod_{j \notin I}\eta_j^2 = O(t)$
and hence $p = O(t^{\frac{1}{2}}) \to 0$ as $t \to 0$. Also $\zeta \to \infty$,
$b = \ot \to 0$ and $\bar{b}\zeta = 1 + O(t^{\frac{1}{2}}) \to
1$.

Now we basically need to show that $\nu = \ot$, since the asymptotic behavior of
the $\eta_j$'s described in Lemma \ref{asymp} follows from \eqref{poly6fam} and $\nu = \ot$, 
after we separete the terms that are $\ot$, more precisely, $\Xi(w) = w(r + cw^5) + w^7
O(t^{1/2}) + O(t^{1/2})$. So, let's look to $\Xi(\nu)$ using \eqref{poly6fam} 
and \eqref{poly6fam2} (recall that $q_j = \tau_j(\nu) = (\nu - \eta_j)/(1 - \bar{\eta_j}\nu)$: 

\begin{eqnarray} \label{eqnu}
\Xi(\nu) & = & r(1- |\nu|^2)^3(\nu - b) \hspace{.1cm} \nonumber \\ 
& = & -c\bar{b}(\nu - \zeta)
\prod_j (\nu - \eta_j)  \hspace{.1cm} =  \hspace{.1cm} -c\bar{b}(\nu - \zeta)
\prod_j q_j (1 - \bar{\eta_j}\nu)
\end{eqnarray}

\begin{claim}\label{hnubounded}
  $h(\nu)$ is bounded away from $0$, and by \eqref{cond62}, $\prod_{j \notin I} q^2_j = 
  O(t)$, and therefore $\prod_{j} q_j = O(t^{1/2})$.
\end{claim}

\begin{proof}
  
Using that $\bar{b}\zeta = 1 + O(t^{1/2})$ and $\bar{b}\nu = O(t^{1/2})$ we see that
$$h^2(\nu) = \frac{r\tau_b(\nu)}{\prod_j q_j } = \frac{r(\nu - b)}{\prod_j q_j }\frac{1}{1
-\nu\bar{b}} = \frac{c\prod_j (1 - \bar{\eta_j}\nu)}{(1- |\nu|^2)^3(1 - \nu\bar{b})}(1 +
O(t^{1/2}))$$ 

is bounded away from zero. Indeed, if $\frac{c\prod_j (1 - \bar{\eta_j}\nu)}{(1-
|\nu|^2)^3(1 - \nu\bar{b})}$ approaches zero, then, since $c\bar{b}(\nu - \zeta)
\prod_j q_j$ is bounded, we get by \eqref{eqnu} that $b - \nu \to 0$ and hence
$\nu \to 0$. But in this case we see that $\prod_j (1 - \bar{\eta_j}\nu) \to 1$, not $0$.

Then we see from \eqref{cond62} that $\prod_{j \notin I} q^2_j = O(t)$, hence
$\prod_{j} q_j = O(t^{1/2})$. 
\end{proof}

We want to show that $b - \nu \to 0$ and hence $\nu \to 0$. For that to follow
from \eqref{eqnu} and Claim \ref{hnubounded}, we need to see that $|\nu|$ 
does not approach $1$ as $t \to 0$. 

\begin{claim} \label{nuto1}
  As $t \to 0$, $|\nu|$ is bounded by a constant strictly smaller than $1$.
\end{claim}

\begin{proof}
Let's look again to $\Xi(w)$, knowing 
that $b = O(t^{1/2})$:

\begin{equation}\label{eqnuto1}
  \Xi(w) = rw(1 - \bar{\nu}w)^3 + cw^3(w - \nu)^3 + w^7 O(t^{1/2}) + O(t^{1/2})
\end{equation}

Then we see that the roots $\eta_j$ lying inside the disc are very close to the 
solutions of $rw(1 - \bar{\nu}w)^3 + cw^3(w - \nu)^3 = 0$, for $t$ very small. 
The non-zero solutions satisfy:

\begin{equation}\label{eqnucontr}
  |w^2 \tau_{\nu}(w)^3| = \frac{r}{c} < 1.
\end{equation}

Now, assume there is a sequence of values of $t$ tending to $0$ and holomorphic
discs such that $\nu \to \nu_0$, $|\nu_0| =1$. From equation \eqref{eqnuto1}, we
conclude that three of the roots $\eta_j$ of $\Xi(w)$ converge to $\nu_0$, say
$\eta_1$, $\eta_2$, $\eta_3$, one converges to $0$, say $\eta_0$, and the other
two solutions, $\eta_4$ and $\eta_5$ converge to square roots of
$\frac{r}{c\nu^3_0}$, by \eqref{eqnucontr} (for values of $w$ lying outside a
neighborhood of $\nu_0$, $\tau_{\nu}(w) = \frac{w - \nu}{1 - \bar{\nu}w} =
-\nu\frac{1 - \nu^{-1}w}{1 - \bar{\nu}w} \to -\nu_0$). 
  
Let $\epsilon_j$ be such that $\eta_j = \nu - \epsilon_j$, for $j = 1 ,2 ,3$.
By \eqref{eqnu}, recalling that $b = O(t^{1/2})$, $\prod_{j} q_j = O(t^{1/2})$
and $\bar{b}\zeta = 1 + O(t^{1/2})$ we see that $(1 - |\nu|^2)^3 = 
K\epsilon_1\epsilon_2\epsilon_3$, where $$K = \frac{-c\bar{b}(\nu - \zeta)(\nu -\eta_0)}{r(\nu - b)}
(\nu -\eta_4)(\nu -\eta_5) \to \frac{c}{r}\nu_0
 - \frac{1}{\nu^4_0}$$ is bounded above and below.

Since $\prod_{j} q_j = O(t^{1/2})$, some $q_j \to 0$, for $j = 1 , 2, 3$. 
Up to relabeling, assume $q_1 \to 0$. We have that, for $j = 1 , 2, 3$:

\begin{equation}
  q_j = \tau_j(\nu) = \frac{\nu - \eta_j}{1 - \bar{\eta_j}\nu} =
   \frac{\epsilon_j}{1 - |\nu|^2 + \nu\bar{\epsilon_j}} .
\end{equation}
 
So, 

\begin{equation}\label{eps1}
 \frac{1}{q_1} = \left(\frac{1 - |\nu|^2}{\epsilon_1}\right) + 
 \nu\frac{\bar{\epsilon_1}}{\epsilon_1} = 
 \left(\frac{K\epsilon_2\epsilon_3}{\epsilon_1^{2}}\right)^{\frac{1}{3}} + 
 \nu\frac{\bar{\epsilon_1}}{\epsilon_1} \to \infty. 
\end{equation}

Passing to a subsequence if needed, we may assume that $|\epsilon_1| \leq
|\epsilon_2| \leq |\epsilon_3|$. So by \eqref{eps1}, $\frac{\epsilon_1}{\epsilon_3} \to 0$.
Therefore,  
  
 \begin{equation}\label{eps3}
 \left|\frac{1}{q_3}\right| = \left|\left(\frac{K\epsilon_1\epsilon_2}{\epsilon_3^{2}}\right)^{\frac{1}{3}} + 
 \nu\frac{\bar{\epsilon_3}}{\epsilon_3}\right| \to 
 \left|\nu_0\right| = 1. 
\end{equation} 
  
But, by \eqref{eqnucontr} and $|\tau_{\eta_j}(\nu)| = |\tau_{\nu}(\eta_j)|$,
$|\eta_3|^2|q_3|^3 \to \frac{r}{c}$, which gives a contradiction since $|\eta_3|
\to |\nu_0| = 1$ and $|q_3| \to 1$.
\end{proof}  

Using Claim \ref{nuto1}, equation \eqref{eqnu}, $b = O(t^{1/2})$ and
$\prod_{j} q_j = O(t^{1/2})$, we see that $\nu = O(t^{1/2})$.

 So $\Xi(w) = w(r + cw^5) + w^7 O(t^{1/2}) + O(t^{1/2})$ and, assuming $|\eta_0|
 \leq |\eta_j| \ \forall j $, we get that for $ j \neq 0$, $\eta_j \to
 -(\frac{r}{c})^{\frac{1}{5}}e^{\frac{j2 \pi i}{5}}$, while $\eta_0
 = \ot$. This finishes the proof of Lemma \ref{asymp}.
 \end{proof}


\begin{proof} (\emph{of Proposition} \ref{limDiscs})
 
 In the limit $t=0$, taking into account that $b = \ot$, $\nu = \ot$, $\eta_0 =
 \ot$ and $\eta_j \to -(\frac{r}{c})^{\frac{1}{5}}e^{\frac{j2 \pi i}{5}}$, we
 have that $\Psi(w)$, $\prod_{j=1}^5\tau_j(w)$, and $h^2(w)$, thought as maps
 form $\D$ to $\CP^1$, uniformly converge to

\begin{equation*}
 \Psi(w) = \frac{r + cw^5}{w^5}; \ \ \prod_{j=1}^5\tau_j(w) = \frac{cw^5 + r}{rw^5 + c}; 
  \ \ h^2(w) = rw^5 + c.
\end{equation*}

So, for instance, in the case $I = \{1,2,3,4,5\}$ we get that $z_1(w)$ , 
$z_2(w)$, $z_3(w)$ uniformly converge to

\begin{equation*}
  z_1(w) = w^2; \ z_2(w) = e^{-i\theta}w\sqrt{rw^5 + c}; \ z_3(w) =  
  e^{i\theta}\frac{cw^5 + r}{rw^5 + c}\sqrt{rw^5 + c};\\
\end{equation*}
$$
  z_0(w) = \frac{z_2^2(w)}{z_1(w)} 
  = e^{-2i\theta}(rw^5 + c).
  $$
  
Hence using the $(\x:\y:\z)$ coordinates of $\CP(1,1,4)$ we have \vspace{-0.1cm}

\begin{equation*}
  \x(w) = e^{-i\theta}\sqrt{rw^5 + c}; \ \ \y(w) = w; \ \ \z(w) =  
  e^{i\theta}\sqrt{rw^5 + c}\frac{cw^5 + r}{rw^5 + c}.
\end{equation*}

In general, for each $I$, the holomorphic discs uniformly converge to discs
$u^{\theta}_{I}$ given by
\vspace{-0.1cm}

\begin{equation*} 
  \x(w) = e^{-i\theta}\sqrt{rw^5 + c} \underset{j\neq0}{\prod_{j\notin I}}\tau_j(w); \ \y(w) = w; \ \z(w) =  
  e^{i\theta}\sqrt{rw^5 + c}\prod_{j\in I }\tau_j(w).
\end{equation*}

Note that none of these discs pass through the singular point $(0:0:1)$ of 
$\CP(1,1,4)$. 

As before, we have extra automorphisms of the disc, given by $w \mapsto e^{ik\frac{2\pi}{5}}w$
 that don't change \eqref{Psi6} and we need to quotient out by this action
 of $\Z/5\Z$. We get that $k \in \Z_5$ acts on $u^{\theta}_{I}$ as follows: 
 
 \begin{center}
 \begin{tabular}{|c||c|c|c|c|c|c|}
   \hline
   $|I|$ &  0 & 1 & 2 & 3 & 4 & 5 \\
\hline
   $u^{\theta}_{I}\mapsto$
    & $u^{\theta +k\frac{2\pi}{5}}_{\emptyset}$
    & $u^{\theta +2k\frac{2\pi}{5}}_{I -k}$
    & $u^{\theta +3k\frac{2\pi}{5}}_{I -k}$
    & $u^{\theta +4k\frac{2\pi}{5}}_{I -k}$
    & $u^{\theta}_{I -k}$
    & $u^{\theta +k\frac{2\pi}{5}}_{I}$\\
   \hline
 \end{tabular} 
 \end{center}\vspace{-0.1cm}
 where $I - k = \{l \in \{1, \dots, 5\}: l \equiv j -k \mod 5, \ j \in I\}$.

We also note that, for fixed $I$, varying $\theta \in [0,2\pi]$ and looking at
the boundary of the discs, the 2-cycle swept by $\del u_I^{\theta}$ is $[\del
u_I^{\theta}] = \pm 5[T^c_{r,0}]$. Therefore, in the case $|I| = 0$ or $5$, after
quotienting by $\Z/5\Z$, the algebraic count is $\pm 1$. In the cases $|I| = 1$
or $4$, the action of $\Z/5\Z$ permutes the indices, so the moduli space of
holomorphic discs is given by $\{u^{\theta}_I; \theta \in [0,2\pi]\}$, where $I
= \{1\}$, respectively $I= \{2,3,4,5\}$, hence the algebraic count is $\pm 5$.
Similarly for $|I| = 2$ or $3$, the action of $\Z/5\Z$ permutes the indices, so
the moduli space of holomorphic discs is given by $\{u^{\theta}_I; \theta \in
[0,2\pi]\}\cup\{u^{\theta}_{I'}; \theta \in [0,2\pi]\}$, where $I = \{1,2\}$ and $I'
=\{1,3\}$, respectively $I= \{3,4,5\}$, $I'= \{2,4,5\}$, hence the algebraic
count is $\pm 10$.

\end{proof}

\begin{lemma}\label{x0xt}
Assuming regularity, each of the above families of discs in $X_0$
has a corresponding family in $X_t$, for a sufficiently small $t$.  
\end{lemma}

\begin{proof}
We consider the 3-dimensional 
complex hypersurface $\X$ inside $\C \times (\CP(1,1,1,2) \setminus (0:0:0:1))$ defined by the equation
\begin{equation}
  \X : z_0z_1 - (1 - t)z_2^2 - tz_3 = 0
\end{equation}
containing 
\begin{equation}
 \Lag  =  \left\{ (t,(z_0:z_1:z_2:z_3)) ; t \in \R \ \ \text{and} \ \ \left|\frac{z_2z_3}{z_1^3} - c \right| = r ; 
\left|\frac{z_2}{z_1}\right|^2 = \left|\frac{z_3}{z_1^2}\right|^2  \right\}   
\end{equation}
as a totally real submanifold. 

Then we consider $\M(\X,\Lag)$ the moduli space of Maslov index 2 holomorphic 
discs in $\X$ with boundary on $\Lag$. By applying the maximum principle to the 
projection on the first factor, we see that such holomorphic discs lie inside the fibers 
$X_t$, for $t \in \R$. Let's consider discs that stay away from the singular point in 
$X_0$, such as those computed above.

Assuming the discs above are regular in $X_0$ implies they are regular as discs
in $\X$. This follows from the splitting $u^*T\X = u^*TX_0 \oplus \C$ and the
$\bar{\del}$ operator being surjective onto 1-forms with values in $u^*TX_0$,
by the assumed regularity, and onto 1-forms with values in $\C$, by regularity
of holomorphic discs in $\C$ with boundary in $\R$. 

Hence $\M(\X,\Lag)$ is smooth near the solutions for $t = 0$ given above and the
map $\M(\X,\Lag) \rightarrow \R$, which takes a disc in the fiber $X_t$ to $t$,
is regular at 0. Therefore for a small $t$, all the Maslov index 2 holomorphic
discs in $X_0$ computed above deform to holomorphic discs in $X_t$. 
\end{proof} 

The regularity of
the discs above is proven in Theorem \ref{reg6fam}. 

The families of discs in the classes $2H -
5\beta +k\alpha$, given by $\eqref{disc6fam}$ have been shown to converge uniformly to
the corresponding ones in $\CP(1,1,4)$ given by \eqref{discP114}. Smoothness
of the moduli space $\M(\X,\Lag)$ near the families of discs given by \eqref{discP114}
in $X_0$ guarantees that each family $\{u^{\theta}_I; \theta \in
[0,2\pi]\}$ has a unique family in $X_t$ converging to it, for all $t$ sufficiently small. 
Hence the counts of Maslov index 2 holomorphic discs in the classes $2H -
5\beta +k\alpha$ for $X_t$ are the same as the ones computed in $X_0$. 
This finishes the proof of Theorem \ref{6fam}. 
\end{proof}

\subsection{Regularity}\label{reg}
In order to prove regularity, we consider the following two lemmas.

\begin{lemma}\label{lr1}
  Let $u_\theta$ be a one parameter family of Maslov index 2 holomorphic 
  discs in a K\"ahler 4 dimensional manifold $X$ with boundary on a Lagrangian $L$. Set 
  $u = u_0$ and $V = \left.\frac{\del}{\del \theta} u_\theta \right|_{\theta = 0}$
  a vector field along $u$, tangent to $TL$ along the boundary of $u$.
  
 If $V$ is nowhere tangent to $u(\D)$ and $u: \D \rightarrow X$ is an
 immersion, then $u$ is regular.
\end{lemma}

\begin{proof}
  As $u$ is an immersion, we can consider the splitting $u^*TX \cong T\D \oplus
  \LB$ as holomorphic vector bundles, where $\LB$ is the trivial line bundle
  generated by $V$. Also $u_{|{S^1}}^*TL \cong TS^1 \oplus Re(\LB)$, where
  $Re(\LB) = \text{span}_\R \{V\}$ and $S^1 \cong \del \D$.
  
 So a section $\zeta \in \Omega^0_{u_{|S^1}^*TL}(\D, u^*TX)$ of $u^*TX$ that
 takes values in $u_{|S^1}^*TL$ along the boundary splits as $\zeta_1 \oplus
 \zeta_2 \in \Omega^0_{TS^1}(\D, T\D) \oplus \Omega^0_{Re(\LB)}(\D, \LB)$. Since
 $J$ is an integrable complex structure, the kernel of the linearized operator
 $D_{\bar{\del}}$ is given by
 
 $$
 \{ \zeta \in \Omega^0_{u_{|S^1}^*TL}(\D, u^*TX) ; \bar{\del}\zeta = 0 \} 
 $$
 which is isomorphic to 
\begin{eqnarray}
   \{\zeta_1 \in \Omega^0_{TS^1}(\D, T\D); \bar{\del}\zeta_1 = 0\} &\oplus& \{\zeta_2 \in 
 \Omega^0_{Re(\LB)}(\D, \LB); \bar{\del}\zeta_2 = 0\} \nonumber \\ 
 \cong T_{Id}\text{Aut}(\D) &\oplus & \text{Hol}((\D,S^1), (\C,\R)) \nonumber
\end{eqnarray}
 
 The last term on the right comes from $\LB$ being trivial. $\text{Aut}(\D)$ is known to 
 be 3 dimensional, while $\text{Hol}((\D,S^1),(\C,\R))$ is the space of real-valued constant functions.
 Therefore,
 $$
 Dim\text{Ker} (D_{\bar{\del}}) = 4 = 2\cdot \chi(\D) + \mu(u^*TX, u_{|S^1}^*TL) = 
 \text{index} (D_{\bar{\del}} )
 $$
\end{proof}

The following lemma sets a sufficient condition for $V$, as given in the previous 
lemma, not to be tangent to $u(\D)$.

\begin{lemma}\label{lr2} 
  Let $u:\D \to X$ be a Maslov index 2 holomorphic disc in a K\"{a}hler 4-manifold $X$ with
  boundary on a Lagrangian $L$ such that $u_{S^1}: S^1 \rightarrow L$ is an immersion,
  and $V$ a holomorphic vector field on $X$ along $u$, tangent to $L$ at the boundary. If $V$ is
   not tangent to $u(\D)$ at the boundary, then $u$ is an immersion and $V$ is nowhere 
   tangent to $u(\D)$. 
  \end{lemma}

\begin{proof}
  
  Suppose that either $du(x) =0$ or $V$ is tangent to $u(\D)$ at a point
  $u(x)$; up to reparametrizing the disc we may assume that $x \neq 0$. Consider
  another holomorphic vector field $W = du(\frac{\del}{\del\theta})$ given by an
  infinitesimal rotation, which is tangent to $u(\D)$, has a zero at 0 and is
  also tangent to $L$ at the boundary. Then the Maslov index can be computed
  using $\det^2(W\wedge V)$, so the number of zeros of $\det(W\wedge
  V)$ is $\mu(u^*TX,u|_{\del \D}TL)/2 = 1$. But $W$ vanishes at $u(0)$ and
  either vanishes or is parallel, as a complex vector, to $V$ at $u(x)$. Since
  the zeros of $W\wedge V$ always occur with positive multiplicity, as the vector
  fields are holomorphic, we get a contradiction. 
   
   \end{proof}
  
 Now we are ready to prove 
  
\begin{theorem}\label{reg3fam}
 The holomorphic discs representing the classes $\beta$ and $H - 2\beta + m\alpha$ in $X_t$
 computed on proposition \ref{bdisc} and theorem \ref{3fam} are regular, for small $t$.
\end{theorem}  

\begin{proof}
  By Lemmas \ref{lr1}, \ref{lr2}, we only need to notice that for each of the holomorphic discs 
  $u_I^\theta$ considered, the vector field $V(w) = \frac{\del}{\del \theta}u_I^\theta(w)$
is not tangent to $u_I^\theta(\del\D)$. We note that in the limit $t = 0$, we have
 that  $u_I^\theta$ uniformly converge, in a compact neighborhood of the boundary,
  to a holomorphic disc given by: 
 $$ z_1(w) = w; \ \ z_2(w) = e^{-i\theta}\sqrt{rw^2 + c} \underset{j\neq0}{\prod_{j\notin I}}\tau_{\eta_j}(w);
 \ \ z_3(w) = e^{i\theta}\sqrt{rw^2 + c} {\prod_{j\in I}}\tau_{\eta_j}(w)$$
 where $\eta_1 = i\sqrt{r/c}$, $\eta_2 = -i\sqrt{r/c}$. So we see that in the 
 limit $t =0$, $V(w) = \frac{\del}{\del \theta}u_I^\theta(w)$ is parallel to the fibers
 of $f(z_0:z_1:z_2:z_3) = \frac{z_2z_3}{z_1^3}$ restricted to $X_0$ and nowhere vanishing.
 Therefore, $V(w)$ is not tangent to $u_I^\theta(\del\D)$ for $t =0$, and by continuity
 the same holds for small $t$. 
  
\end{proof}

\begin{theorem}\label{reg6fam}
 The holomorphic discs in $X_0$ computed in Theorem \ref{6fam} 
 are regular. By Lemma \ref{x0xt}, for small $t$, the corresponding
 holomorphic discs in the classes $2H - 5\beta + m\alpha$ in $X_t$
 are also regular.
\end{theorem}

\begin{proof}
  Similar to the other cases, for each considered holomorphic disc $u_\theta$, we have the 
 vector field $V(w) = \frac{\del}{\del \theta}u_\theta(w) = (i\x(w), -i\z(w))$ in coordinates $(\x,\z)$, 
 for $\y=1$ on $\CP(1,1,4)$ along the boundary. These vectors are not tangent to $u_\theta(\D)$
 along the boundary, since they are nonvanishing and parallel to the fibers of  $f(\x,\z) = \x\z$.
 Hence, by lemmas \ref{lr1}, \ref{lr2}, these discs are regular.
\end{proof}

\subsection{Orientation}

The choice of orientation of the moduli space of holomorphic discs is determined
by a choice of spin-structure on the Lagrangian; see \cite{CHO} section 5. In
this section we choose a spin structure on our Lagrangian $T(1,4,25)$ torus and argue that, under
the choice of orientations made in \cite{FOOO}, see also
section 7 of \cite{CHO}, the evaluation map from each of the moduli spaces of
Maslov index 2 holomorphic discs considered in this section to the $T(1,4,25)$ torus
is orientation preserving. We use the same definition of spin-structure given by 
C. Cho in section 6 of \cite{CHO}:

\begin{definition}
  A \emph{spin structure} on an oriented vector bundle $E$ over a manifold $M$ 
  is a homotopy class of a trivialization of $E$ over the 1-skeleton of $M$ that 
  can be extended over the 2-skeleton.
\end{definition}

In case of surfaces, it's enough to consider a stable trivialization of the
tangent bundle. We see that $\del \alpha$ and $\del \beta$ form a basis of
$H_1(T^c_{r,0}, \Z) $ and hence they induce a trivialization of the tangent bundle
of $T^c_{r,0}$ oriented as $\{\del \alpha, \del \beta\}$. 

The orientation of the moduli space at a disc $u: (\D, \del \D) \rightarrow
(X^{2n},L^n)$ is then given by the orientation of the index bundle of the linearized
operator $D\bar{\del}_u$ that is induced by the chosen trivialization of the tangent
bundle $TL$ along $\del \D$, as described in 
\cite{FOOO}.

The rough idea is that we extend the trivialization of the tangent bundle of the
Lagrangian to a neighborhood of $\del \D$, then take a concentric circle
contained in it, and pinch it to a point $O \in \D$, the part of the disc inside
the circle becoming a $\CP^1$. The trivialization of $TL$ along the pinched
neighborhood gives a trivialization of its complexification $TX$. This way,
considering the isomorphisms given by the trivializations, the linearized 
operator is homotopic to a $\bar{\del}$ operator on $\D\cup\CP^1$, whose kernel 
consists of pairs $(\xi_0,\xi_1)$ where: $\xi_0$ is
a holomorphic section of the trivial bundle $\C^n$ over
the disc, with boundary on the trivial subbundle $\R^n$, i.e, a constant maps into 
$\R^n$; and $\xi_1$ is a holomorphic section of the bundle induced by $u^*TX$ over
 $\CP^1$, which we denote by $TX|_{\CP^1}$. These
sections must match at $O \in \D$ and the `south pole' $S \in \CP^1$. In other
words, Fukaya, Oh, Ohta, Ono show that the index of the linearized operator (seen as a virtual
vector space Ker$D\bar{\del}_u$ $-$ CoKer$D\bar{\del}_u$) is isomorphic to the
kernel of the homomorphism:

\begin{equation}\label{Orie}
  (\xi_0, \xi_1) \in Hol(\D,\del\D: \C^n, \R^n) \times Hol (\CP^1, TX|_{\CP^1}) \rightarrow 
  \xi_0(O) -\xi_1(S) \in  \C^n \cong TX|_S
\end{equation}

Now the kernel can be oriented by orienting $\R^n \cong Hol(\D,\del\D: \C^n, \R^n)$ (which
is essentially the trivialization of the tangent space of the Lagrangian), since
$Hol (\CP^1, TX|_{\CP^1})$ and $\C^n$ carry complex orientations. For a
detailed account of what we just discussed, see Chapter of
\cite{FOOO} Part II, also Proposition 5.2 in \cite{CHO}.

Denote by $\tilde{\M}(\gamma)$ the space of holomorphic discs on $\CP^2$
with boundary on $T^c_{r,0}$ in the class $\gamma$, not quotiented out by
Aut$(\D)$. By the same argument as in section 8 of \cite{CHO},  the factor $\R^n \cong
Hol(\D,\del\D: \C^n, \R^n)$ in \eqref{Orie}
corresponds to the subspace of $T_u\tilde{\M}(\gamma)$ given by the deformations of $u$
which correspond to translations
along the boundary of $T^c_{r,0}$, i.e., generated by $V =
\frac{\del}{\del \theta}u_\theta$ and by infinitesimal rotations in Aut$(\D)$. 
This
way, we orient the moduli space of discs accordingly with our chosen orientation
$\{\del \alpha, \del \beta\}$. In particular, $\M(\beta)$, which consists of 
one-parameter family of discs $u_\theta$ described in Proposition
\ref{bdisc}, is oriented in the positive direction of $\theta$, since 
$\frac{\del u_\theta}{\del\theta}$ and the tangent vector to the boundary of
$u_\theta$ form a positive oriented basis of $TT^c_{r,0}$;
while the other moduli spaces $\M(H - 2\beta +m\alpha)$ and $\M(2H - 5\beta
+k\alpha)$ are oriented in the negative direction of the parameter $\theta$, since 
in these cases $\frac{\del u_\theta}{\del\theta}$ and the tangent vector to the boundary of
$u_\theta$ form a negative oriented basis.

\begin{proposition}
  The evaluation maps from $\M_1(\beta)$, $\M_1(H - 2\beta +m\alpha)$ and $\M_1(2H - 5\beta
+k\alpha)$ to $T^c_{r,0}$ are all orientation preserving.
\end{proposition}
 
Here the subscript 1 refers to the moduli space with one marked point at the
boundary. The proof of the proposition above follows from the same argument as
in Proposition 8.2 in \cite{CHO}. 
\newpage
As a corollary of all we have done in this section, we get 
\begin{theorem}
  In the region corresponding to $T(1,4,25)$ tori, the mirror superpotential is given by 
  \eqref{WChe2}: 
  
  \begin{equation*}
    W_{T(1,4,25)} = u + 2\frac{e^{-\Lambda }}{u^2}(1 + w)^2 + 
 \frac{e^{-2\Lambda }}{u^5w}(1 + w)^5
  \end{equation*}
\end{theorem}
 
\section{The monotone torus} \label{MonTor}

  In this section we show that we can modify our symplectic form in a
  neighborhood of $D_5$ to a new one for which $T^c_{r,0}$ is Lagrangian
  monotone. Recall that a Lagrangian $L$ in a symplectic manifold $(X,\omega)$
  is called monotone if there exists a constant $M_L$ such that for any disc $u$
  in $\pi_2(X,L)$ satisfies $$ \int u^*\omega = M_L\mu_L(u) $$ where $\mu_L$ is
  the Maslov class.
  
  Recall the relative homotopy classes $H = [\CP^1]$, $\beta$ and $\alpha$
  defined after Proposition \ref{bdisc}, with Maslov indices $\mu(H) = 6$,
   $\mu(\beta) = 2$ and $\mu(\alpha) = 0$. A disc in the
  class $\alpha$ is given by the Lefschetz thimble over the interval $[0, c -r]$
  with respect to the symplectic fibration $f$, so $\int_\alpha \omega = 0$. We
  see that $L=T^c_{r,0}$ satisfies the monotonicity condition if and only if
  $[\omega]\cdot \beta = \int_\beta \omega = \Lambda/3$, where $\Lambda =
  \int_{H}\omega$. 
  
  We could try to compute $\int_\beta \omega$ ($= \int_\beta \tilde{\omega}$), which
  depends on our choice of $c$ and $r < c$. As $r \to 0$, $[\omega]\cdot \beta$
  converges to $0$. We could take then a very large value for $c$ and $r$ very
  close to $c$. Nonetheless, a careful computation shows that the symplectic
  area $[\omega]\cdot \beta$ remains smaller than $\Lambda/3$. Taking then
  another approach, we look at table \ref{HomTab} and see that $D_5$
  intersects $\beta$ in 2 points and $H$ in 5 points. We can then build 
  a 2-form $\sigma$ supported in a neighborhood of $D_5$, so that the ratio
  $[\sigma]\cdot \beta /[\sigma]\cdot H = 2/5 > 1/3$. By adding a large enough
  multiple of $\sigma$ to $\omega$, we get a K\"ahler form $\hat{\omega}$ for which
  $[\hat{\omega}]\cdot[\beta] = \Lambda /3$. 
  
  \begin{proposition}\label{monform}
   There is a K\"{a}hler form $\hat{\omega}$ for which $T^c_{r,0}$ is Lagrangian
  monotone. Moreover, $\hat{\omega}$ can be chosen to agree with $\omega$ away from a 
  neighborhood of $D_5$ that is disjoint from $T^c_{r,0}$.
  \end{proposition}
  
  \begin{proof}
Take a small enough value of $r$, for which it is straightforward to see that
 $[\omega]\cdot \beta < \Lambda/3$. In order to make $T^c_{r,0}$ monotone
we perform a K\"ahler inflation in a neighborhood of the quintic $D_5$ (see section
\ref{HomClass}) to achieve $[\hat{\omega}]\cdot \beta = 
\int_{[\CP^1]}\hat{\omega}/3$, keeping $[\hat{\omega}]\cdot \alpha = 0$. 

Take a small neighborhood $\mathcal{N}$ of $D_5 = \{s_5 = 0\}$ not intersecting
$T^c_{r,0}$, where $s_5 = x\xi^2 - 2cy^2z\xi + c^2y^5$. Take a cutoff function
$\chi$ such that $\chi(|s_5|^2)$ is equal to $1$ in a neighborhood of $D_5$ and
is equal to $0$ in the complement of $\mathcal{N}$. We then define $\hat{\omega}
= \omega + K\sigma$ for 

\begin{equation*}
 \sigma = \frac{i}{2}\del\bar{\del} \log \left(|s_5|^2 +
\varepsilon \chi(|s_5|^2)(|x|^2 + |y|^2 + |z|^2)^5 \right) 
\end{equation*}
where $K$ and $\varepsilon$ are constants to be specified. We use the fact that
 $\del\bar{\del} \log(|f|^2) = 0$ for a holomorphic function $f$, to note that
 the expression for $\sigma$ is the same for the homogeneous coordinates
 $(1:\frac{y}{x}:\frac{z}{x})$, $(\frac{x}{y}:1:\frac{z}{y})$ and
 $(\frac{x}{z}:\frac{y}{z}:1)$, therefore $\sigma$ defines a 2-form on $\CP^2$,
 and also to note that $\sigma = \del\bar{\del} \log(|s_5|^2) = 0$ outside
 $\mathcal{N}$, so $T^c_{r,0}$ is Lagrangian with respect to $\hat{\omega}$. 
 
  \begin{lemma}
   $[\sigma] = 5[\omega_{FS}]$ is independent of $\varepsilon$ and the cutoff 
   function $\chi$.
  \end{lemma} 
 
 \begin{proof}[Proof of Lemma]
   To determine the cohomology class of $\sigma$, it is enough to compute $\int_{[\CP^1]}
\sigma$. For this we consider $[\CP^1] = \{x = 0\}$, and write $\sigma=
\frac{1}{4}dd^c\text{log}\psi_j$, where 
$$
\psi_1 = \frac{|s_5|^2 +
\varepsilon \chi(|s_5|^2)(|x|^2 + |y|^2 + |z|^2)^5}{|y|^{10}},
$$
$$
\psi_2 = \frac{|s_5|^2 +
\varepsilon \chi(|s_5|^2)(|x|^2 + |y|^2 + |z|^2)^5}{|z|^{10}},
$$ are
homogeneous functions respectively defined on $\{y \neq 0\}$, $\{z \neq 0\}$,
such that

\begin{equation}
\frac{\psi_1}{\psi_2} = \frac{|z|^{10}}{|y|^{10}}
\end{equation}

We then divide $[\CP^1] = \{x = 0\}$ into two hemispheres $H_+$, $H_-$,  
contained in $\{y \neq 0\}$, $\{z \neq 0\}$, respectively, to compute

\begin{eqnarray}
\int_{[\CP^1]} \omega & = & \frac{1}{4}\int_{H_+} dd^c\text{log}\psi_1+ 
\frac{1}{4}\int_{H_-} dd^c\text{log}\psi_2 = \frac{1}{4}\int_{\del H_+} d^c\text{log}\psi_1 
+ \frac{1}{4}\int_{\del H_-} d^c\text{log}\psi_2 \nonumber \\
                                   & = &\frac{1}{4}\int_{\del H_+} d^c\text{log}\psi_1 - d^c\text{log}\psi_2 
= \frac{1}{4}\int_{\del H_+} d^c\text{log}\left(\frac{|z|^{10}}{|y|^{10}}\right) = \frac{5}{4}
\int_{\del H_+} d^c\text{log}\left(\frac{|z|^2}{|y|^2}\right) \nonumber \\
\end{eqnarray}
which by comparison with the same calculation for $\omega_{FS}$ is $5$ times the
area of $[\CP^1]$ with respect to the Fubini-Study form $\omega_{FS}$.
\end{proof}

In particular, taking $\varepsilon \to 0$ we get that $\sigma$ converges to a
distribution supported at $D_5 =\{s_5 = 0\}$. Now considering $\alpha$, $\beta$,
$H = [\CP^1]$ as cycles in $H_2(\CP^2, \CP^2 \setminus \mathcal{N})$ and
$[\sigma] \in H^2(\CP^2, \CP^2 \setminus \mathcal{N})$ we see that their
$\sigma$-areas are a constant $\pi$ times their intersection number with 
$D_5$, i.e.,
$\int_{\alpha} \sigma = \pi\alpha \cdot [D_5] = 0$, $\int_{\beta} \sigma = \pi
\beta \cdot [D_5] = 2\pi$ and $\int_{H} \sigma = \pi [\CP^1] \cdot [D_5] = 5\pi$.  

Then, since the ratio between the $\sigma$-area of $\beta$ and $H$ is $2/5 >
1/3$, we can choose a constant $K$, so that $[\hat{\omega}]
\cdot \beta = [\hat{\omega}] \cdot H/3$. Given this value of $K$, we can choose $\varepsilon$
small enough to ensure that $\hat{\omega}$ is nondegenerate and hence a
K\"{a}hler form for which $T^c_{r,0}$ is monotone Lagrangian. 
\end{proof}
  
  \subsection{The monotone $T(1,4,25)$ torus is exotic}
  
  We are now going to prove that the count of Maslov index 2 holomorphic discs
  is an invariant of monotone Lagrangian submanifolds. We will see that it
  suffices to show it is an invariant under deformation of the almost complex
  structure. Let $L$ be a monotone Lagrangian submanifold of a symplectic
  manifold $(X, \omega)$ and $J_s$, $s \in [0,1]$ a path of almost complex
  structures such that $(L,J_0)$, $(L,J_1)$ are regular, i.e., for $k = 0 , 1$,
  Maslov index 2 $J_k$-holomorphic discs are regular. Note that since $L$ is
  monotone there are no $J_s$-holomorphic discs of nonpositive Maslov index for
  any $s \in [0,1]$. Since $L$ is orientable, Maslov indexes are even, so the
  minimum Maslov index is 2. Therefore, there is no bubbling for Maslov index 2
  $J_s$-holomorphic discs bounded by $L$, when we vary $s$.
 
Consider then, for $\beta \in \pi_2(X,L)$, $\mu(\beta) = 2$, the moduli spaces
$\M(\beta, J_s)$ of $J_s$-holomorphic discs representing the class $\beta$
, modulo reparametrization.  
We choose
the path $J_s$ generically so that the moduli space 
$$
 \tilde{\M}(\beta) =
\bigsqcup_{s = 0}^{1} \M(\beta,J_s) 
$$ 
is a smooth manifold, with 

\begin{equation*}
\del \tilde{\M}(\beta) = \M(\beta,J_0) \cup \M(\beta,J_1).
 \end{equation*}

  \begin{lemma}\label{lemMon}
    If $L$ is oriented and monotone, the classes and algebraic count of Maslov index 2 $J$-holomorphic 
    discs with boundary on $L$ are independent of $J$, as long as $(L,J)$ is regular.
  \end{lemma}
  
  \begin{proof} Since $L$ is oriented, there is no $J$-holomorphic discs with
  odd Maslov index and for $L$ monotone, the Maslov index of a disc is
  proportional to its area. Hence, Maslov index 2 $J$-holomorphic discs are the
  ones with minimal area, for any almost complex structure $J$. L. Lazzarini
  (\cite{Laz}) and Kwon-Oh (\cite{KOh}) proved that for any $J$-holomorphic disc
  $u$ with $u(\del\D) \subset L$, there is a somewhere injective $J$-holomorphic
  disc $v$, with $v(\D) \subset u(\D)$. Therefore, by minimality of area, Maslov
  index 2 discs are somewhere injective. By connectedness of the space of
  compatible almost complex structures we can consider $J_s$, $s \in [0,1]$ a
  generic path of almost complex structures such that $(L,J_0)$, $(L,J_1)$ are
  regular as above. The 
  result follows immediately from the cobordism $\tilde{\M}(\beta)$
  between $\M(\beta,J_0)$ and $\M(\beta,J_1)$. 
  \end{proof}
  
  \begin{theorem}
    If $L_0$ and $L_1$ are symplectomorphic monotone Lagrangian submanifolds
    of a symplectic manifold $(X, \omega)$, with an almost complex structure $J$
    so that $(L_0, J)$ and $(L_1,J)$ are regular, then algebraic counts of Maslov index 2 
    $J$-holomorphic discs, and in particular the numbers of different classes 
    bounding such discs, are the same.
  \end{theorem}
  
  \begin{proof}
    Let $\phi: X \rightarrow X$ be a symplectomorphism with $\phi(L_1) = L_0$. 
    Apply Lemma \ref{lemMon} with $L = L_0$, $J_0 = J$, $J_1 = \phi_*J$.
  \end{proof}
  
 \begin{corollary}
   The monotone $T(1,4,25)$ torus is not symplectomorphic to either the 
  monotone Chekanov torus or the monotone Clifford torus.
 \end{corollary}
 
 \begin{remark}

We can try to find an exotic torus in $\C^2$ by considering the $T(1,4,25)$ torus in
affine charts. If we restrict to the coordinate charts $\{y \neq 0\}$ or $\{z
\neq 0\}$, only the discs in the class $\beta$ remain. Hence we cannot
distinguish the $T(1,4,25)$ torus, considered in the charts $\{y \neq 0\}$ or $\{z
\neq 0\}$, from the usual Chekanov torus, which also bounds a single family of
holomorphic discs in $\C^2$. In the $\{x \neq 0\}$ coordinate chart, another
family of holomorphic discs in the class $2H -5\beta -\alpha$ remains present,
besides the one in the class $\beta$. This can be checked directly or just by
observing that the intersection numbers of the complex line $\{x = 0\}$ with
$H$, $\beta$, $\alpha$ are 1 , 0 and 2, respectively. Therefore our methods
cannot distinguish the $T(1,4,25)$ torus, considered in the chart $\{x \neq 0\}$,
from the usual Clifford torus in $\C^2$, which also bounds two families of
Maslov index 2 holomorphic discs in $\C^2$, whose boundaries also generate the
first homology group of the torus. 

\end{remark}
 
 \subsection{Floer Homology and non-displaceability}
   
 The modern way to show that a Lagrangian submanifold $L$ of a symplectic
 manifold $X$ is non-displaceable by Hamiltonian diffeomorphisms is to prove
 that its Floer homology $HF(L,L)$ is non-zero. The version of Floer Homology we
 use in this section to prove that $HF(T^c_{r,0}) \neq 0$ (for some choice of
 local system) is the Pearl Homology, introduced by Oh in \cite{Oh}. Here we
 will follow the definitions and notation similar to the ones given in \cite{BC,
 BC2}.

  \begin{figure}[h!] 

\begin{center}
\scalebox{0.5}{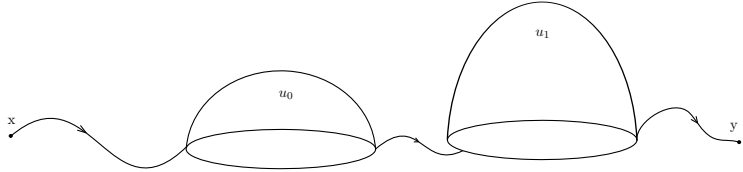}
 \end{center}
 \caption{A trajectory contributing to the differential of the pearl complex}
\label{figPearl}
\end{figure}
   
  Let $(X,\omega)$ be a symplectic manifold, $J$ a generic almost complex
  structure compatible with $\omega$, and $L$ a monotone Lagrangian submanifold,
  with monotonicity constant $M_L$. We also choose a $\C^*$ local system over L
  (we don't need to use the Novikov ring because $L$ is monotone, so the area
  of holomorphic discs is proportional to the Maslov index) and a spin structure
  to orient the appropriate moduli spaces of holomorphic discs.
  To define the pearl complex we fix a Morse function $f: L \to \R$ and a metric
  $\rho$ so that $(f, \rho)$ is Morse-Smale and we denote the gradient flow by
  $\gamma$. 
 
  The pearl
  complex $\mathcal{C}(L;f,\rho,J) = (\C[q,q^{-1}]\langle Crit(f) \rangle, d)$ is
  generated by the critical points of $f$, and the differential counts 
  configurations consisting of gradient flow lines of $\gamma$
  together with $J$-holomorphic discs as illustrated in Figure \ref{figPearl}.
  More precisely, for $|x| = ind_f(x)$, $|y| = ind_f(y)$, $B \in \pi_2(M,L)$,
  $\delta(x,y,B) = |x| - |y| - 1 + \mu(B)$:
  
\begin{equation}
    dx = \sum_{|y| = |x| - 1} \#
  \mathcal{P}(x,y,0)\cdot y \ + \sum_{\underset{B \neq 0}{\delta(x,y,B) = 0,}} (-1)^{|y|} \#
  \mathcal{P}(x,y,B)\text{hol}_{\nabla}(\del B)q^{\mu(B)}\cdot y
\end{equation}

  where $\nabla$ is the chosen
  local system and $\mathcal{P}(x,y,B)$ is the the moduli space of
  ``pearly trajectories", whose elements are gradient flow lines of $f$ from $x$
  to $y$ when $B = 0$, and otherwise tuples $(u_0,t_1,u_1,\cdots, t_k,
  u_k)$, $k \in \Z_{\geq 0}$ so that:
   
   \begin{itemize}
     \item[i.] For $0 \leq j \leq k$, $u_j$ is a non-constant $J$-holomorphic disc with boundary on 
     $L$, up to reparametrization by an automorphism of the disc fixing $\pm 1$.
     \item[ii.] $\sum_j  [u_j] = B$.
     \item[iii.] For $1 \leq j \leq k$, $t_j \in (0, +\infty)$, and
     $\gamma_{t_j}(u_{j-1}(1)) = u_j(-1)$.
     \item[iv.] $\gamma_{-\infty}(u_0(-1)) = x$, $\gamma_{+\infty}(u_k(1)) = y$. 
   \end{itemize}
 
The choice of spin structure on $L$ gives an 
 orientation for the moduli space of holomorphic discs, and together with the 
 orientation of the ascending and descending manifolds of each critical point of 
 $f$, one can get a coherent orientation for $\mathcal{P}(x,y,B)$. For a detailed
 account of how to orient the space of pearly trajectories, see appendix A.2.1 of \cite{BC2}. 
 
 We have a filtration given by the index of the critical point. For simplicity
 we write $\mathcal{C}_*(L)$ for $\mathcal{C}_*(L;f,\rho,J)$. Note that 
 $d = \sum_{j \geq 0} \delta_{2j}$, where $\delta_0: \mathcal{C}_*(L) \to 
 \mathcal{C}_{* -1}(L)$ is the Morse differential of the function $f$ and
  $\delta_{2j}: \mathcal{C}_*(L) \to \mathcal{C}_{* -1 + 2j}(L)$ considers 
  only configurations for which the total Maslov index is $2j$. This gives
  a spectral sequence (the Oh spectral sequence), converging to the Pearl homology, for which the 
  second page is the singular homology of $L$ with coefficients in $\C[q,q^{-1}]$.

 Let $\Lag$ be the space of $\C^*$ local systems in $L$ and consider
 the `superpotential' function $W: \Lag \to \C[q,q^{-1}]$, 
     
    \begin{equation} \label{m0_2}
   W(\nabla)  =  \sum_{\beta, \mu(\beta) = 2} n_{\beta} q^2
\text{hol}_{\nabla}(\del \beta) = \sum_{\beta, \mu(\beta) = 2} n_{\beta} z_\beta(\nabla) 
\end{equation}

where $z_\beta(\nabla) = q^{\frac{\int_{\beta} \omega}{M_L}} 
\text{hol}_{\nabla}(\del \beta) = q^2\text{hol}_{\nabla}(\del \beta)$ and
$n_\beta$ is the count of Maslov index two $J$-holomorphic discs bounded by
$L$ in the class $\beta$. 

Assume also that the inclusion map $H_1(L) \rightarrow H_1(X)$ is trivial, so we have that
the ring of regular functions on the algebraic torus $\Lag \cong \hom (H_1(L),\C^*)$ is generated by
the coordinates $z_j = z_{\beta_j}(\nabla)$ for relative classes $\beta_j$ such that 
$\del \beta_j$ generates $H_1(L)$.

The following result is the analogue of Proposition 11.1 of \cite{CHOOH} (see 
also section 12 of \cite{CHOOH}) in the pearly setting:

  \begin{proposition}
   Let $f$ be a perfect Morse function. Denote by $p$ the index 0 critical point,  
    by $q_1, \cdots, q_k$ the index 1 critical points, by $\Gamma_1, \cdots, 
    \Gamma_k$, the closure of the stable manifold of $q_1, \cdots, 
    q_k$, respectively,
    and by $\gamma_1, \cdots, \gamma_k$ the closure of the respective 
    unstable manifolds.
    Set $z_j = z_{\gamma_j}$. 
    Then 
    \begin{equation*}
     \delta_2 (p) = \sum_j \pm z_j\frac{\del W}{\del z_j} q_j 
    \end{equation*}
  \end{proposition}
  
  In particular $\delta_2(p) = 0$ precisely for the local systems corresponding to the critical points of $W$.
  
  \begin{proof} ({\it up to choice of orientations})
    We note that the only possible pearly trajectories contributing to the coefficient of 
    $q_j$ in $\delta_2(p)$ consist of a holomorphic discs $u$ with
    $u(-1) = p$ together with a flow line from $u(1)$ ending in $q_j$, i.e., $u(1) \in \Gamma_j$.
  Hence, 
    \begin{equation}\label{delta2}
       \delta_2(p) = \sum_{\beta, \mu(\beta) = 2} \pm n_{\beta} z_\beta(\nabla)
     \sum_j ([\del \beta]\cdot [\Gamma_j])  q_j 
    \end{equation}
   
     Since $[\gamma_1], \cdots, [\gamma_k]$ form a basis for $H_1(L)$, we can write
     $[\del \beta] = \sum_j a_j [\gamma_j]$, where $a_j = [\del \beta]\cdot 
     [\Gamma_j]$. So $z_\beta(\nabla)$ is a constant multiple of $\prod_j 
     z_j^{a_j}$, therefore \eqref{delta2} gives precisely  $\delta_2(p) = \sum_j \pm z_j\frac{\del W}{\del z_j} q_j$.
  \end{proof}
  
 \begin{corollary}
 Consider the monotone $T(1,4,25)$ torus $T^c_{r,0}$, endowed with the standard spin structure 
   and local system $\nabla$ such that $\text{hol}_{\nabla}(\del \beta) = \frac{9}{4} 
   e^{k\frac{2\pi}{3}i}$, for some $k \in \Z$, and $\text{hol}_{\nabla}(\del \alpha) 
   = \frac{1}{8}$, where $\alpha$ and $\beta$ are as defined in section 
   \ref{HomClass}. Then the Floer homology $HF(T^c_{r,0}, \nabla)$ is non-zero. Therefore 
    $T^c_{r,0}$ is non-displaceable.
 \end{corollary} 
 
 \begin{proof}
   Since $T^c_{r,0}$ has dimension 2, all the boundary maps $\delta_{2j}$ are 
   zero for $j \geq 2$. Hence the pearl homology $HF(T^c_{r,0}, \nabla)$ is the homology of 
   $(H_*(T^c_{r,0})\otimes \C[q,q^{-1}], \delta_2)$. Writing $u = z_\beta$ and $w = z_\alpha$, the `superpotential' 
 is given by
   \begin{equation*}
    W_{T(1,4,25)} = u + 2\frac{q^{6}}{u^2}(1 + w)^2 + 
 \frac{q^{12}}{u^5w}(1 + w)^5
  \end{equation*}
  
  The result follows from computing the critical points of $W_{T(1,4,25)}$
 which are $w = \frac{1}{8}$, $u = \frac{9}{4} e^{k\frac{2\pi}{3}i}q^2$.
   \end{proof}
  
  \begin{remark}
   It can be shown that in fact for any monotone Lagrangian torus $\delta_2 = 0$ for the local systems
$\nabla$ which are critical points of $W$, so $HF(T^c_{r,0}, \nabla) \cong 
    H_*(T^c_{r,0})\otimes \C[q,q^{-1}]$.
  \end{remark}
  
  \section{Prediction for $\CP^1\times\CP^1$} \label{TorCP1xCP1}

In this section we apply the same techniques of sections \ref{Pred}, \ref{ExoTor} 
to predict the existence of an exotic monotone torus in $\CP^1 \times \CP^1$
 bounding 9 families of Maslov index 2 holomorphic discs, hence
 not symplectomorphic to the Clifford or Chekanov ones. 

We use coordinates $((x : w), (y : z))$ on $\CP^1 \times \CP^1$. We smooth two
corners of the divisor $\{xwyz = 0\}$. First considering the fibration given
by $f_0((x : w), (y : z)) = \frac{xy}{wz}$ and $D_2 =\{\xi = 0\}$, for $\xi =
\frac{xy - wz}{2t}$ and $t$ a positive real number, we smooth the corner of $\{xy
= 0\}$ to get a new divisor $D_2 \cup \{wz=0\}$. Then using $f = \frac{z \xi}{w
y^2}$, we smooth the corner of $D_2 \cup \{z=0\}$ and get a new 
anticanonical divisor $D = f^{-1}(c) \cup \{w = 0\}$, for $c$ a positive real
number. Then we consider a similar singular Lagrangian torus fibration on the
complement of $D$. In particular we define the $T(1,2,9)$ (see remark \ref{rmk}) type torus:

\begin{definition}
Given $c> r >0$ and $\lambda \in \mathbb{R}$, 
\begin{equation}
  T^c_{r,\lambda}  = \left\{ ((x : w), (y : z)) ; \left|\frac{z\xi}{wy^2} - c \right| = r ;
 \left|\frac{z}{y}\right|^2 - \left|\frac{\xi}{wy}\right|^2 = \lambda \right\},  
\end{equation}
\end{definition}
which is Lagrangian for a symplectic form similar to \eqref{sympform}. 

\begin{remark} \label{rmk} The role of the parameter $t$ in the definition of
$\xi$ is less obvious than in the case of $\CP^2$, since here it amounts to a
single rescaling. However as in the case of $\CP^2$, its presence is motivated
by considerations about degenerations. More precisely the choice of $\xi =
\frac{xy - wz}{2t}$ is based on a degeneration of $\CP^1 \times \CP^1$ to
$\CP(1,1,2)$, which can be embedded inside $\CP^3$; see Proposition 3.1 of
\cite{DA09}. Applying two nodal trades to the standard polytope of $\CP^1 \times
\CP^1$ and redrawing the almost toric diagram as in Figure \ref{Amt}, we see that
the Chekanov torus corresponds to the central fiber of $\CP(1,1,2)$, and the 
torus  fiber obtained by lengthening both cuts to pass through the central fiber of 
$\CP^1 \times \CP^1$ corresponds to the central fiber of $\CP(1,2,9)$, 
therefore we denote it by $T(1,2,9)$.
\end{remark}

We then proceed as in section \ref{Pred} to predict the number of families of
Maslov index 2 holomorphic discs this torus should bound, at least for some
values of $t$, $c$ and $r$.

It is known that the Clifford torus bounds four families of Maslov index 2
holomorphic discs in the classes $\beta_1$, $\beta_2$, $H_1 -\beta_1$, $H_2
-\beta_2$, where $\beta_1= [\D\times\{1\}]$ and $\beta_2 = [\{1\}\times\D]$ seen
in the coordinate chart $y= 1$, $w=1$ and $H_{1} = [\CP^1]\times\{pt\}$ and $H_2
= \{pt\}\times [\CP^1]$. On the almost toric fibration illustrated on Figure
\ref{figCP1xCP1}, it is located in the top chamber and has superpotential given
by 

\begin{equation}
 W_{Clif} = z_1 + z_2 + \frac{e^{-A}}{z_1} +
\frac{e^{-B}}{z_2}, 
\end{equation}

where $z_1$ , $z_2$ are the coordinates associated with $\beta_1$, $\beta_2$, 
$A = \int_{[\CP^1]\times\{pt\}} \omega$, $B = \int_{\{pt\}\times [\CP^1]} 
\omega$. (For a monotone symplectic form $A = B$.)

The first wall-crossing towards the Chekanov type tori gives rise to the change of coordinates 
$z_1 = v_1(1 + \w)$, $z_2 = v_2(1 + \w)^{-1}$, where $\w = e^{-A}/z_1z_2 = e^{-A}/v_1v_2$.
Hence the superpotential becomes

\begin{equation}
  W_{Che} = v_2 + v_1(1 + \w) + e^{-B}\frac{(1 + \w)}{v_2} = v_1 + v_2 + \frac{e^{-A}}{v_2} 
  + \frac{e^{-B}}{v_2} + \frac{e^{-A -B}}{v_1v_2^2}.
\end{equation}

\begin{figure}[h!]
\begin{center}
\scalebox{0.8}{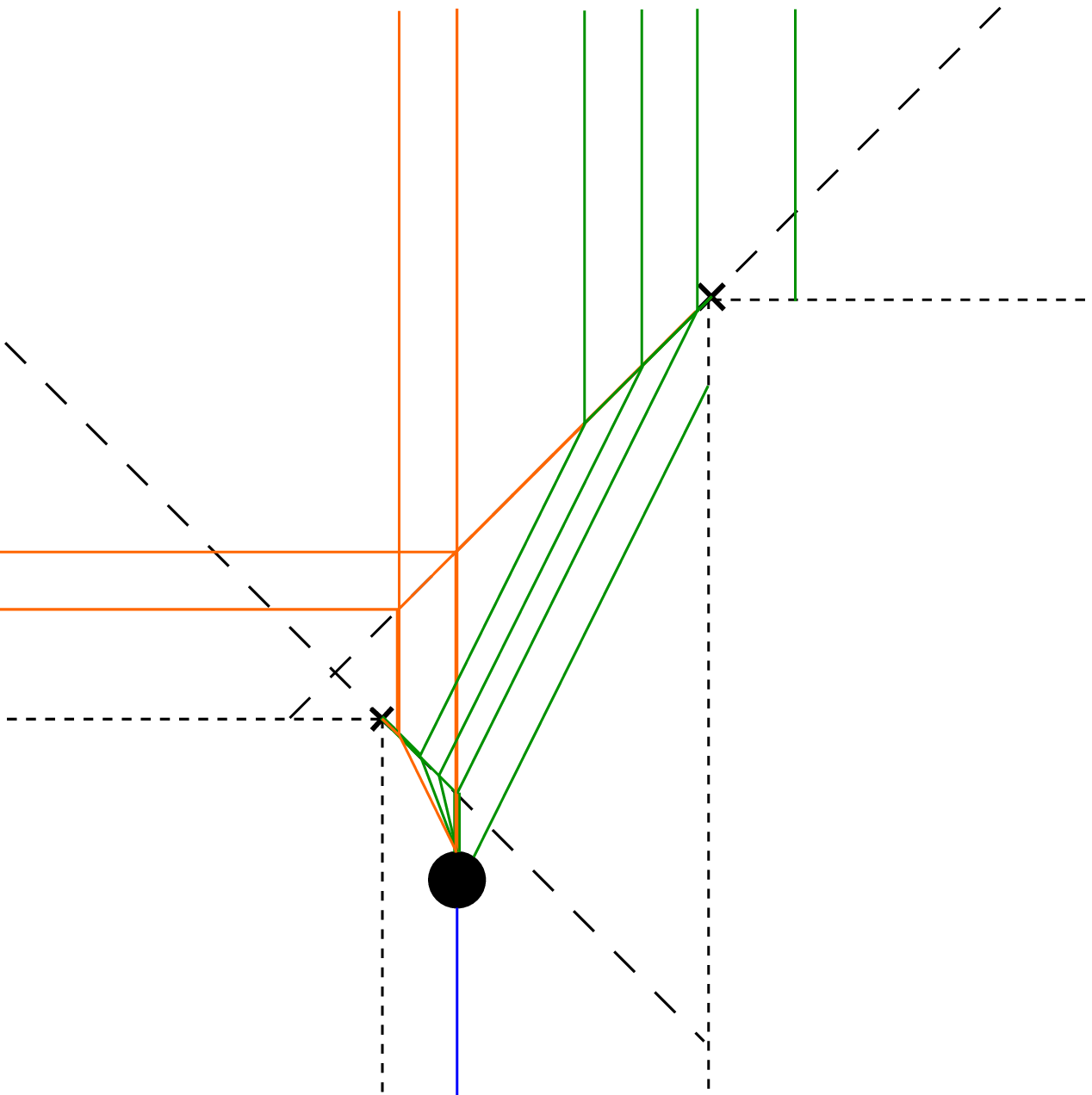}
 \end{center}
 \caption{A $T(1,2,9)$ type torus in $\CP^1 \times \CP^1$
 bounding 9 families of Maslov index 2 holomorphic discs. 
 The superpotential is given by $ W_{T(1,2,9)} = u_2 +  \frac{e^{-A}}{u_2} + \frac{e^{-B}}{u_2} +
  \frac{e^{-A}u_1}{u_2^2} + \frac{e^{-B}u_1}{u_2^2}
 + \frac{e^{-A - B}}{u_1u_2^2} + 3\frac{e^{-A - B}}{u_2^3}  + 3\frac{e^{-A - B}u_1}{u_2^4} + 
 \frac{e^{-A - B}u_1^2}{u_2^5}$.}
\label{figCP1xCP1}
\end{figure}

Crossing now the other wall towards the $T(1,2,9)$ type tori, we get the change 
of coordinates $u_1 = v_1(1 + w)$, $u_2 = v_2(1 + w)$, $w = v_1/v_2 = 
u_1/u_2$.
\newpage
The superpotential is then given by

\begin{eqnarray}
  W_{T(1,2,9)} & = & u_2 + (e^{-A} + e^{-B})\frac{1 + w}{u_2} + e^{-A - B}\frac{(1 + w)^3}{u_1u_2^2} \nonumber \\
  & = & u_2  + \frac{e^{-A}}{u_2} + \frac{e^{-B}}{u_2} +
   \frac{e^{-A}u_1}{u_2^2} + \frac{e^{-B}u_1}{u_2^2} \nonumber \\ 
 & & +  \frac{e^{-A - B}}{u_1u_2^2} + 3\frac{e^{-A - B}}{u_2^3}  + 3\frac{e^{-A - B}u_1}{u_2^4} + 
 \frac{e^{-A - B}u_1^2}{u_2^5}. 
\end{eqnarray}

 \subsection{The homology classes} \label{HomClassP1xP1}
 
 We consider the torus $T^c_{r,0}$. Using notation similar to that in section \ref{CHD}, let's call $\beta$
 the class of the Maslov index 2 holomorphic disc lying on the conic $z = \xi$
 that projects into the region $|f - c| \leq r$, and $\alpha$ the Lefschetz 
 thimble  associated to the critical point of $f$ at the origin lying above the 
 segment $[0, c - r]$ (oriented to intersect positively with $\{z = 0\}$). 
 
 As before we use positivity of intersection with some complex curves to restrict 
 the homology classes. 
 
 \begin{lemma}\label{HomTabP1xP1} 
For fixed $c$ and $r$, for $t$ sufficiently small, the intersection number of the
 classes $\alpha$, $\beta$, $H_1$ and $H_2$ with the varieties $\{x =0\}$,
 $\{y =0\}$, $\{w = 0\}$, $\{z =0\}$, $D_3 = f^{-1}(c) \cup \{((0:1),(1:0)),
 ((1:0),(1:0))\}$, $D_2 = \{\xi =0\}$, $D'_3 = \{x\xi - cw^2y = 0\}$, $D_6 = D_3
 \cup D'_3$ (all of them disjoint from $T^c_{r,0}$) and their Maslov indeces $\mu$, 
 are as giving in the table below:
 
 \begin{center}
 \begin{tabular}{ | c || c | c | c | c | c | c | c | c || c |}
  
\hline
    Class & $x = 0$ & $y = 0$ & $w = 0$ & $z = 0$ & $D_2$ & $D_3$ & $D'_3$ & $D_6$ & $\mu$  \\ 
   \hline 
    $\alpha$ & 1 & 0 & 0 & 1  & -1 & 0 & 0 & 0 &  0 \\ \hline
    $\beta$   & 0 & 0 & 0 & 0  &  0 & 1  & 1 & 2 & 2 \\   \hline
      $H_1$   & 1 & 0 & 1 &  0  &  1 & 1 & 2 & 3 & 4 \\   \hline
      $H_2$     & 0 & 1 & 0 &  1  &  1 & 2 & 1 & 3 & 4 \\   \hline
  \end{tabular}
\end{center}
\end{lemma}
 
 \begin{proof}
  
The intersection numbers of $H_1$ and  $H_2$ with the given complex curves are computed 
using Bezout's theorem, and the Maslov index is twice the intersection number with the anticanonical
divisor $D = D_3 \cup \{w = 0\}$.

  By construction, the intersection of $\alpha$ with $\{z =0\}$ is one, with
  $D_2$ is negative one and with $D_3$, $\{y =0\}$ and $\{w =0\}$ it is clearly zero,
  as well as the intersection of $\beta$ and $T^c_{r,0}$ with $\{ywz =0\}$ and $D_2$.
  Also clear is the intersection of $\beta$ with $D_3$.
  
  To understand the intersection of the torus $T^c_{r,0}$, $\alpha$ and $\beta$
  with $\{x =0\}$ and $D'_3$, we look at the family of conics $\mathcal{C} = \{z
  = e^{i\theta}\xi; \theta \in [0, 2\pi]\}$ containing $T^c_{r,0}$, the thimble representing
  the class $\alpha$ and holomorphic discs representing the class $\beta$,
  similar to the ones in Proposition \eqref{eqD5} (for instance one where $z =
  \xi$, and $Re(z) > 0$). We use the coordinate chart $y =1$, $w =1$.
  
  For $\{x =0\} \cap \mathcal{C}$, we have $z = e^{i \theta}\xi = -e^{i\theta} 
 z/2t $, so the intersection is only at $z =0$, for $t$ small enough, and 
  with same sign as for $\{z =0\}$, since $x$ is a multiple of $z$ along 
  $\mathcal{C}$.
  
  For $D'_3 \cap \mathcal{C}$, we note that $x = z(2te^{-i\theta} + 1)$ along   
  $\mathcal{C}$. Considering $f = z\xi$, we have
  
  \begin{equation*}
   0 = x\xi -c = z(2te^{-i\theta} + 1)\xi - c = f(2te^{-i\theta} + 1) - c
  \end{equation*}
  
   So, $f = \frac{c}{(2te^{-i\theta} + 1)}$, so we see that for $t$ very small,      
 $D'_3 \cap \mathcal{C}$ intersects in a circle projecting via $f$ inside the 
 region $|f - c| < r$, therefore $D'_3$ intersects $T^c_{r,0}$, $\alpha$ and 
  $\beta$ respectively at $0$, $0$, and $1$ point (counting positively as $D'_3$ 
  and our representative of the $\beta$ class are complex curves). 
 \end{proof}
 \begin{remark}
   $D_6$ was found by considering the degeneration of $\CP^1\times\CP^1$ to
   $\CP(1,1,2)$ in a similar manner as in section \ref{HomClass} $D_5$ was
   found using the degeneration of $\CP^2$ to $\CP(1,1,4)$. Here it turns out
   that $D_6 = D_3 \cup D'_3$. 
   \end{remark}
   
   \begin{lemma} \label{lemHomClassP1xP1}
The only classes in $\pi_2(\CP^1\times\CP^1, T_{r,0})$ which may contain holomorphic discs of 
Maslov index 2 are $\beta$, $H_1 - \beta$, $H_2 - \beta$, $H_1 - \beta + \alpha$, 
 $H_2 - \beta + \alpha$ and $H_1 + H_2 -3\beta+ k\alpha$, $-1 \leq k \leq 2$.
\end{lemma}

\begin{proof}

Maslov index $2$ classes must be of the form $\beta + k\alpha + m(H_1 -
2\beta) + n(H_2 - 2\beta)$. Considering positivity of intersections with complex 
curves the proof follows from the inequalities for $k$, $m$ and $n$ given by the 
table:
 \begin{center}
 \begin{tabular}{ | c || c | c | c | c | c | c | c |}
  
\hline
Curve & $x = 0$ & $y = 0$ & $w = 0$ & $z = 0$ & $D_2$ & $D_3$ & $D'_3$  \\ 
   \hline 
Inequality & $-m \leq k$ & $0 \leq n$ & $0 \leq m$ & $-n\leq k$  & $k \leq m + n$ & $m \leq 1$ & $n \leq 1$ \\ \hline
  \end{tabular}
\end{center}

\end{proof}
 
 \subsection{The monotone torus}
 
 In order to make $T^c_{r,0}$ a monotone Lagrangian torus, we deform our
 symplectic form using K\"{a}hler inflation in neighborhoods of complex curves
 that don't intersect $T^c_{r,0}$,in a similar way as we did in section
 \ref{MonTor}. First one can inflate along $\{y = 0\}$ or $\{w = 0\}$ to get a
 monotone K\"{a}hler form $\tilde{\omega}$ for $\CP^1 \times \CP^1$, i.e.,
 $\int_{H_1}\tilde{\omega} = \int_{H_2}\tilde{\omega}$, for which
 $\int_{\alpha}\tilde{\omega} = 0$ . In order for $T^c_{r,0}$ to be monotone, we
 need a K\"{a}hler form $\hat{\omega}$, satisfying the same conditions as
 $\tilde{\omega}$ plus $\int_{H_2}\hat{\omega} = 2\int_{\beta}\hat{\omega}$.
 Noting that the intersection numbers of $D_6$ with $\alpha$, $\beta$, $H_1$,
 $H_2$ are 0 ,2, 3 and 3, respectively, we can get $\hat{\omega}$ by adding a
 specific multiple of a 2-form supported on a neighborhood of $D_6$ to
 $\tilde{\omega}$ as in Proposition \ref{monform}, so as to satisfy $\int_{\alpha}\hat{\omega} = 0$
 and $\int_{H_1}\hat{\omega} = \int_{H_2}\hat{\omega} = 2\int_{\beta}\hat{\omega}$. 
 The last equality can be achieved because the ratio between the intersection numbers 
 $[D_6]\cdot H_1 = [D_6]\cdot H_2$ and $[D_6]\cdot \beta$  is 2/3 which is greater 
 than 1/2. 
 
 Therefore one only need to compute the expected Maslov index 2 holomorphic discs
 in the classes $\beta$, $H_1 - \beta$, $H_2 - \beta$, $H_1 - \beta + \alpha$, 
 $H_2 - \beta + \alpha$ and $H_1 + H_2 -3\beta+ k\alpha$, $-1 \leq k \leq 2$ to 
 prove:
\begin{conjecture}
  There is a monotone $T(1,2,9)$ torus, of the form $T^c_{r,0}$, in $\CP^1\times\CP^1$, 
  bounding 9 families of Maslov index 2 holomorphic discs, that is not symplectomorphic to the 
  monotone Chekanov torus nor to the monotone Clifford torus.
\end{conjecture}

\end{document}